 \documentclass[12pt]{article}

\usepackage{mathtools}
\mathtoolsset{showonlyrefs}
\usepackage{eucal}
\usepackage{color}

\usepackage{amsmath,bbm}
\usepackage{amssymb}
\usepackage{epic}
\usepackage{enumerate}
\usepackage{graphicx}
\usepackage{pict2e}

\usepackage{pdfsync}

\title{Tacnode GUE-minor Processes and Double Aztec Diamonds}

\author{Mark Adler\thanks{2000
{\em Mathematics Subject Classification}. Primary:
60G60, 60G65, 35Q53; secondary: 60G10, 35Q58. {\em Key
words and Phrases}: interlacing, random tiling,Kasteleyn, dimer, Airy
process, extended kernels, random Hermitian ensembles. \newline
 Department of Mathematics, Brandeis University,
Waltham, Mass 02454, USA. E-mail: adler@brandeis.edu.
The support of a National Science Foundation grant \#
DMS-01-00782 is gratefully acknowledged.}~~~Sunil Chhita\thanks{Department of Mathematics,
Royal Institute of Technology (KTH), Stockholm, Sweden. E-mail: chhita@kth.se. 
The support of the Knut and Alice Wallenberg Foundation grant KAW 2010.0063 is gratefully acknowledged.}~~~Kurt Johansson\thanks{Department of Mathematics,
KTH Royal Institute of Technology, Stockholm, Sweden. E-mail: kurtj@kth.se. The support of the Swedish Research Council (VR) and grant KAW 2010.0063 of the Knut and Alice Wallenberg Foundation are gratefully acknowledged. } ~~~ Pierre
van Moerbeke\thanks{ Department of Mathematics,
Universit\'e de Louvain, 1348 Louvain-la-Neuve, Belgium
and Brandeis University, Waltham, MA 02454, USA. E-mail: pierre.vanmoerbeke@UCLouvain.be and
vanmoerbeke@brandeis.edu. The support of a National Science
Foundation grant \# DMS-07-00782, FNRS, PAI grants is
gratefully acknowledged.}}

\date{}

\newcommand{\MAT}[1]{\left(\begin{array}{*#1c}}
\newcommand{\mat}{\end{array}\right)}
\newcommand{\qed}{\leavevmode\unskip\nobreak\penalty200\hskip2pt\null
\nobreak\hfill\rule{1.1ex}{1.1ex}%\parfillskip=0pt
\medbreak }

\newcommand{\I}{{\rm i}}

\newcommand{\BH}{{\mathbb H}}
\newcommand{\BL}{{\mathbb L}}
\newcommand{\BP}{{\mathbb P}}

\newcommand{\BZ}{{\mathbb Z}}

\newcommand{\al}{\alpha}

\newcommand{\Id}{\mathbbm{1}}

\newenvironment{proof}{\medskip\noindent{\it Proof:\/} }{\qed}
\newenvironment
        {remark}{\medskip\noindent\underline{\it Remark:\/} }{\medbreak}

\newcommand{\om}{\omega}

\newcommand{\la}{\langle}
\newcommand{\ra}{\rangle}

\newcommand{\Dt}{\Delta}

\newcommand{\BR}{{\mathbb R}}
\newcommand{\lb}{\lambda}

\newcommand{\dis}{\displaystyle}

\newcommand{\BK}{{\mathbb K}}

\def\be#1\ee{\begin{equation}#1\end{equation}}
\def\bea#1\eea{\begin{eqnarray}#1\end{eqnarray}}
\def\bean#1\eean{\begin{eqnarray*}#1\end{eqnarray*}}

 \newtheorem{definition}{Definition}[section]
 \newtheorem{theorem}[definition]{Theorem}
 \newtheorem{lemma}[definition]{Lemma}
 
 \newtheorem{proposition}[definition]{Proposition}

\catcode `!=11

\newdimen\squaresize
\newdimen\thickness
\newdimen\Thickness
\newdimen\ll! \newdimen \uu! \newdimen\dd! \newdimen \rr! \newdimen
\temp!

\def\sq!#1#2#3#4#5{%
\ll!=#1 \uu!=#2 \dd!=#3 \rr!=#4
\setbox0=\hbox{%
 \temp!=\squaresize\advance\temp! by .5\uu!
 \rlap{\kern -.5\ll!
 \vbox{\hrule height \temp! width#1 depth .5\dd!}}%
 \temp!=\squaresize\advance\temp! by -.5\uu!
 \rlap{\raise\temp!
 \vbox{\hrule height #2 width \squaresize}}%
 \rlap{\raise -.5\dd!
 \vbox{\hrule height #3 width \squaresize}}%
 \temp!=\squaresize\advance\temp! by .5\uu!
 \rlap{\kern \squaresize \kern-.5\rr!
 \vbox{\hrule height \temp! width#4 depth .5\dd!}}%
 \rlap{\kern .5\squaresize\raise .5\squaresize
 \vbox to 0pt{\vss\hbox to 0pt{\hss $#5$\hss}\vss}}%
}%end of \hbox
 \ht0=0pt \dp0=0pt \box0
}%end of \sq!

\def\vsq!#1#2#3#4#5\endvsq!{\vbox to \squaresize{\hrule
width\squaresize height 0pt%
\vss\sq!{#1}{#2}{#3}{#4}{#5}}}

\newdimen \LL! \newdimen \UU! \newdimen \DD! \newdimen \RR!

\def\vvsq!{\futurelet\next\vvvsq!}
\def\vvvsq!{\relax
  \ifx     \next l\LL!=\Thickness \let\continue=\skipnexttoken!
  \else\ifx\next u\UU!=\Thickness \let\continue=\skipnexttoken!
  \else\ifx\next d\DD!=\Thickness \let\continue=\skipnexttoken!
  \else\ifx\next r\RR!=\Thickness \let\continue=\skipnexttoken!
  \else\def\continue{\vsq!\LL!\UU!\DD!\RR!}%
  \fi\fi\fi\fi
  \continue}

\def\skipnexttoken!#1{\vvsq!}

\def\place#1#2#3{\vbox to 0pt{\vss
\rlap{\kern#1\squaresize
  \raise#2\squaresize\hbox{$#3$}}
\vss}}

\catcode `!=12

\makeatletter

\newcommand{\Rmnum}[1]{\expandafter\@slowromancap\romannumeral #1@}
\makeatother

\begin{document}

\sloppy
\maketitle

%\newpage

\begin{abstract}

We study random domino tilings of a Double Aztec diamond, a
region consisting of two overlapping Aztec diamonds. The random tilings give rise
to two discrete determinantal point processes called the $\mathbb{K}$- and $\mathbb{L}$-particle processes.
The correlation kernel of the $\mathbb{K}$-particles was derived in Adler, Johansson and van Moerbeke
(2011), who used it to study the limit process of the $\mathbb{K}$-particles with different weights
for horizontal and vertical dominos. Let the size of both, the Double Aztec diamond and the overlap, tend to infinity such that the two arctic ellipses just touch; then they show that the fluctuations of the $\mathbb{K}$-particles near the tangency point tend to the tacnode process. 
In this paper, we find the limiting
point process of the $\mathbb{L}$-particles in the overlap when the weights of the horizontal and vertical dominos are equal,
or asymptotically equal, as the Double Aztec diamond grows, while keeping the overlap finite.
In this case the two limiting arctic circles are tangent in the overlap and 
the behavior of the $\mathbb{L}$-particles in the vicinity of the point of tangency can then be viewed as two colliding GUE-minor process, which we call the tacnode
GUE minor process. As part of the derivation of the kernel for the $\mathbb{L}$-particles
we find the inverse Kasteleyn matrix for the dimer model version of Double Aztec diamond.

\end{abstract}

\tableofcontents

%\newpage

\section{Introduction and main results}

\begin{figure}
\begin{center}
\includegraphics[height=4in]{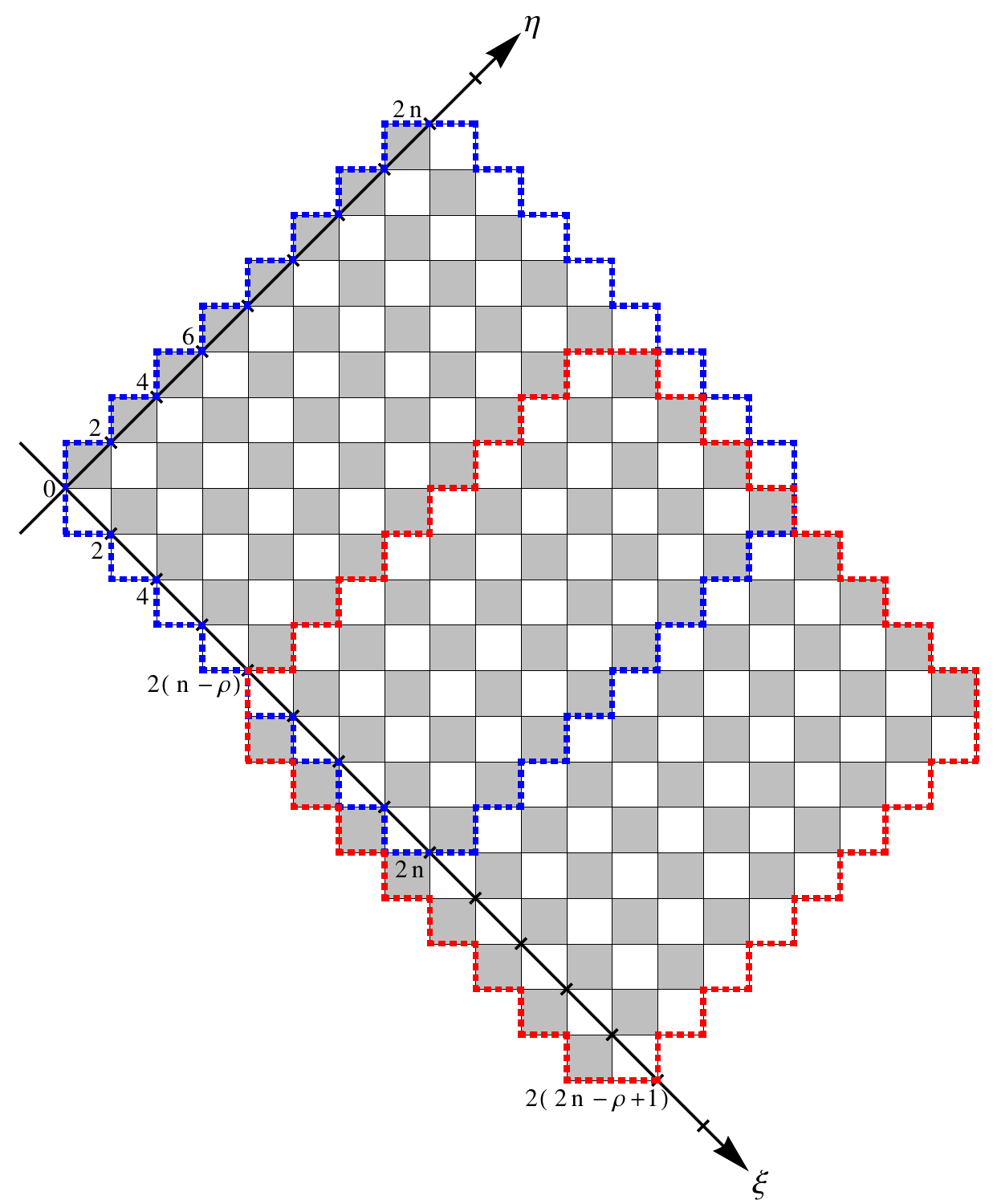}
\caption{Double Aztec diamond with $n=8$ and overlap $\rho=4$, with the $(\xi, \eta)$ coordinates with Aztec diamond $A$ enclosed by the blue dotted line and Aztec diamond $B$ enclosed by the red dotted line. The overlap contains $\rho$ lines (through black squares) $\xi=2s$ for $n-\rho<s\leq n$.
} 
\label{Fig:Kastcoords} 
\end{center}
\end{figure}

%{\bf SUNIL: referring to figure 1, it is a bit ambiguous which point $2(2n-\rho_1)$. Maybe it is better to put it to the right of the axis $\xi$.} 

The problem of a random domino tiling of a single Aztec diamond has been widely investigated by the combinatorics and probability community; the highlight was the existence of an inscribed arctic circle: inside the circle the dominos display a disordered pattern and outside a regular brick wall pattern; see \cite{EKLP,EKLP2,JPS,FS03,Johansson3,Jo02b}. When the weight $a$ of vertical dominos is different from the one of horizontal dominos, then the arctic circle gets replaced by an inscribed arctic ellipse.  %Johansson \cite{Johansson3} then showed that the domino tilings near the arctic circle fluctuate like the Airy process, upon observing the boundary with an appropriate magnifying glass; the process is run with a time which is tangential to the boundary. This was done by showing that domino tilings of Aztec diamonds can be translated into non-intersecting random walks and a point process, which turn out to be a determinantal process. 

In \cite{AJvM} the authors   
  investigate the domino tiling of two overlapping Aztec diamonds, each of size $n$, with weight $0<a<1$ for vertical dominos and weight $1$ for horizontal dominos. When the size of the diamonds and the overlap both become very large, in such a way that the two arctic ellipses for the single Aztec diamonds merely touch, then a new critical process, the {\em tacnode process}, will appear near the point of osculation (tacnode); it is run with a time in the direction of the common tangent to the ellipses. The kernel governing the local statistics of the tacnode process is given by a perturbation of the {\em Airy process kernel} by an integral of two functions. It was also shown in \cite{AJvM} that this tacnode process has some universal character: it coincides with the one found in the context of two groups of non-intersecting random walks \cite{AFvM12} and Brownian motions, meeting momentarily \cite{Joh10, FV}; see also \cite{Del12}.

  Another ingredient here is the process given by the successive interlacing eigenvalues of minors of a GUE-matrix: the so-called GUE-minor process. In \cite{JoNo} this process has arisen in the following model: magnifying the infinitesimal region about the point of tangency of the arctic ellipses with the edge of a single Aztec diamond for large $n$, leads to a determinantal process of interlacing points on the successive lines through (say) the black squares, parallel to the edge of the diamond. This yields the GUE-minor process, see also \cite{OR}.
 
  In the present work, we consider two overlapping Aztec diamonds with an overlap, which remains finite, when the size of the diamonds tends to $\infty$. In order to maintain the osculation of the two inscribed ellipses, the geometry forces the weight $a$ of the vertical ones to tend to the weight $1$ of the horizontal ones, say at a rate $\beta/ \sqrt{n/2}$. Macroscopically this amounts to two Aztec diamonds with inscribed arctic circles intersecting infinitesimally. In view of the comments above, it seems natural that this process be related to the GUE-minor kernel. Indeed, 
  when $n\to \infty$, looking with a magnifying glass at the infinitesimal overlap of the diamonds gives rise to a new determinantal process, with local statistics given by the so-called {\em tacnode GUE-minor kernel}; it is a finite rank perturbation of the GUE-minor kernel mentioned above. As far as we are aware, there is no Random Matrix theory counterpart of this distribution. 
  
  In \cite{AJvM}, the authors considered successive lines through black squares perpendicular to the region of overlap with dots in the black square each time the dominos covering that square is pointing to the right of or above the line; these are called the ${\mathbb K}$-particles. In this work, we shall mainly consider successive lines parallel to the region of overlap and put a dot in the black square each time the dominos covering that square is pointing to the left of or above that line; these give the ${\mathbb L}$-particles (Section 1.2). We deduce the kernel for the ${\mathbb L}$-particles from the one of the ${\mathbb K}$-particles, by first obtaining the inverse Kasteleyn matrix \cite{Kas:61} for the double Aztec diamond in terms of the ${\mathbb K}$-kernel and then deduce the ${\mathbb L}$-kernel of the particles from that same inverse Kasteleyn matrix (Section 2). The inverse Kasteleyn matrix of a single Aztec diamond had been obtained for $a=1$ by \cite{Hel:00} and generalized for all $a$  recently in \cite{CJY:12}.  

  In Section 1.2, we study the specific interlacing pattern of the ${\mathbb L}$-particles. We state in Section 1.3 and prove in Section 5 that, in the scaling limit $n\to \infty$, the ${\mathbb L}$-process in the infinitesimal overlap is indeed driven by the {\em tacnode GUE-minor kernel}. Also, we state in Section 1.3 and prove in Section 5 that upon thinning at the rate $p_n=1-2/\sqrt{n/2}$, the ${\mathbb K}$-process is also driven by the same tacnode GUE-minor kernel, but with a (somewhat surprising) shift in one of the discrete parameters.

  % \vspace*{-4cm}
 
 % {\includegraphics[width=167mm,height=206mm]{Fig1_DoubleAztec}}

 % \vspace{-6,5cm}

%Figure 1. Double Aztec diamond of size $n=7$ and overlap $\rho=3$, with the $(\xi,\eta)$ coordinates.

 %  \vspace{ .6cm}

% \noindent  {\bf 1.1. Domino tilings of double Aztec diamonds and random surface.}

\subsection{ Domino tilings of double Aztec diamonds and random surface}
 Consider two overlapping Aztec diamonds $A$ and $B$, of equal sizes $n$ and overlap $\rho$, with opposite orientations; i.e., the upper-left square for diamond $A$ is black and is white for diamond $B$. The size $n$ is the number of squares on the upper-left side and the amount of overlap $\rho$ counts the number of lines of black squares common to both diamonds $A$ and $B$.  Let $\xi,\eta$ be a system of coordinates as indicated in Figure~\ref{Fig:Kastcoords}. The even lines $\xi=2k$ for $0\leq k\leq 2n-\rho$  and the odd lines $\eta=2k-1$ for $1\leq k\leq n $ run through black squares. The $\rho$ even lines $\xi=2(n-\rho+1), \ldots, 2n$ belong to the overlap of the two diamonds. Cover this double diamond randomly by dominos, horizontal and vertical ones, as in Figure~\ref{Fig:Smallsims}. The position of a domino on the Aztec diamond corresponds to four different patterns, given in Figure~\ref{Fig:Smallsims} below: North, South,
 East, West.
 
 Define 
 $$
 M:=n-\rho+1= \#\left\{\begin{aligned}&\mbox{white squares of diamond $A$  along the line $\eta=0$,   }\\
 &\mbox{having an edge in common with the boundary} 
 \end{aligned} \right\},
 $$
 and define $m$ such that 
 $$
 M-1=n-\rho= \left\{\begin{aligned}& 2m,~~~~~\mbox{when $n$ and $\rho$ have  same parity}\\
 &2m-1,~~\mbox{when $n$ and $\rho$ have opposite parity.}
 \end{aligned}\right.
 $$
 Throughout the paper we assume, for simplicity, the same parity for $n$ and $\rho$, so that $n-\rho=2m$. 
 
 \begin{figure}
 \begin{center}
 \includegraphics[height=3.5in]{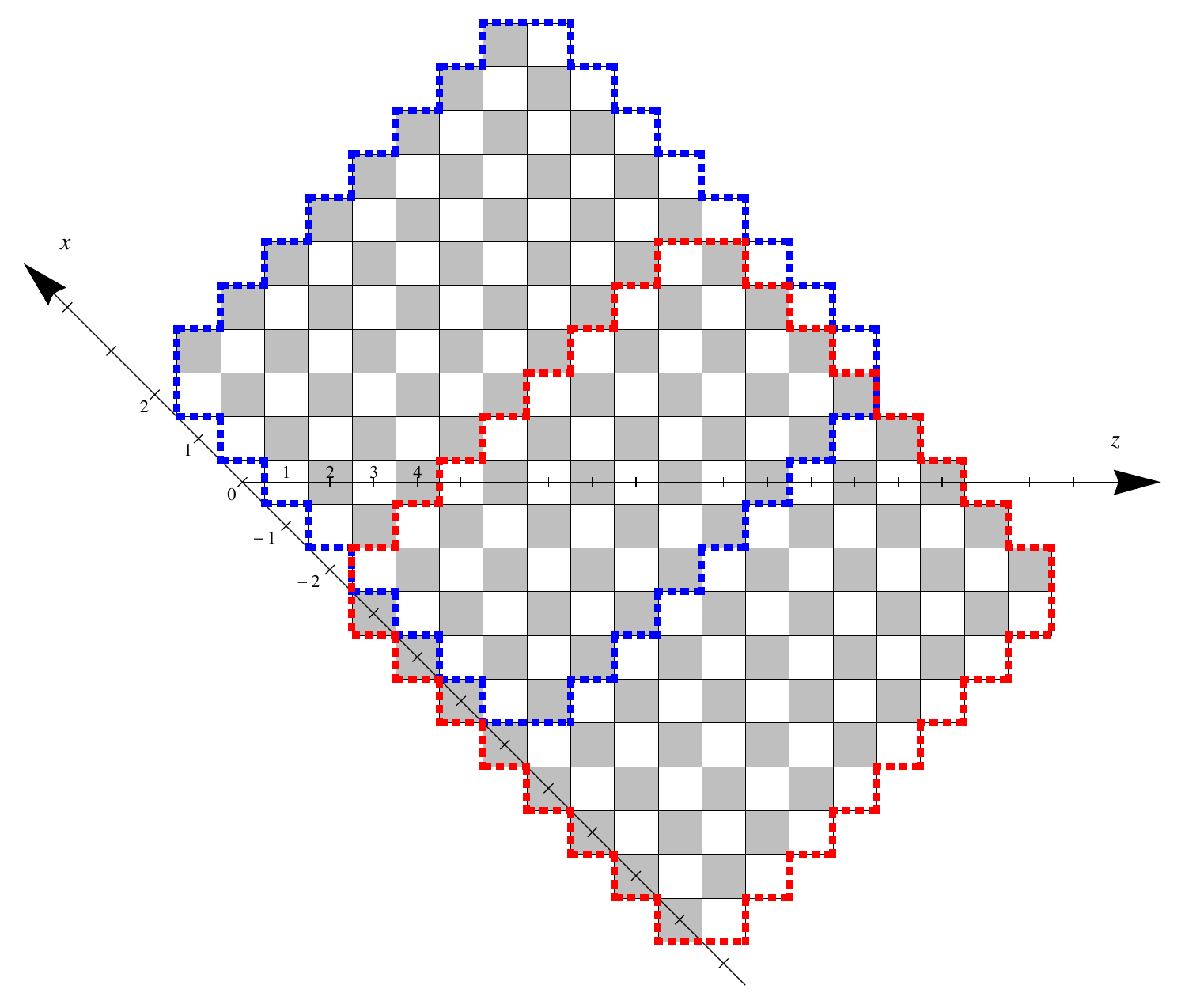}
  \caption{Double Aztec diamond with $n=8$ and overlap $\rho=4$ with the $(z,x)$ coordinates of the Aztec diamond.}
  \label{Fig:zxcoords}
  \end{center}
 \end{figure}
  Together with this arbitrary domino-tiling of the double Aztec diamond $A\cup B$, one defines a piecewise-linear random surface, by means of a height function $h$ specified by the heights, prescribed on the single dominos according to figure \ref{FigLevelline} above; this height can be taken to be piecewise-linear on each domino. This height function is different from the usual one by Cohn, Kenyon, Propp \cite{CKP}, but related to it by an affine transformation. Let the upper-most edge of the double diamond $A\cup B$ have height $h=0$. Then, regardless of the covering by dominos, the height function along the boundary of the double diamond will always 
be as indicated in Figure~\ref{FigLevelline}, with height $h=2n$ along the lower-most edge of the double diamond. Away from the boundary the height function will, of course, depend on the tiling; the associated heights are given in Figure~\ref{FigLevelline}.     

\begin{figure}
\begin{center}
\includegraphics[height=2.75in]{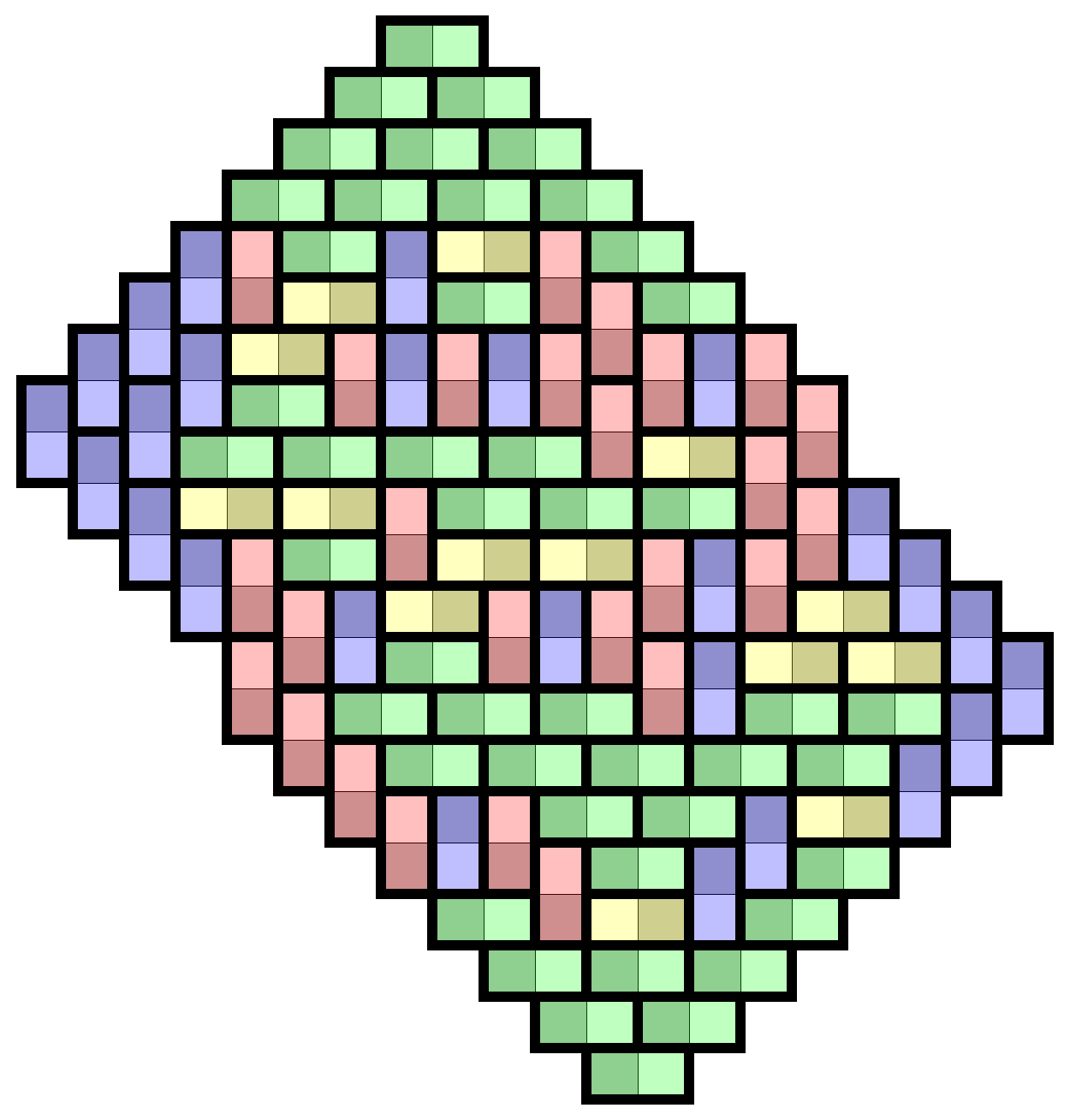}\hspace{2mm}
\includegraphics[height=2.75in]{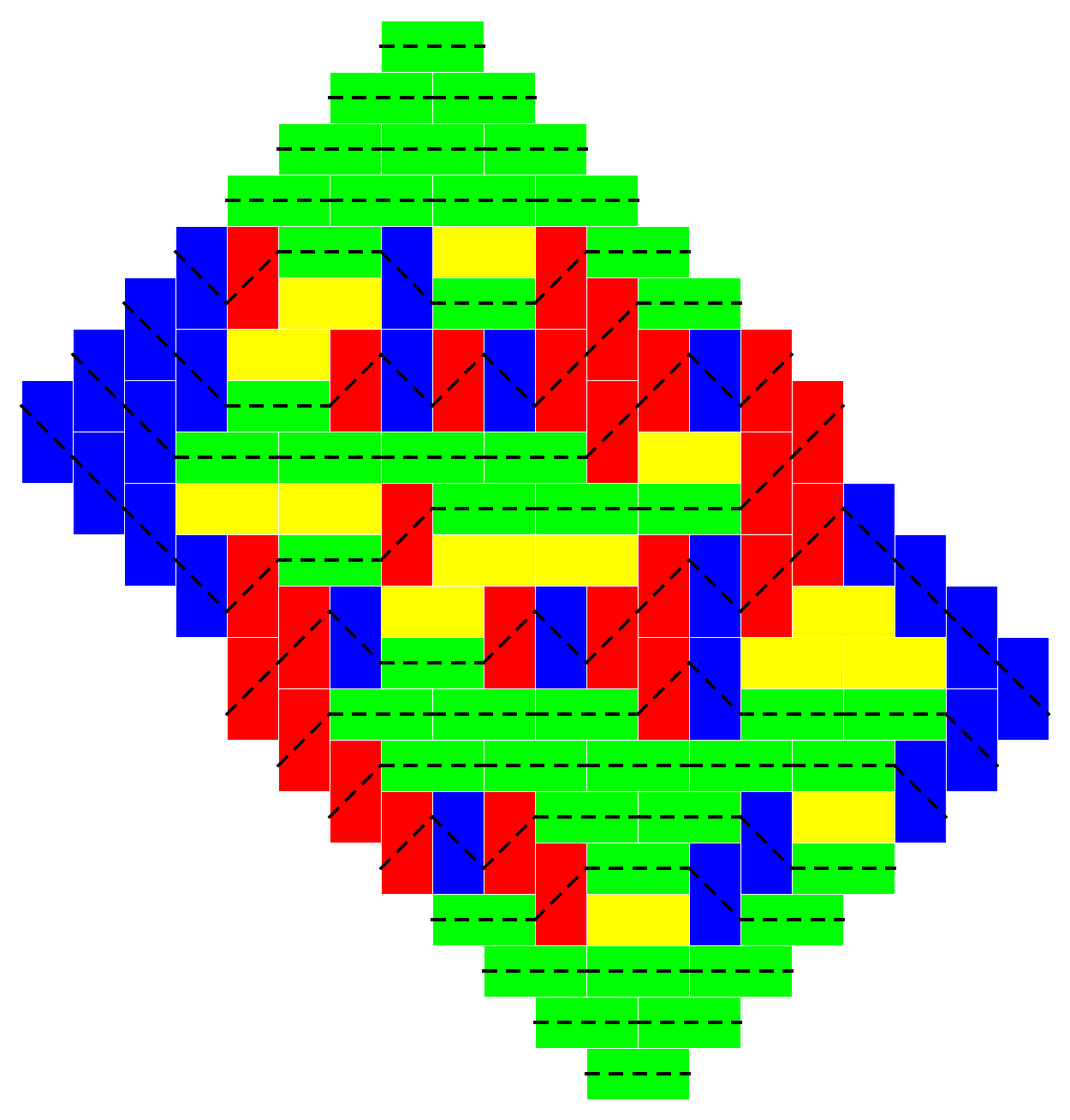}
\includegraphics[height=0.75in]{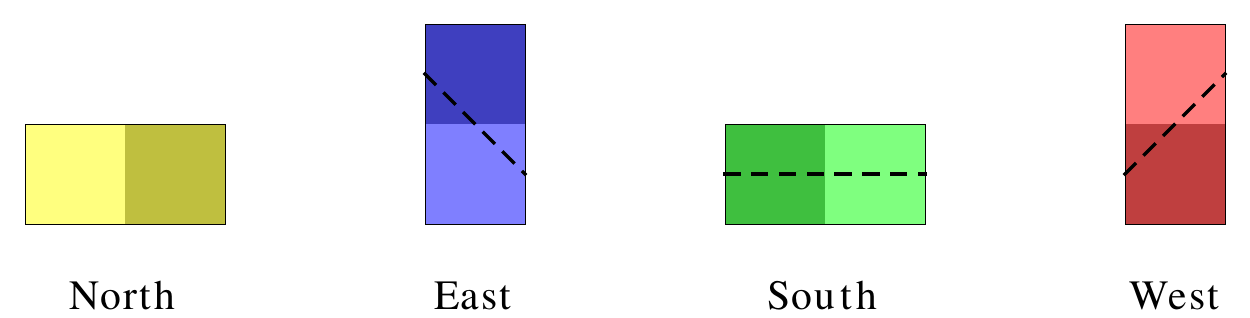}
\caption{Random Tiling of a double Aztec diamond with $n=8$ and $\rho=4$.  The figure on the left shows the underlying checkerboard structure while the 
figure on the right shows the same tiling with the level lines.  These are shown explicitly for the four types of dominos in the bottom figure. These level lines are the same as the DR lattice paths \cite{LRS:01}. 
}
\label{Fig:Smallsims} 
\end{center}
\end{figure}
 The height function $h$ obtained in this way defines the domino tiling in a {\em unique} way, because a white square together with its height specifies in a unique way to which domino it belongs to: North, South, East and West; the same holds for black squares. 
  %
  %For example, a white square with a constant height along the four edges can only be covered by a North domino. 
\begin{figure}
\begin{center}
\includegraphics[height=3in]{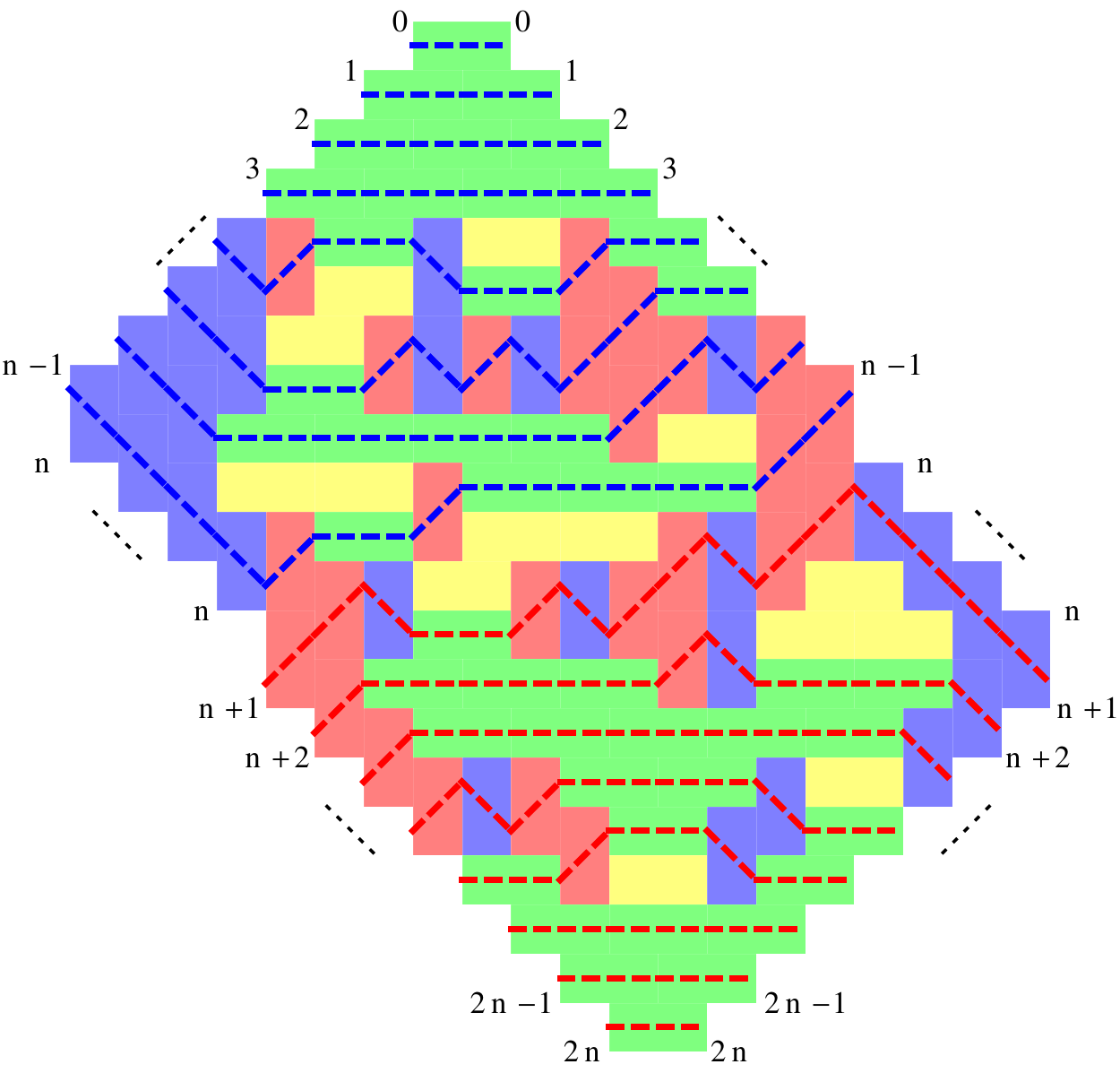}\vspace{10mm}
\includegraphics[height=0.75in]{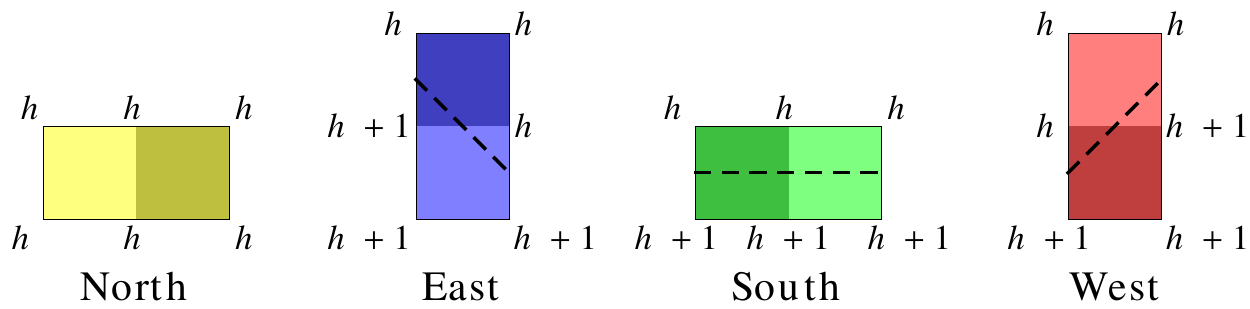}     
\caption{The level lines including the height function around the boundary.  The heights change in the interior only when crossing a level line. The bottom figure shows the height change for each individual domino. }
\label{FigLevelline}
\end{center}
\end{figure}

This height function associates thus a piece-wise linear random  surface with each random tiling and two groups of level curves of this random surface corresponding to the half-integer values:%   
      $$\underbrace{ \frac12 ,~\frac 32,~\frac 52,\ldots, n-\frac 12}_{\mbox{$A$-level curves}} ,~\underbrace{n+\frac 12,\ldots,~ 2n-\frac 12}_{\mbox{$B$-level curves}}.$$ 
      
      Put the weight $a>0$ on vertical dominoes and the weight $1$ on horizontal dominoes, so that the probability of a tiling configuration $T$ can be expressed as 
\be\BP(\mbox{domino tiling}~T)=\frac{\dis a^{\#\mbox{vertical domino's in $T$} }}{\dis\sum_{\mbox{\tiny all possible tilings $T$}}{a^{\#\mbox{vertical domino's in}~T}}}
\label{Ptiling}\ee
Remember $2m:=n-\rho$ throughout the paper. 
We will also use the coordinates indicated in figure \ref{Fig:zxcoords}. These are the coordinates, which we will call diamond coordinates, that were used for the
particle processes in \cite{AJvM}.
The transformation from diamond coordinates $(z,x )$ to Kasteleyn coordinates $(\xi ,\eta )$ is given by :
\be
\begin{aligned}
z&=\eta +1\\
x &=\tfrac 12 (\eta -\xi +2m+1)
\end{aligned}
\Longleftrightarrow
\begin{aligned}
\xi &= z -2x +2m  \\
\eta  &=z -1
\end{aligned}
\label{K}\ee%$$

%\vspace*{.5cm}
%\noindent{\bf 1.2. Two determinantal point processes ${\mathbb L}$ and ${\mathbb K}$}   
\subsection{Two determinantal point processes ${\mathbb L}$ and ${\mathbb K}$}\label{subs1.2}   

\medspace

 { \bf I. \em   The ${\mathbb L}%^{\mbox{\tiny \rm twoAzt}}
   $-process} is specified by putting a dot in the middle of the black square when the line $\xi=2s$ in $(\xi,\eta)$-coordinates for $0\leq s\leq 2n-\rho$ intersects a level curve. We call these dots {\em $\mathbb{L}$-particles}. See Figure~\ref{FigKLparticles} for an example.  More precisely we can put a blue dot when intersecting $A$-level curves and a red dot when intersecting $B$-level curves to distinguish the dots coming from the two Aztec diamonds; see Figure~\ref{Figredblue}. In other terms, put a dot in the black square each time the random surface goes down one unit along the line $\xi=2s$. 
We are concerned with the probabilities of the following kinds of events, where $[k,\ell ] $ is an interval of odd integers along the $\eta$-axis (so $k$ and $\ell$ can be taken odd):
 $$\begin{aligned}
& \left\{\mbox{The line}~  \{ \xi=2s \}  ~\mbox{has an $\eta$-gap} \supset [k,\ell]%_{\mbox{\tiny odd}}
 \right\}
 \\& =\left\{\mbox{Interval $[k,\ell] \subset \{ \xi\!=\!2s \}$ in $\eta$-coordinates contains no dot-particles}\right \}
\\&=\left\{\mbox{The random surface is flat along the $\eta$-interval   $[k,\ell]%_{\mbox{\tiny odd}}  
 \subset \{ \xi=2s \}$}\right\}
 \\&=\left\{\mbox{Dominos covering $[k,\ell] \subset \{ \xi\!=\!2s \}$ are pointing to the left of} \right. \\
& \left. \, \mbox{or above the line $\{ \xi\!=\!2s \}$}\right\}
 \\&=\left\{\mbox{Dominos covering $[k,\ell] \subset \{ \xi\!=\!2s \}$ are red or yellow in upper Figure \ref{FigKLparticles}}\right\}
\end{aligned}
 $$ 

 \begin{theorem} \label{main1'}  The ${\mathbb L} $-particles on the successive lines $\{\xi=2s \}$ for $1\leq s\leq 2n-\rho$ form a determinantal point process with correlation kernel
\be
\begin{aligned}
%\lefteqn{L_{n,m} (\xi_1,\eta_1;\xi_2,\eta_2)}\\
{\mathbb L}_{n,\rho }%^{\mbox{\tiny \rm twoAzt}}
 (\xi_1,\eta_1;\xi_2,\eta_2)  =&(1+a^2){\mathbb L}^{(0)}_{n }(\xi_1,\eta_1;\xi_2,\eta_2)
\\  
%\nonumber
&- (1+a^2)\langle ((I-{K}_n)^{-1}_{_{\geq n-\rho+1}} A_{\xi_1,\eta_1}) (k), B_{\xi_2,\eta_2}(k) \rangle_{_{\geq n-\rho+1 }}.
\end{aligned}
\label{Lkernel}
\ee
given by a perturbation of      %the one-Aztec diamond 
  a kernel ${\mathbb L}^{(0)}_{n }$ by an inner-product\footnote{$\la f(k), g(k)\ra_{\geq \alpha}=\sum_{\alpha}^{\infty} f(k)g(k) 
$ is an inner product in $\ell^2[\alpha,\infty]$.} involving the resolvent of yet another kernel $K_n$, all given by formulas (\ref{oneAzt}) (section (\ref{L0kernel}))
 This shows that, given $q$ lines $\{\xi=2s_i \}$ and  integers $k_i,~\ell_i$, with $0\leq k_i< \ell_i\leq n$, the gap probability is expressed as the 
Fredholm determinant\footnote{The variables $\eta_i$ below run through odd values only.} 
  \be\begin{aligned}
        \lefteqn{
        {\mathbb P}\left(\bigcap_{i=1}^{q}\left\{\mbox{The line}~  \{ \xi=2s_i \}  ~\mbox{has an $\eta$-gap} \supset [k_i,\ell_i]%_{\mbox{\tiny odd}}
    %\right\}
   \right\}\right)
   }\\
  & \hspace*{.1cm}     =\det\left(\Id-\left[\chi_{[ k_i , \ell_i ]} (\eta_i){\mathbb L}%^{\mbox{\tiny \rm twoAzt}}
   _{n,\rho}(2s_i,\eta_i;2s_j,\eta_j)\chi_{[ k_j , \ell_j ]}(\eta_j)\right]_{1\leq i,j\leq q}\right),
    \end{aligned}  \label{I2}  \ee
of the kernel ${\mathbb L}_{n,\rho}%^{\mbox{\tiny \rm twoAzt}}
 $.

 \end{theorem}
 
This theorem will be proved in section \ref{Kernels}.

 \begin{figure}
 \includegraphics[height=2.75in]{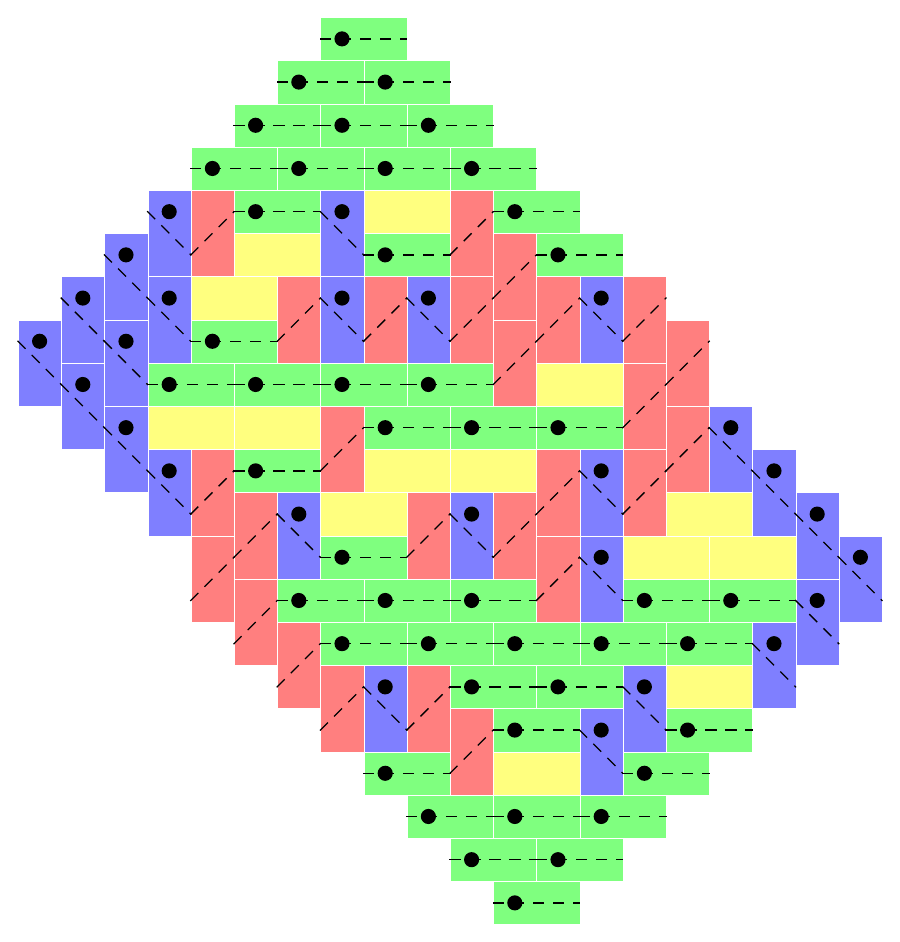} \hspace{2mm}
  \includegraphics[height=2.75in]{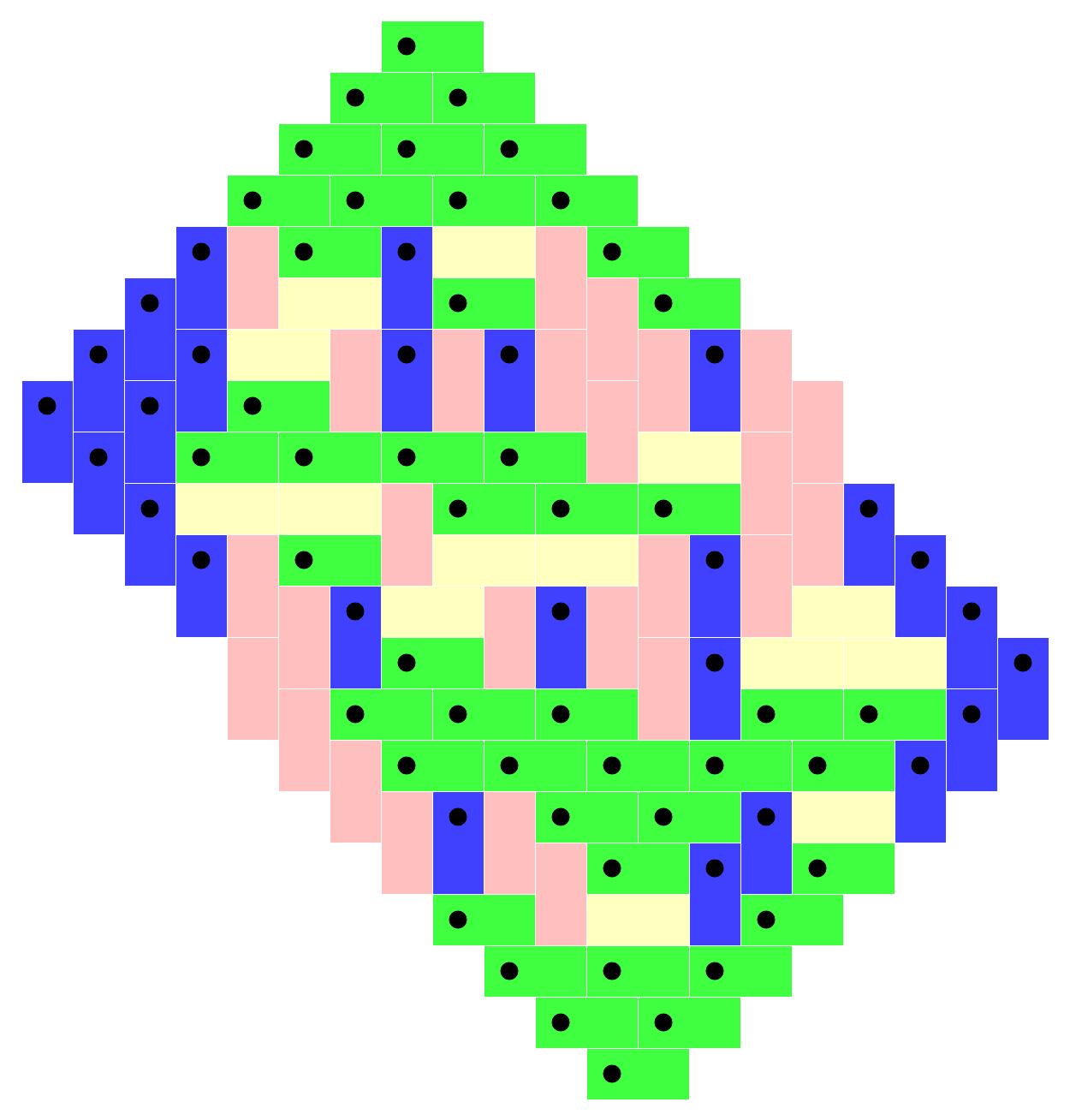} \vspace{5mm}\\
  \includegraphics[height=2.75in]{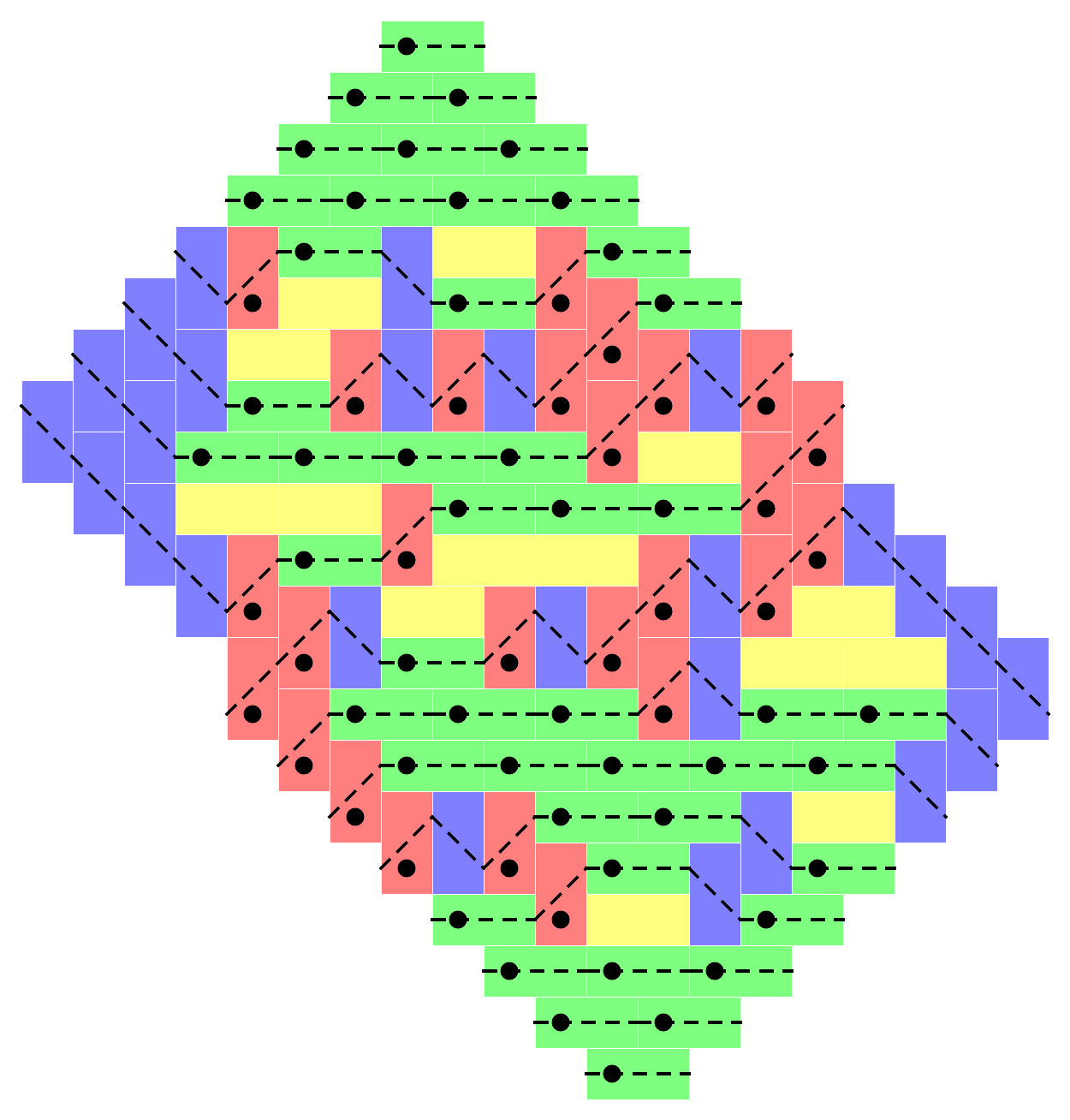}\hspace{2mm}
  \includegraphics[height=2.75in]{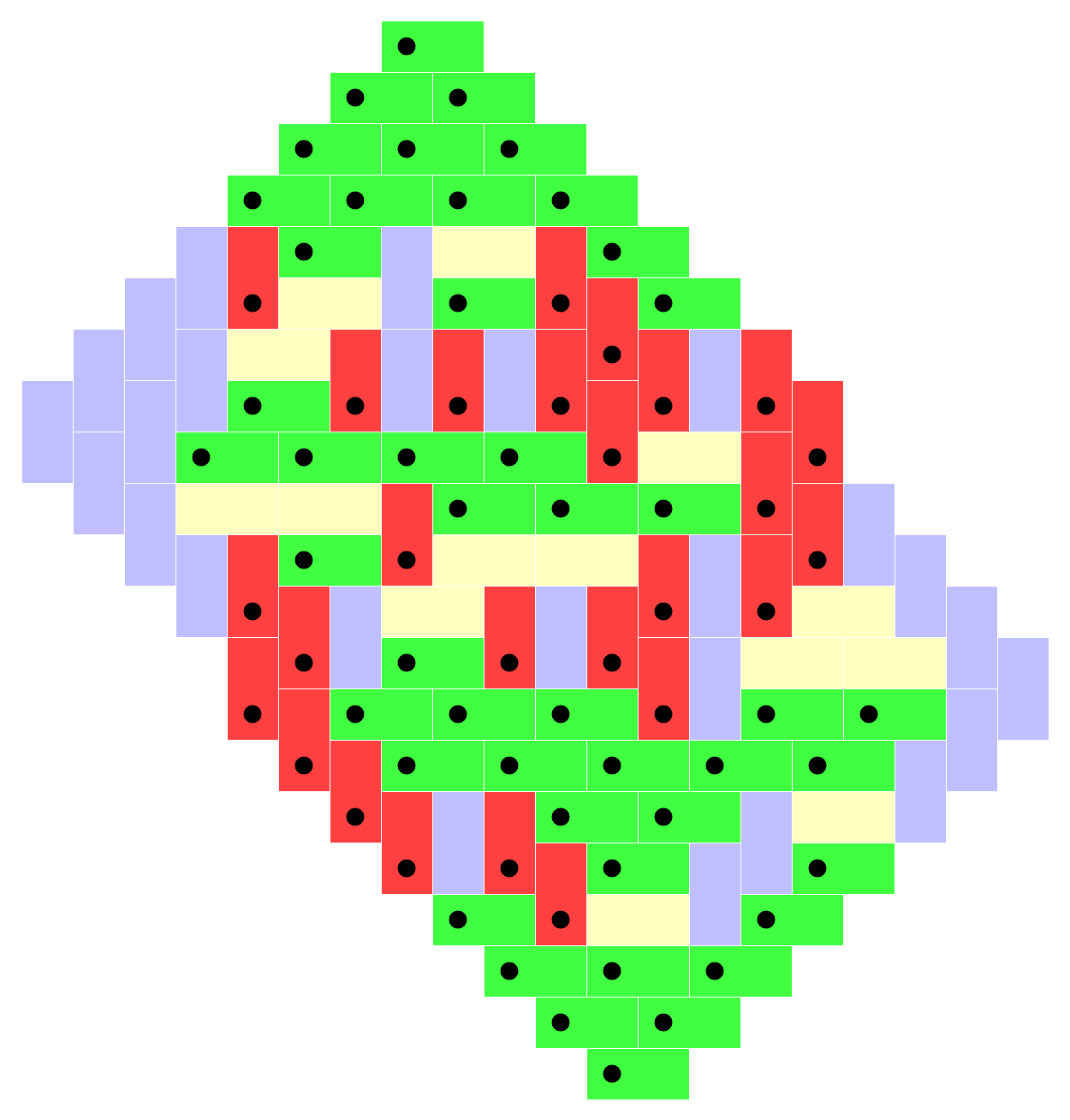}
 \caption{The top figures show the $\mathbb{L}$ particle process using the level lines (on the left) and the dominos on the right.  There is an $\mathbb{L}$ particle for every green (south) and blue (east) domino.  The bottom figures show the $\mathbb{K}$ particle process. There is a $\mathbb{K}$ particle for every green (south) and red (west) domino.} 
 \label{FigKLparticles}
 \end{figure}

{ \bf II.\em  ~ The ${\mathbb K}%^{\mbox{\tiny \rm twoAzt}}
 $-process}. Now we put instead a blue dot in the middle of the black square when the line $z=2k$ in $(z,x)$-coordinates for $1\leq k\leq n$ intersects an $A$-level curve and a red dot when intersecting a $B$-level curve; i.e., put a dot each time the random surface goes down one unit along the line $z=2k$; see Figure~\ref{FigKLparticles}. These dots define the 
{\em $\mathbb{K}$-particles}.
In this instance, we are concerned with the probabilities of the following kinds of events, where $[k,\ell ] $ is an interval along the $x$-axis:
 $$\begin{aligned}
& \left\{\mbox{The line}~  \{ z=2r \}  ~\mbox{has an $x$-gap} \supset [k,\ell]\right\}
 \\& =\left\{\mbox{The interval $[k,\ell] \subset \{ z\!=\!2r \}$ in $x$-coordinates contains no dot-particles}\right\}
%\\&= \left\{\begin{aligned}&\mbox{The domino's covering black squares along the  }\\
%& \mbox{interval $[k,\ell] \subset Y_{2r}$ are oriented left or down}\end{aligned}\right\}
 \\&=\left\{\mbox{The random surface is flat along the $x$-interval $[k,\ell] \subset \{ z=2r \}$}\right\}
 %\end{aligned}
 %$$
 %$$\begin{aligned}  
%\\ &=\left\{\begin{aligned}&\mbox{The line $X_{2r} $ does not intersect any of the outlier}\\
%  &\mbox{paths in the black squares along $[k,\ell]\subset X_{2r}$ }\end{aligned}\right\}
 \\&=\left\{\mbox{Dominos covering $[k,\ell] \subset \{ z\!=\!2r \}$ are pointing to the left } \right.
\\& \, \left. \mbox{
or below $\{ z\!=\!2r \}$}\right\} 
 \\&=\left\{\mbox{Dominos covering $[k,\ell] \subset  \{ z\!=\!2r \}$ are blue or yellow in lower Figure \ref{FigKLparticles}}\right\}
 \end{aligned}$$ 

%$$ 

\begin{theorem} {\rm \cite{AJvM}}\label{main2'} The ${\mathbb K} $-particles on the successive lines $\{z=2r \}$ for $1\leq r\leq n$ form a determinantal point process with correlation kernel
 given by perturbing the one-Aztec diamond  kernel ${\mathbb K}^{0}_{n }%^{\mbox{\tiny \rm twoAzt}}
 $ with an inner-product 
 involving the resolvent of the kernel $K_n%_{2m+1}^{(1)}(0)
 $, all defined in (\ref{K13}) and (\ref{K12}):
\be
  \begin{aligned}
  {(-1)^{x-y}    \BK%^{\mbox{\tiny \rm twoAzt}}
   _{n,\rho}(2r,x;2s,y) }  
 =~& {\mathbb K}_n ^{0}
  \bigl(2r,x;2s,y\bigr) %  
% -\Id_{s<r} (-1)^{x-y}\psi_{2(s-r)}(x,y)  +S(2r,x;2s,y)%
 \\  
 &~~ -\left\la(\Id-K_n )_{_{\geq n-\rho+1}}^{-1}a_{-y,s}(k),b_{-x,r}(k) \right\ra _{_{\geq n-\rho+1}}    .\end{aligned}  %
 \label{K1}\ee
This shows that given $q$ lines $\{z=2r_i \}$ and  integers $k_i,~\ell_i$, with $ r_i-m-n\leq k_i< \ell_i\leq  r_i+m$, we have
  \be\begin{aligned}
        \lefteqn{
        {\mathbb P}\left(\bigcap_{i=1}^{q}\left\{\mbox{the line}~\{z=2r_i\} ~\mbox{has an $x$-gap} \supset [ k_i , \ell_i ]\right\}\right)
   }\\
  & \hspace*{2cm}     =\det\left(\Id-\left[\chi_{[ k_i , \ell_i ]}(x_i)
   {\mathbb K}_{n,\rho}%^{\mbox{\tiny \rm twoAzt}}
   (2r_i,x_i;2r_j,x_j)\chi_{[ k_j , \ell_j ]}(x_j)\right]_{1\leq i,j\leq q}\right).
    \end{aligned}  \label{I2}  \ee
with a kernel ${\mathbb K}_{n,\rho}%^{\mbox{\tiny \rm twoAzt}}
 $.
  
\end{theorem}

The theorem was proved in \cite{AJvM} where the $\BK$-particles were called outlier particles.

\medbreak

\noindent The dot particles of the ${\mathbb L}$-process satisfy the following interlacing pattern.

\begin{figure}[h!]
 \begin{center}
  \includegraphics[height=3.5in]{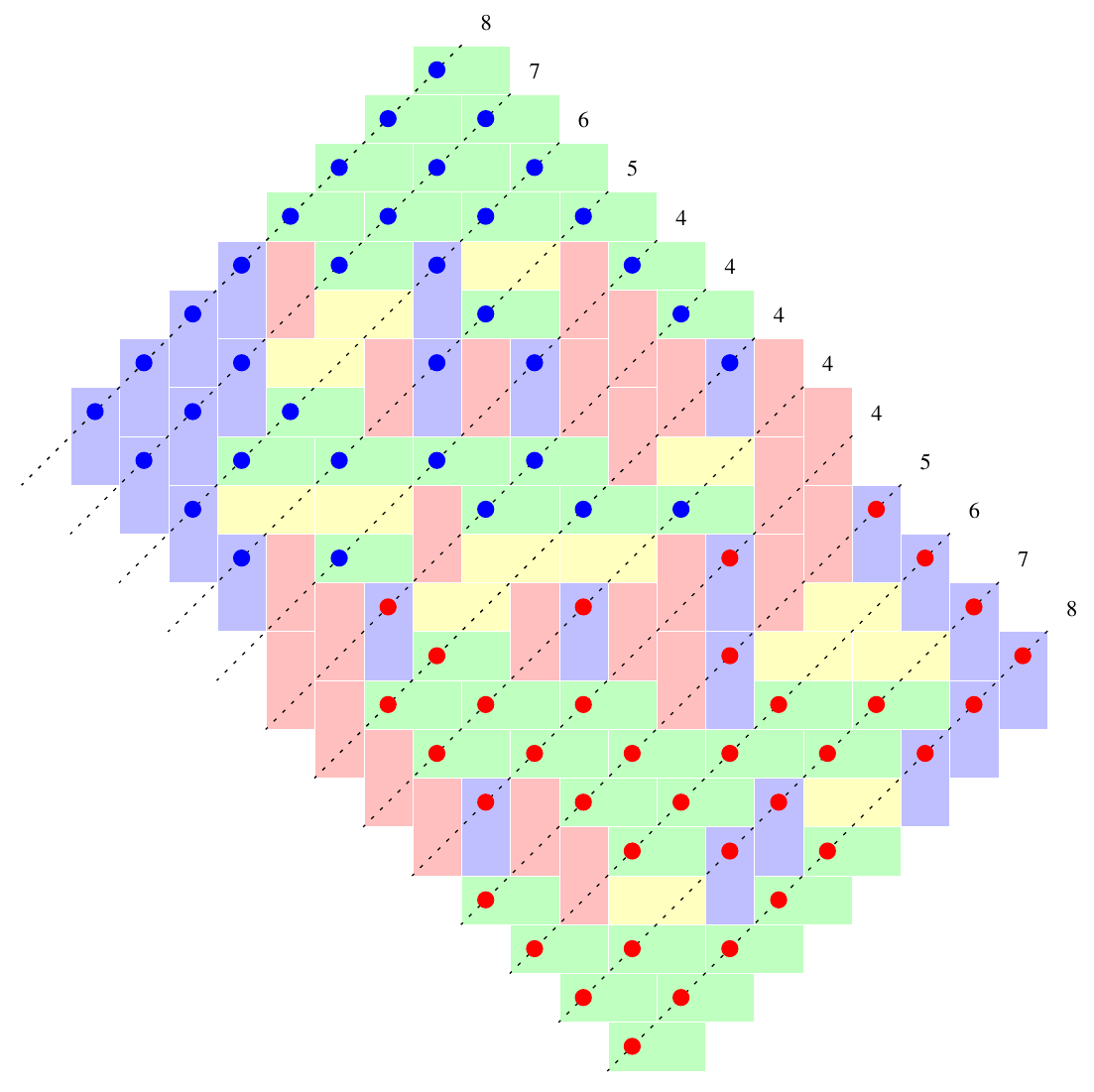}
  \caption{The red and blue $\mathbb{L}$ particles.  The blue $\mathbb{L}$ particles correspond to Aztec diamond $A$ while the red particles correspond to Aztec diamond $B$ see Figure~\ref{Fig:Kastcoords}.  The numbers represent the total number of particles on each line $\xi=2s$}
\label{Figredblue}
 \end{center}
\end{figure}

\begin{proposition} \label{Interla}
 For the ${\mathbb L}$-process, the lines $\xi=2s$ contain blue dots and red dots, according to the following interlacing patterns, with varying numbers:
$$\begin{array}{llllllllll}
 
 \mbox{lines $\xi=2s$}&\vline&\in\mbox{Diamond}&\vline  & \mbox{\# of blue and red dots} \\
 \hline
 0\leq  s< n-\rho&\vline&\in A&\vline  &  \mbox{$n-s$ blue dots  }\\  
 s=n-\rho&\vline &\in A &\vline & \rho \mbox{ blue dots}\\
 n-\rho<  s <  n &\vline &\in A\cap B&\vline & \rho 
\mbox{ dots with }\left\{\begin{aligned} &\mbox{$n-s$ blue dots} 
 \\ 
 &\mbox{to the right of}\\&  \mbox{$s- n+\rho$ red dots}\\ &\mbox{for each } s \\\end{aligned}\right.  \\
 s=n &\vline &\in A\cap B&\vline  & \rho \mbox{ red dots}\\
  n< s\leq  2n-\rho &\vline&\in  B&\vline & \mbox{$\rho+s-n$ red dots  }\\
 \end{array}
 %\vspace*{.93cm}\left.\begin{array}{c}\\ \\ \\ \end{array}\right\}
 $$
 with interlacing of the blue dots and interlacing of the red dots, with regard to the $\eta$-coordinate; in the overlap of the two diamonds, the dots interlace as well, with the right most dot on the line $\xi=2s$ being to the right of the right most dot on the line $\xi=2s+2$;  also the left most dot on the line $\xi=2s+2$ is to the left of the left most dot on the line $\xi=2s$. Notice that the overlap contains $\rho$ lines (through black squares) $\xi=2s$ with $n-\rho<s\leq n$.
  \end{proposition}
 
 Figure~\ref{Figinterlace} shows the above proposition schematically for the example tiling in Figure~\ref{Fig:Smallsims}. The proposition will be proved in section
\ref{Sec:Interlacing}.

 \begin{figure}
 \begin{center}
  \includegraphics[height=3in]{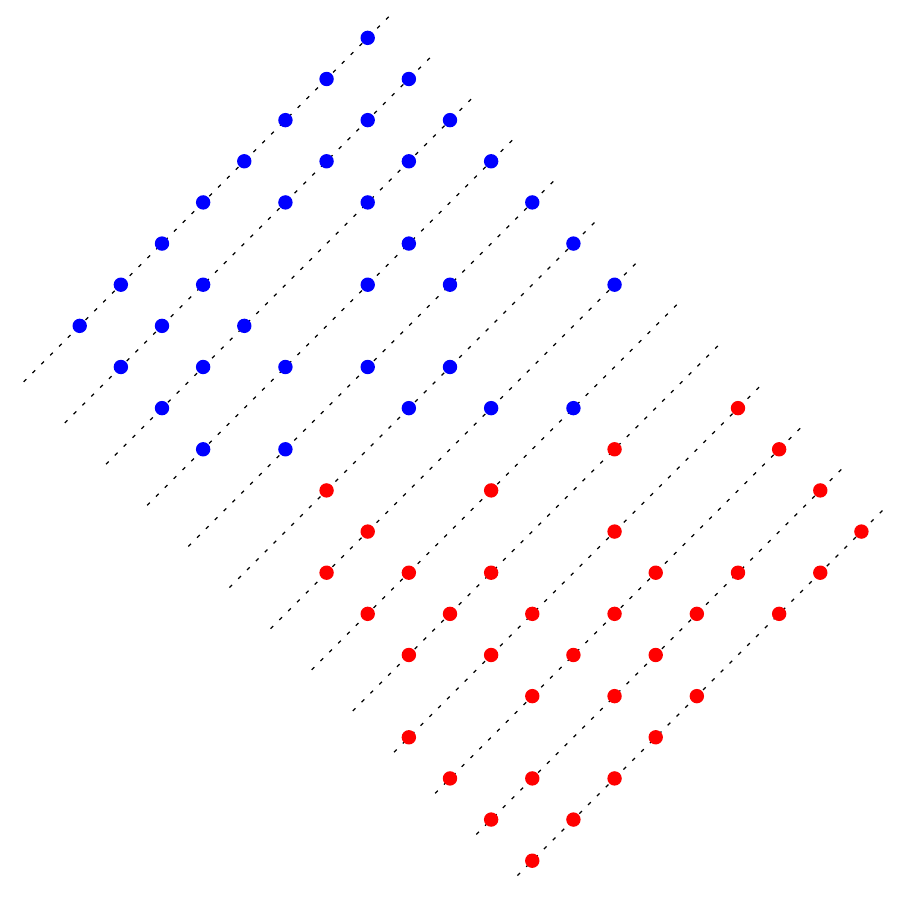}
  \caption{The interlacing system of blue and red dots. The $\rho+1$ lines, $\{\xi=2s\}\in A\cap B$ and $\{\xi=2(n-\rho)\} \in A$, contain $\rho$ dots. All the other lines contain more dots.  See Proposition~\ref{Interla} for more details on the interlacing.} 
 \label{Figinterlace}
 \end{center}
 \end{figure}

It is interesting to notice that, passing from the ${\mathbb L}%^{\mbox{\tiny \rm twoAzt}}
 _{n,\rho}$-process to the ${\mathbb K}
 %^{\mbox{\tiny \rm twoAzt}}
  _{n,\rho}$-process, the dot is maintained in the horizontal dominos, whereas a dot in an East domino gets replaced by a dot in a West domino; compare the pictures given in Figure~\ref{FigKLparticles}.
%$$

 Recall the $\BK_{n,\rho}$-point process is a process of dots along the lines $\eta=2k-1$ or what is the same $z=2k $. For the sake of the main theorem below, the point of view will be switched around: namely, the $\BK_{n,\rho}$-process induces a determinantal process of dots along the lines $\xi=0$ up to $\xi=2(2n-\rho)$, inherited from the dots on the lines $z=2k$. As mentioned, this process can be obtained from the $\BL_{n,\rho}$-process by keeping the dots belonging to the horizontal domino's and moving the dots from the vertical domino's with a black square below to the vertical ones with a black square above; see Figure~\ref{FigKLparticles}.

\subsection{ The Tacnode GUE-minor kernel and the main Theorem}  
We now define a new kernel, the coupled GUE-minor kernel, depending on two parameters $\beta,\rho$: %(later $\rho=2\sigma=$ overlap of two diamonds)
 \footnote{The subscript $\geq -\rho$ refers to the space $\ell^2(-\rho,\ldots,\infty)$.}
\be
\begin{aligned}
 \BK_{\beta, \rho}^  {\mbox{\tiny tac}}& (u_1,y_1;u_2,y_2)
\\ = & \BK^ {\mbox{\tiny minor}}  (u_1,\beta-y_1 ;~u_2 ,\beta-y_2 ) \\&~~+2\Bigl\la (\Id - {\cal K}^\beta ( \lambda ,\kappa ))^{-1}_{\geq -\rho} ~{\cal A}^{\beta,y_1-\beta}_{u_1 }(\kappa), {\cal B}^{\beta,y_2-\beta}_{u_2 }(\lambda)\Bigr\ra
  _{_{ \geq -\rho }}
   \label{2min0}\end{aligned}\ee
  where $\BK^ {\mbox{\tiny minor}}  (n,x;n',x')  $ is the GUE-minor kernel, defined for $n,n'\in \BZ$, rather than $\mathbb N$. 
This kernel will appear below as the appropriate scaling limit of the $\BL$-particle kernel.
Here and below we shall define functions, which involve integration over small circles $\Gamma_0$ and an imaginary line $L:=0^{+}+i\BR\! \uparrow$; the line $L$ needs to be always to the right of the contour $\Gamma_0$. Define
\be
\begin{aligned}
 \BK^ {\mbox{\tiny minor}}  (n,x;n',x')   
 := &-\Id_{n>n'}2^{n-n'}\BH^{n-n'}(x-x')
\\
&+\frac{2}{(2\pi \I)^2}
\int_{\Gamma_0}dz\int_{L} \frac{dw}{w-z}
 \frac{e^{-z^2+2zx}}{e^{-w^2+2wx'}}\frac{w^{n'}}{z^n}
,\end{aligned}
\label{8.6}\ee
with  $\BH^{m}(z)$ defined for $m\geq 1$ as
\be
\begin{aligned}
   \BH^{m}(z)&:=\frac{z^{m-1}}{(m-1)!}\Id _{z\geq 0}.
 \end{aligned} \label{Hm} \ee
Kernel (\ref{2min0}) contains the functions:
  \be\begin{aligned}
{\cal K}^\beta (\lambda,\kappa)&:=\oint_{\Gamma_0}  \frac{d\zeta} {(2\pi \I)^2 }\int_{ L} \frac{d\om}{ \om-\zeta   }
   \frac{e^{-2\zeta^2+4 \beta \zeta }} { e^{-2\omega^2+4 \beta \omega }}
\frac{\zeta^{\kappa } }{\omega^{\lambda +1}} 
\\
   {\cal A}^{\beta,y }_{v }(\kappa)&:=  %\frac{  (-1)^{ -\sigma+\kappa}e^{2(\sigma+ \tau_2\sqrt{2\sigma}})}{ (2\pi i)^2}  %\sqrt{2\sigma}(2\tau_2+\sqrt{2\sigma})
 ~
 %\\
 %&~~~~~~~~~~~~~~
  \oint_{\Gamma_0}  \frac{d\zeta}{ (2\pi \I)^2} 
   \!\! \int_{ L} \frac{d\om}{\zeta\!-\! \om  }
  \frac{e^{-\zeta^2-2  y  \zeta}} 
  {e^{-2\om^2+4 \beta \om}}  %\frac{\om^{2\sigma-\kappa-1}}{\zeta^{\eta+\sigma+1}} 
  %@@ 
  \frac{\zeta^{-v }}{\om^{\kappa +1}} 
%
%\\
%&  ~~~~~~~~~  %(\sqrt{t})^{  \sigma-\eta- \kappa-1 }  
+% (-1)^{    -\sigma+\kappa }
 %e^{2(\sigma+\tau_2\sqrt{2\sigma})}
  \int_{L}\frac{d\zeta}{2\pi \I}\frac{e^{\zeta^2-2\zeta (y +2\beta) }}
  {\zeta^{v+\kappa +1 }}
%  ,\end{aligned}\ee
 % 
%\be\begin{aligned}
\\
{\cal B}^ {\beta,y }_{u}(\lambda)&:= %\\
 %&~~~~~~~~~~~~~~
  \oint_{\Gamma_0}  \frac{ d\zeta}{ (2\pi \I)^2}   \int_{ L} \frac{d\om}{ \zeta\!-\!\om   }
  ~~\frac{e^{-2\zeta^2+4\zeta \beta }} {e^{- \om^2-2\om  y    }} % \frac{\om^{\xi+\sigma }}{\zeta^{2\sigma-\lambda}}
 % @@
  \frac{\zeta^{\lambda  }}{\om^{-u  }}
%
%\\
%&  ~~~~~~~~~  %(\sqrt{t})^{  \sigma-\eta- \kappa-1 }  
+\oint_{\Gamma_0}\frac{d\omega}{2\pi \I} ~~
% {e^{-\omega^2+2(\tau_1+\sqrt{2\sigma}) \omega}}{\omega^{\xi-\sigma+\lambda }}, 
 \frac{\omega^{u +\lambda  }}
  {e^{ \omega^2-2 \omega (y +2\beta)}}.\end{aligned}
  \label{defAB}\ee %
To be precise in the limit theorems we should replace $\Id _{z\geq 0}$ by $\Id _{z>0}+\frac 12 \Id_{z=0}$ in the case of the $\mathbb{L}$-process, (\ref{limL})
below, and by $\Id _{z>0}$ in the case of the $\mathbb{K}$-process, (\ref{limK}) below. Since these changes do not affect the limiting point process we will ignore this fine point. Some properties of the kernel are given in section \ref{Sec:GUEminorproperties}. 
 Notice that the scaling in the Theorems below could have been derived from the scaling used in the limit of the $\BK$-kernel to the tacnode process, combined with the way the weight $a\to 1$. \newline \noindent The main statement of the paper reads as follows.
\begin{theorem} \label{a=1}Let the  size $n=2t+\epsilon$, $\epsilon\in\{0,1\}$, 
of the diamonds go to infinity, while keeping the overlap $\rho=n-2m$ finite and, and letting 
the weight of the vertical domino's $a\to 1$ as  
$$
a= 1-\frac {  \beta}{\sqrt{n/2}} , ~~\mbox{with $\beta \in \mathbb{R} $ fixed.}
$$
The %Kasteleyn 
 coordinates $(\xi,\eta)$ are scaled as follows,
\be \begin{aligned}
\xi_i = 4t+2\epsilon-2 u_i &,& \eta_i =2t  +   2[y_i \sqrt{t }]-1 ,~~\mbox{with}~ u_i \in \mathbb{Z} \ , y_i \in \BR.
\end{aligned}
\label{scK}\ee
With this scaling, the following limit holds:
%\footnote{Note that $\Delta \eta_2/2=1$, because of the oddness of $\eta_2$. } $n$~:  
\be
\begin{aligned}
 \lim_{n \to\infty}  ~(-a)^{(\eta_1-\eta_2)/2}(-\sqrt{t})^{(\xi_1-\xi_2)/2} &
   {\mathbb L} _{n,\rho}
    (\xi_1,\eta_1;\xi_2,\eta_2) \sqrt{t}
\\&~~~~~=\BK_{\beta,\rho}^{\mbox{\tiny tac}} (u_1,y_1;u_2,y_2)
 .\end{aligned}
 \label{limL}\ee
We also have that the rescaled $\mathbb{L}$-particle process converges weakly to the determinantal
point process given by the tacnode GUE-minor kernel.
\end{theorem}

%\vspace*{.6cm}

Remember the remark at the end of subsection \ref{subs1.2}.  The $\BK$-process induces a process of particles along the consecutive lines $\{\xi=2s\}$. This is to say the roles of $2r$ and $x$ in the kernel $\BK_{n,\rho}$ get reversed from the point of view of scaling: the variable $2r$ turns into the discrete variable $u$ and the variables $x$ into the continuous variable $y$. The simulations of Figure \ref{Figsim2} show that the lines $\{\xi=2s\}$ belonging to the overlap $A\cap B$ contain long dense stretches of $\BK$-particles. But performing a random thinning, one nevertheless is led in the limit to a point process kernel, which turns out to be the same tacnode GUE-minor kernel, except for some shift. This is the content of the next Theorem.

%The point process of $\BK$-particles does not have a point process scaling limit in the same way as for the $\BL$-particles since there will be long dense stretches of $\BK$-particles, see for an example the relatively large simulation of the $\mathbb{K}$ particles given in Figure~\ref{Figsim2}.  

\begin{theorem}\label{Ka=1}
Let $a$, $n$ and $\rho$ be as in the previous theorem and 
consider the same scaling as above, but expressed in the $(x,z)$-coordinates, using the map (\ref{K}),
\be \begin{aligned}
2x_i%=x=u-\sigma+\tau_1\sqrt{t}
=-\rho +2u_i- \epsilon+ 2[y_i\sqrt{t  }], %~~~x_2=y=v-\sigma+\tau_2\sqrt{t}=u_2-\sigma+\tau_2\sqrt{t},
~~~~~2r_i =2t+  2[y_i\sqrt{t }].%,~~x_i, u_i \in \mathbb{Z} \ , y_i \in \BR.
  %~2r_2=2s=2t+ 2\tau_2\sqrt{t}
  \label{sc1}
\end{aligned}
\ee
% \be
  %
%\ee
Let the ${\mathbb K}$-particles be thinned out at the rate $p_n=1-2/ \sqrt{t}$. We then have the following limit: 
 \be
\begin{aligned}
 \lim_{n\to \infty}      (1-p_n)a%e^{\tfrac {2\beta}{\sqrt{t}}}
  ^{r_2-r_1}%a^{\frac{\eta_2-\eta_1}2} 
    (\sqrt{t} )^{ {x_1-x_2-r_1+r_2 } }%(\sqrt{t})^{\frac 12 (\eta_1-\eta_2+\xi_2-\xi_1} }% (\sqrt{t})^{\frac{\xi_1-\xi_2}{2}}
  %e^{\sqrt{\frac{2\sigma}{t}} (r-s)  }
& (-1  )^{x_1-x_2}%(-1)^{\frac 12(\eta_1-\eta_2+\xi_2-\xi_1)}
  \BK _{n,\rho} (2r_1,x_1;2r
_2,x_2)\sqrt{t}
\\
 &
~~=  \BK_{\beta,\rho}^  {\mbox{\tiny tac}} (u_2+1,y_2;u_1,y_1).
%&=\tfrac 12 \BK^ {\rm\tiny GUE-minor}  (u_2 +1,-y_2+\beta;~u_1 ,-y_1+\beta) \\
%&~~~+\left\la (\Id - {\cal K}^\beta  ( \lambda-2\sigma,\kappa-2\sigma))^{-1} {\cal A}^{\beta,y_2-\beta}_{u_2+1 }(\kappa-2\sigma), {\cal B}^{\beta,y_1-\beta}_{u_1 }(\lambda-2\sigma)\right\ra
 % _{\ell^2(0,1,\ldots,\infty)}
  %\\
  %
  %\BK^ {\rm\tiny GUE-minor}  (u_2 +1,-y_2+\beta;~u_1 ,-y_1+\beta) \\
%&~~~+\left\la (\Id - {\cal K}^\beta ( \lambda ,\kappa ))^{-1} {\cal A}^{\beta,y_2-\beta}_{u_2 +1 }(\kappa), {\cal B}^{\beta,y_1-\beta}_{u_1 }(\lambda)\right\ra
 % _{\ell^2(-2\sigma,\ldots,\infty)}
\end{aligned}
\label{limK}\ee
Interpreted as a weak limit of a point process this means that
if we thin the $\mathbb{K}$-process by removing each $\BK$-particle independently with probability $p_n$, then the
resulting point process converges weakly to a determinantal point process given by the correlation kernel on the right hand side of (\ref{limK}). %
\end{theorem}

%\remark For $n$ even, the overlap $\rho=n-2m$ must be even, whereas for $n$ odd, the overlap must be odd. This is because we are only considering Double Aztec
%diamond where the number of corners in the Double Aztec diamond before the overlap is even ($=2m$).

Notice the kernel $\BK_{\beta,\rho}^  {\mbox{\tiny tac}}$ is the same as the one in Theorem \ref{a=1}, except for the shift in $u_2$ and the flip $u_1\leftrightarrow u_2$ and $y_1 \leftrightarrow y_2$.

The geometry of the level curves for the double Aztec diamond, when $n\to \infty$, looks as in Figure~\ref{Figsim1} below.  The $\mathbb{K}$ and $\mathbb{L}$ particle processes have also been plotted for this simulation and is found in Figure~\ref{Figsim2} below.  As $n \to \infty$, the particles take continuous values on each line; they are constrained by the same interlacing as in Proposition \ref{Interla}. %This precise geometry of the level curves will also be responsible for the $1/2$ in front of limit (\ref{limK}) for the ${\mathbb K}$-process.

\begin{figure}
 \includegraphics[height=3in]{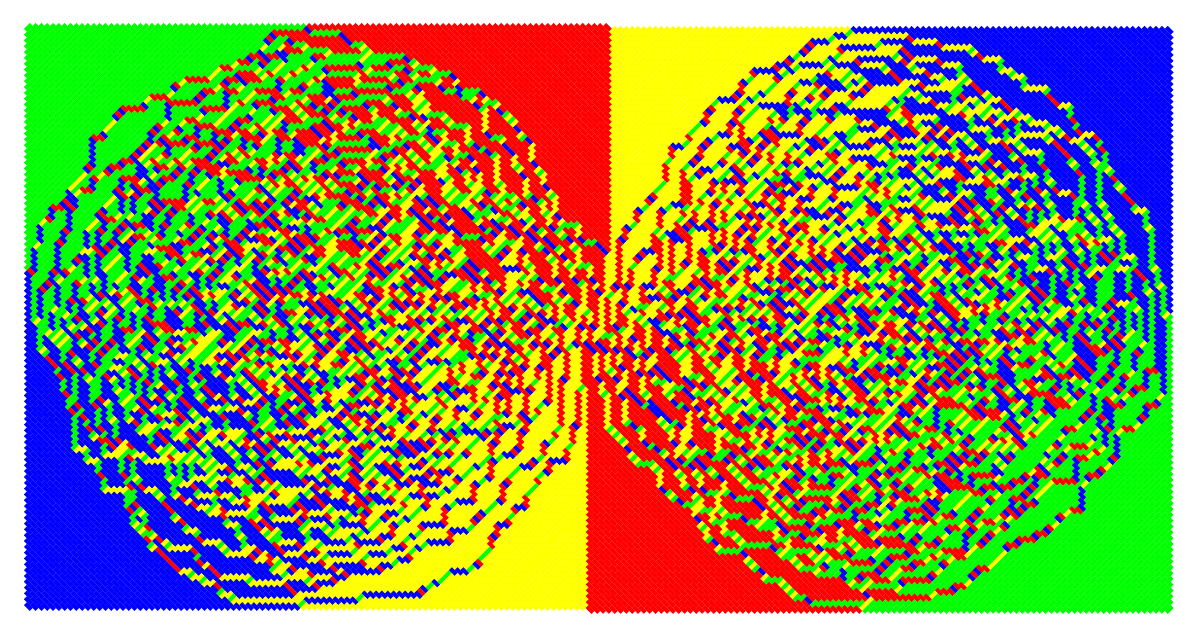}
 \\
 \includegraphics[height=3in]{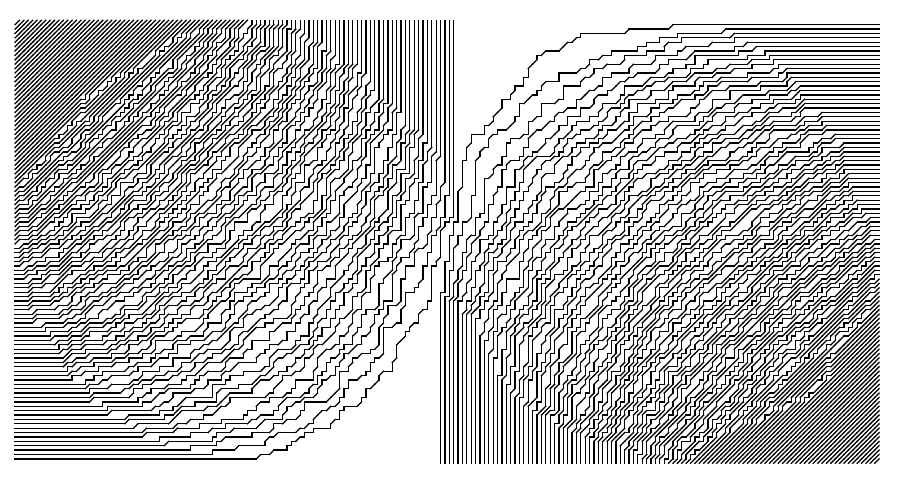}
 \caption{A simulation of a double Aztec diamond with $n=100$ and $\rho=4$ with the weight $a$ of vertical and horizontal tiles equal to $1$.  Both figures are rotated by $\pi/4$ counter-clockwise.  For this simulation, the top figure  shows the underlying domino tiling while the bottom figure shows the level lines. The simulation was made using the generalized domino shuffle~\cite{Pro:03}.} 
 \label{Figsim1}
\end{figure}

\begin{figure}
 \includegraphics[height=3in]{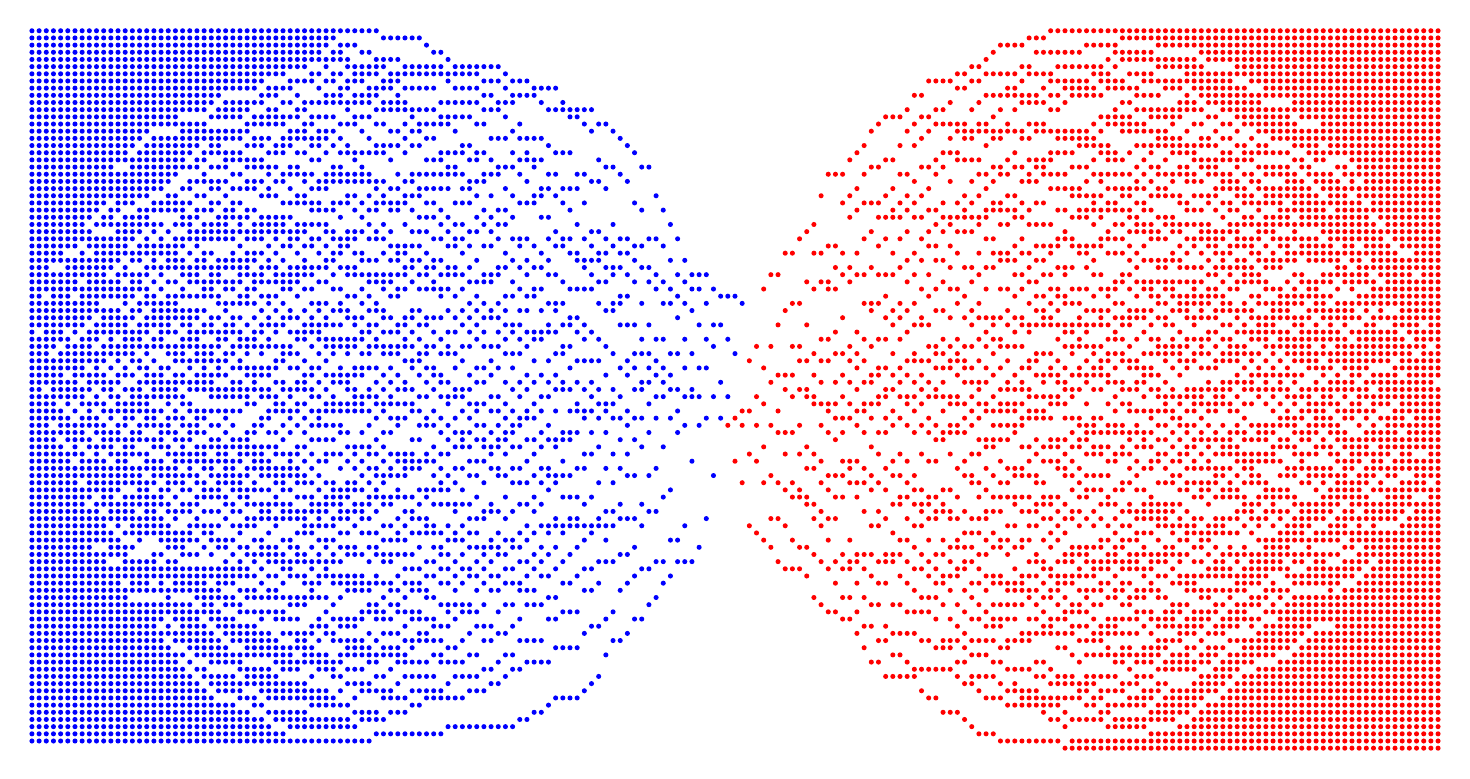} \\
 \includegraphics[height=3in]{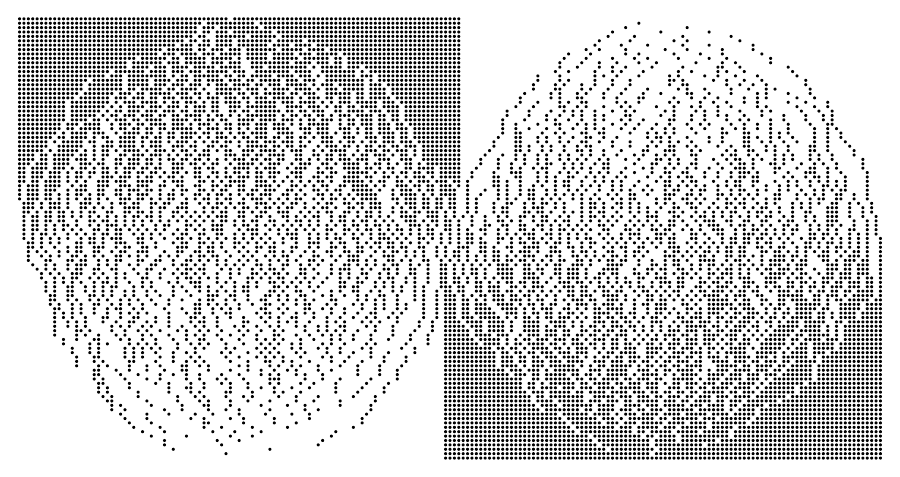}
 \caption{ The top picture shows the blue and red $\mathbb{L}$-particles for the simulation of Figure~\ref{Figsim1}: a particle is created each time a line $\{\xi=2s\}$ traverses a blue or green domino, as in Figure \ref{Figsim1} (or, alternatively, crosses a level line); then blue particles belong to an $A$-level line and a red particles to a $B$-level line.    %@@@@@
  The bottom picture shows the $\mathbb{K}$-particles for the same simulation in Figure~\ref{Figsim1}:  a particle is created each time a line $\{\eta=2r+1\}$ traverses a red or green domino, as in Figure \ref{Figsim1} (or, alternatively, crosses a level line). %As in Figure~\ref{Figsim1}, the pictures are rotated by $\pi/4$ counter-clockwise.
  }
 \label{Figsim2}
\end{figure}

\section{The kernel for the $\BL$-process, via Kasteleyn}\label{Kernels}

%\noindent{\bf 1.3The $\BL$-particle process}

 \subsection{The $\BL$-particle process}

\noindent The kernel $ {\mathbb L} _{n,\rho}$ for the $\BL$-process is given by formula (\ref{Lkernel}), i.e.,
%
% generated by the dots and circles 
\be
\begin{aligned}
%\lefteqn{L_{n,m} (\xi_1,\eta_1;\xi_2,\eta_2)}\\
{\mathbb L}_{n,\rho }%^{\mbox{\tiny \rm twoAzt}}
 (\xi_1,\eta_1;\xi_2,\eta_2)  =&(1+a^2){\mathbb L}^{(0)}_{n }(\xi_1,\eta_1;\xi_2,\eta_2)
\\  
%\nonumber
&- (1+a^2)\langle ((I-{K}_n)^{-1}_{_{\geq n-\rho+1}} A_{\xi_1,\eta_1}) (k), B_{\xi_2,\eta_2}(k) \rangle_{_{\geq n-\rho+1 }},
\end{aligned}
\label{Lkernel1}
\ee
with
$$
\begin{aligned}
\lefteqn{{\mathbb L}^{(0)}_{n} (\xi_1,\eta_1;\xi_2,\eta_2)}\\
& := - \Id_{(\xi_1 < \xi_2)} \int_{\Gamma_{0,a}}  \frac{ dw}{2\pi \I}~\frac{(1+aw)^{(\eta_1-\eta_2)/2-1}}{(w-a)^{(\eta_1-\eta_2)/2+1}} w^{\frac{\xi_1-\xi_2}2} 
\\
&+  \int_{\Gamma_{0,a}} \frac{dz }{(2\pi \I)^2} \int_{\Gamma_{0,a,z}} \frac{dw}{w-z} \frac{(1+az)^{(\eta_1-1)/2}(z-a)^{n-(\eta_1+1)/2}w^{n-\xi_2/2}}{(1+aw)^{(\eta_2+1)/2}(w-a)^{n-(\eta_2-1)/2}z^{n-\xi_1/2}}
%\\
%&=:  -\Id_{(\xi_1 < \xi_2)}{\mathbb L}'^{ }_{n } (\xi_1,\eta_1;\xi_2,\eta_2) + {\mathbb L}''^{ }_{n }  (\xi_1,\eta_1;\xi_2,\eta_2) 
\end{aligned}
\label{L0kernel}
$$
\be
\begin{aligned}
 {K}_n(j,k)& = \frac{(-1)^{j+k}}{(2\pi \I)^2}\oint_{\Gamma_{0,a}} dw \oint_{\Gamma_{0,a,w}} \frac{dz}{z\!-\!w} ~ 
 \frac{z^{n-j}(1+aw)^n(w-a )^{n+1}}{ w^{ n+1-k}(1+az)^n (z-a)^{n+1}},
\\
A_{\xi_1,\eta_1} (k) &=
 \frac{(-1)^{k}}{(2\pi \I)^2}\oint_{\Gamma_{0,a}} dz   
\oint_{\Gamma_{0,a,z}} \frac{dw}{w\!-\!z} \frac{(1+az)^{(\eta_1-1)/2}(z-a)^{n-(\eta_1+1)/2}w^{n-k}}{(1+aw)^n(w-a)^{n+1}z^{n-\xi_1/2}}\\
&-\frac{(-1)^k}{2\pi\I}\oint_{\Gamma_{0,a}}\frac{w^{\xi_1/2-k}}{(1+aw)^{n-(\eta_1-1)/2}(w-a)^{(\eta_1+3)/2}}~dw
%\nonumber
\\
% \end{aligned}
%$$
%
%$$
%\begin{aligned}
 B_{\xi_2,\eta_2}(k) &=  \frac{(-1)^{k}}{(2\pi \I)^2}\oint_{\Gamma_{0,a}} dw\oint_{\Gamma_{0,a,w}} \frac{dz}{w\!-\!z} \frac{(1+aw)^n (w-a)^{n+1} z^{n-\xi_2/2}}{(1+az)^{(\eta_2+1)/2}(z-a)^{n-(\eta_2-1)/2}w^{n+1-k}}\\
&+\frac{(-1)^k}{2\pi\I}\oint_{\Gamma_{0,a}}(1+az)^{n-(\eta_2+1)/2}(z-a)^{(\eta_2+1)/2}z^{k-1-\xi_2/2}~dz
\label{oneAzt}%\nonumber
 \end{aligned}
\ee

As was shown in \cite{AJvM} the $\BK$-particles form a determinantal point process and
 the kernel $\BK_{n,\rho}$ for the $\BK$-process is given in $(z,x)$ coordinates by
 \be
  \begin{aligned}
  {(-1)^{x-y}    \BK%^{\mbox{\tiny \rm twoAzt}}
   _{n,\rho}(2r,x;2s,y) }  
 =~& {\mathbb K}_n ^{0}
  \bigl(2r,x;2s,y\bigr) %  
% -\Id_{s<r} (-1)^{x-y}\psi_{2(s-r)}(x,y)  +S(2r,x;2s,y)%
 \\  
 &~~ -\left\la(\Id-K_n )_{_{\geq n-\rho+1}}^{-1}a_{-y,2s}(k),b_{-x,2r}(k) \right\ra _{_{\geq n-\rho+1}},
\end{aligned}  %
 \label{K1}\ee
 where $K_n(j,k)$ is defined as before in (\ref{oneAzt}) and where
 \be
 \begin{aligned}
     {\mathbb K}_n ^{0}\bigl(2r,x;2s,y\bigr) &= {\mathbb K}_{n+1}^{\mbox{\tiny \rm OneAzt}}
  \bigl(2(n-r+1 ),m-x+1;2(n-s+1 ),m-y+1\bigr) %  
% -\Id_{s<r} (-1)^{x-y}\psi_{2(s-r)}(x,y)  +S(2r,x;2s,y)%
 \\ &=-\Id_{s<r}  (-1)^{x-y}\psi_{2r,2s}(x,y)  +S(2r,x;2s,y) 
 \end{aligned} 
\label{K13}
\ee
\begin{equation}
\begin{split}
& a_{x,2s+\epsilon_1}(k) %&:=\oint_{\gamma_{r_1}}\frac{dz}{z} F_{x,s}(z) h^{(1)}_k(\tfrac 1z)\\
 :=\frac{(-1)^{k-x}}{(2\pi \I)^2}\oint_{\Gamma_{0,a}}du
 \oint_{\Gamma_{0,a,u}}\frac{dv}{u-v}~
 \frac{v^{x+m}}{u^{k+1}}\frac{(1+av)^s(1-\frac av)^{n-s+1-\epsilon_1}}
 {{(1+au)^n}{ (1-\tfrac au)^{  n+1 }}%\vp_a(2n;u)% (1+av)^n(1-\frac av)^{n+1}
 }\\
&\, = \frac{(-1)^{k-x}}{(2\pi \mathrm{i})^2} \oint_{\Gamma_{0,a}} dv \oint_{\Gamma_{0,a,v}} \frac{du}{u-v} \frac{v^{x+m}}{u^{k+1}} \frac{ (1+a v)^s (1-\frac{a}{v})^{n-s+1-\epsilon_1}}{(1+a u)^n(1-\frac{a}{u} )^{n+1}}\\
&\, - \frac{(-1)^{k-x}}{2 \pi \mathrm{i}} \oint_{\Gamma_{0,a}} dv \frac{v^{x+m-k-1}}{(1+a v)^{n-s} (1-\frac{a}{v})^{s+\epsilon_1}}\\
%&\hspace{10mm}= \\
  %
\end{split}\end{equation}
\begin{equation}\begin{split}
 & b _{y,2r+\epsilon_2}(\ell) % &:=\oint_{\gamma_{r_2}} \frac{dw}{w} G _{y,r}(w) h^{(2)}_\ell(w).\\
   :=\frac{(-1)^{\ell-y}}{(2\pi \I)^2}\oint_{\Gamma_{0,a}}du
 \oint_{\Gamma_{0,a,u}}\frac{dv}{v-u}~
 \frac{v^{\ell}}{u^{y+m+1}}\frac{{(1+av)^n}{ (1-\tfrac av)^{  n+1 }}%\vp_a(2n;v)%(1+au)^n(1-\frac au)^{n+1}
 }{(1+au)^r(1-\frac au)^{n-r+1-\epsilon_2}}
 \\
&\, =\frac{(-1)^{l-y}}{(2\pi \mathrm{i})^2} \oint_{\Gamma_{0,a}} dv \oint_{\Gamma_{0,a,v}} \frac{d u}{v-u} \frac{v^l}{u^{y+m+1}} \frac{(1+a v)^n(1-\frac{a}{v})^{n+1}}{(1+a u)^r(1-\frac{a}{u})^{n-r+1-\epsilon_2}} \\
&\, +\frac{(-1)^{l-y}}{2\pi \mathrm{i}} \oint_{\Gamma_{0,a}} dv \frac{ (1+a v)^{n-r} (1-\frac{a}{v})^{r+\epsilon_2}}{v^{y+m-l+1}} \\
& S(2r+\epsilon_1,x; ~2s+\epsilon_2, y)  %:={\oint_{ \gamma_{r_3}}\frac{ dz}{z}  \oint_{ \gamma_{r_2}} \frac{dw}w \frac{F_{-y,s}(z)G _{-x,r}(w)w}{w-z}  } \\
 : =\frac{(-1)^{x-y}} {(2\pi \I)^2}\!\oint_{\Gamma_{0,a}}du
 \oint_{\Gamma_{0,a,u}}\frac{dv}{v-u} \\
&\hspace{60mm} \frac{v ^{x-m-1 } }{u ^{y-m } } 
 \frac{(1+au)^s(1-\tfrac au)^{n-s+1-\epsilon_2}}
 {(1+av)^r(1-\tfrac av)^{n-r+1-\epsilon_1}}
\\
&\psi_{2r+\epsilon_1,2s+\epsilon_2}(x,y):= \oint _{\Gamma_{0,a}} \frac{dz}{2\pi iz} z^{x-y}
\frac{(1+az)^{s-r}}{(1-\frac az)^{s-r+\epsilon_2-\epsilon_1}}.
\end{split}\label{K12}
\end{equation}

%$\gamma_{r}$ refers to circles of radius $r$ around $0$, with radii subjected to the following inequalities:
%  \be 
%0<a<r_3<r_2<s_2<s_1<r_1<s_3<a^{-1}~~\mbox{  and  }~~ 0<a<\rho<1.\label{rad'}\ee

A single Aztec diamond of size $n$ leads to a determinantal process as well (see   \cite{Johansson3}), for which the kernel is given by the following expression:
%. Also, In view of Theorem \ref{main1'}, define the kernel  by setting $u\mapsto -u,~v\mapsto -v$,
\be\begin{aligned}
 \lefteqn{{\mathbb K}_{n }^{\mbox{\tiny OneAzt}}
 (2r,x;2s,y)}\\
 &\!=\!\frac {(-1)^{x\!-\!y}}{(2\pi \I)^2} \oint_{\gamma_{r_3}}du 
  \oint_{\gamma_{r_2}}\frac{dv}{v\!-\!u} 
 \frac{v^{-x}}{u^{1-y}} \frac{(1+au)^{n -s}(1-\tfrac au)^{s}}
 {(1+av)^{n -r}(1-\tfrac av)^{r}}
   -\Id_{s>r}\psi_{2r,2s}(x,y)% \int_{\Gamma_{0,a}}\frac {dz}{2\pi \I z} 
  % z^{x-y}\left(\frac{1+az}{1-\frac az}\right)^{2(s-r)}
 . \end{aligned} \label{OneAztec}\ee 

It is not immediate to go from knowing the kernel for the $\BK$-particle process to the kernel for the $\BL$-particles process. 
To do so we will use the fact that we can get the inverse Kasteleyn matrix for the dimer version of Double Aztec diamond. The inverse
Kasteleyn matrix will be explained in terms of the kernel $\BK_{n,\rho}$. Using the inverse Kasteleyn matrix it is possible
to show that the $\BL$-particles form a determinantal point process and compute the kernel.

 \subsection{The Kasteleyn Matrix}
%\noindent{\bf 2.2. The Kasteleyn Matrix} 

 Suppose that $G=(V,E)$ is a bipartite graph. A dimer is an edge and a dimer covering is a subset of edges such that each vertex is incident to only one edge.  The dual of the double Aztec diamond is a subset of the square grid graph with a certain boundary condition while a domino tiling of the double Aztec diamond is a dimer covering of its dual graph.    Kasteleyn, in \cite{Kas:61}, introduced a matrix, later named the \emph{Kasteleyn matrix} which one can use to compute the number of domino tilings of the graph.  Since the graph in this paper is bipartite, the Kasteleyn matrix is a type of signed weighted (possibly complex entries) adjacency matrix with rows indexed  by the black vertices and columns indexed by the white vertices.   The sign of the entries is chosen so that the product of the entries of the Kasteleyn matrix for edges surrounding each face is negative.  This is called the \emph{Kasteleyn orientation}.  We will describe the Kasteleyn matrix for the double Aztec diamond below but first we state Kasteleyn's theorem for bipartite graphs and Kenyon's formula \cite{Ken:97}.
 
 Suppose that $K$ denotes the Kasteleyn matrix for a finite bipartite graph $G$.  
\begin{theorem}[\cite{Kas:61}]
The number of weighted dimer coverings of $G$ is equal to $|\det K|$ 
\end{theorem}
Suppose that $E=\{e_i\}_{i=1}^n$ are a collection of distinct edges with $e_i=(b_i,w_i)$, where $b_i$ and $w_i$ denote black and white vertices. 

\begin{theorem}[\cite{Ken:97}] \label{Kenyonsthm}
    The dimers form a determinantal point process on the edges of $G$ with correlation kernel given by
    \begin{equation}
	  L(e_i,e_j) = K(b_i,w_i) K^{-1} (w_i,b_j) 
    \end{equation}
 where $K(b,w) = K_{bw}$ and $K^{-1}(w,b)= (K^{-1})_{wb}$
\end{theorem}

The above theorem means that by knowing the inverse of the Kasteleyn matrix, which we call the \emph{inverse Kasteleyn matrix}, we can derive the correlation kernel of the dominos. We can now introduce the Kasteleyn matrix of the double Aztec diamond. 
 
  Let 
\begin{equation}
	\mathtt{W} = \{ (x_1,x_2): x_1 \in 2\mathbb{Z} +1, x_2 \in 2\mathbb{Z} \}
\end{equation}
denote the set of white vertices and let
\begin{equation}
	\mathtt{B} = \{ (x_1,x_2): x_1 \in 2\mathbb{Z} , x_2 \in 2\mathbb{Z} +1 \}
\end{equation}
denote the set of black vertices.  The dual graph of the double Aztec diamond written in $(\xi,\eta)$ co-ordinates has white vertices given by
\begin{equation}
	\mathtt{W}_{AD}=\left\{ (x_1,x_2) \in \mathtt{W} : \begin{array}{l}1 \leq 
	                                   x_1 \leq 2(2m+n)+1, 0 \leq x_2 \leq 2(n-1) \\
	                                   \mbox{ or } 1 \leq x_1 \leq 2n-1, x_2=2n
	                                  \end{array}
 \right\}
	\end{equation}
and has black vertices given by
\begin{equation}
	\mathtt{B}_{AD}=\left\{ (x_1,x_2)\in \mathtt{B} : 
	\begin{array}{l}
	0 \leq x_1 \leq 2(2m+n) , 1\leq x_2 \leq 2n-1 \\ \mbox{ or } 2(2m+1) \leq x_1 \leq 2(2m+n), x_2=-1 \end{array} \right\}.
\end{equation}
We denote by $A_{n,m}$ to be vertex set of the dual graph of the double Aztec diamond with $(\xi,\eta)$ co-ordinates.  Figure~\ref{Figdualgraph} shows the dual graph of the double Aztec diamond with $n=8$ and $m=4$.

\begin{figure}
\begin{center}
\includegraphics[height=4in]{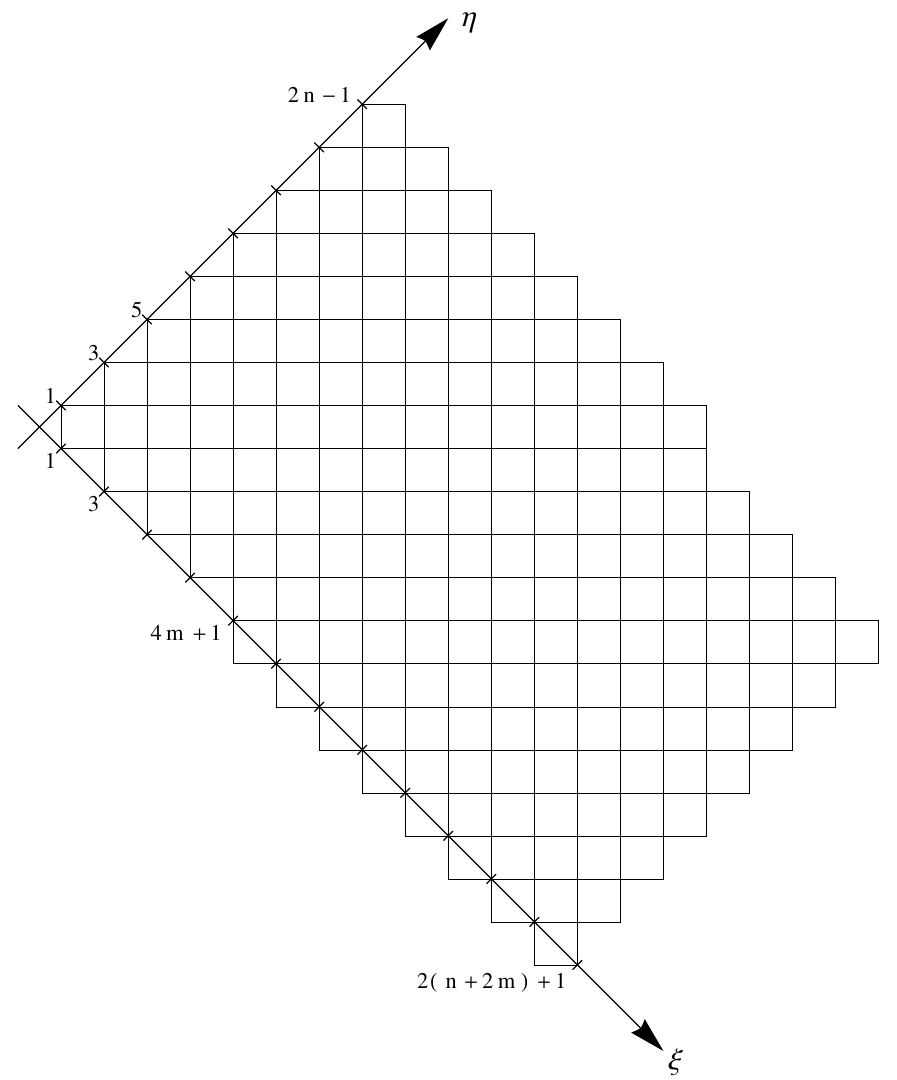}
\includegraphics[height=1.5in]{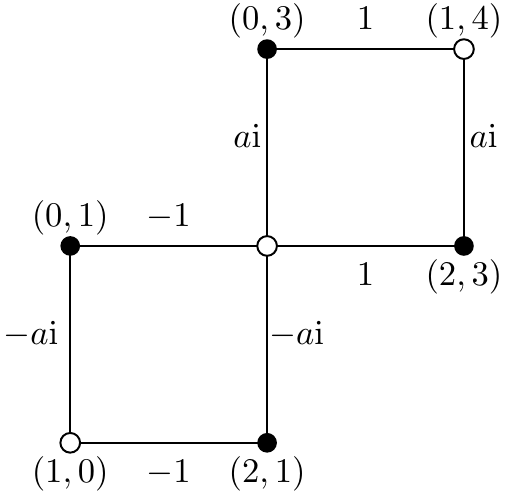}
\caption{The left hand figure shows the dual graph of the double Aztec diamond with $n=8$ and $m=4$ in the $(\xi,\eta)$ co-ordinates.  The right hand figure shows the weights, Kasteleyn orientation, black and white vertices for the two most left squares, with $(\xi,\eta)$-coordinates.}
\label{Figdualgraph}
\end{center}
\end{figure}

Let $K_a$ denote the Kasteleyn matrix for the double Aztec diamond with entries 
\begin{equation} \label{inverseK:defn_of_K}
      K_a (\mathtt{b},\mathtt{w}) = \left\{ \begin{array}{ll}
                                          (-1)^{-(b_1+b_2+1)/2} \alpha(r) & \mbox{if } \mathtt{w}=\mathtt{b}+e_r \\
                                          (-1)^{-(b_1+b_2-1)/2} \alpha(r) & \mbox{if } \mathtt{w}=\mathtt{b}-e_r \\
                                          0 & \mbox{otherwise}
                                         \end{array} \right.
\end{equation}
with $e_1=(1,1)$, $e_2=(-1,1)$, $\mathtt{b}=(b_1,b_2) \in \mathtt{B}$, $r=1,2$, $\alpha(1)=1$ and $\alpha(2)=-a \mathrm{i}$.  The choice of sign for the entries of the matrix is the same as \cite{CJY:12}.  These are chosen so that entries of the inverse Kasteleyn matrix are discrete analytic functions when $a=1$.

\begin{theorem} \label{thm:inverseK}
      The entries of inverse Kasteleyn matrix for the double Aztec diamond, $K_a$, defined by~\eqref{inverseK:defn_of_K} are given by 
\begin{equation}\label{InvKastFormula}
\begin{split}
      &K_a^{-1} (\mathtt{w},\mathtt{b})\\
& =-(-1)^{(w_1-w_2+b_1-b_2+2)/4}\mathbb{K}_{n,\rho} \left(b_2+1,\frac{b_2-b_1+2m+1}{2};w_2+1,\frac{w_2-w_1+2m+1}{2} \right),
\end{split}
\end{equation}
where as before $n-\rho=2m$.
\end{theorem}

In other words, using~\eqref{K}, $(x,z) \leftrightarrow (\xi,\eta)$, 
\begin{equation*}
K_a^{-1} (\mathtt{w},\mathtt{b})= -(-1)^{(w_1-w_2+b_1-b_2+2)/4}\mathbb{K}_{n,\rho}\left( z(\mathtt{b}),x(\mathtt{b});z(\mathtt{w}),x(\mathtt{b}) \right)
\end{equation*}
so essentially the two kernels are the transpose of each other modulo the co-ordinate transformation~\eqref{K}.
 
 The proof of this theorem involves the characterization of $K_a^{-1}$:  $K_a.K_a^{-1}=\mathbbm{I}$ and is given in Section~\ref{Sec:inverseK}.  In order to find such a formula for $K_a^{-1}$ we used a guess following the approach in \cite{CJY:12}. More explicitly, using the $\mathbb{K}$ particle correlation kernel one can compute the joint probabilities of 
$\BK$-particles.  As these particles correspond to east and south dominos, this joint probability should be equal to a corresponding formula written in terms of the inverse Kasteleyn matrix by using Theorem~\ref{Kenyonsthm}.  These two sides can be compared which gives a  guess for the inverse Kasteleyn matrix in terms of the $\BK$ particle correlation kernel and leads to the formula in the theorem.
 
Since we now have the inverse Kasteleyn matrix we can use Theorem \ref{Kenyonsthm} to prove theorem \ref{main1'}. The basic observation is that we have an $\BL$-particle
at a black vertex $b$ if and only if a dimer covers the edge $(b,b+e_1)$ or the edge $(b,b-e_2)$. By using theorem \ref{thm:inverseK} we can compute the probability of
seeing $\BL$-particles at given black vertices $b_1,\dots,b_{\ell}$ by summing over all the possibilities for the dimers and using theorem \ref{Kenyonsthm} and thus deduce the $\mathbb{L}$ kernel, see
section \ref{Sec:Lkernel} for the details.  

For general weights and boundary conditions of the square grid, if we define a point process on the  black vertices such that a particle is present at a black vertex iff a dimer is incident to that black vertex, from Theorem \ref{Kenyonsthm}, we can recover the particle correlation kernel  provided we know the inverse Kasteleyn matrix of the model. In general,  the reverse, i.e. to go from the particle correlation kernel to the inverse Kasteleyn matrix, is quite complicated. 
However, for the double Aztec diamond, we were able to express the inverse Kasteleyn matrix in terms of the $\mathbb{K}$-kernel. There should also be an analogous formula for the inverse Kasteleyn matrix in terms of the $\mathbb{L}$-kernel.  This formula could then be used to give the $\mathbb{K}$-kernel in terms of the $\mathbb{L}$-kernel.  
%\newpage

\section{The tacnode GUE-Minor kernel and its symmetry} \label{Sec:GUEminorproperties}

 Recall from (\ref{2min0}) the tacnode GUE-Minor kernel, about which we show the following.

 \begin{proposition}
The kernel $\BK^{\mbox{\tiny tac}} _{\beta,\rho} (u_1,y_1 ; u_2,y_2)$ is invariant under the involution
 \be
u_1 \leftrightarrow \rho-u_2 \qquad , \qquad y_1 \leftrightarrow -y_2,
\label{inv}\ee
and is a finite rank perturbation of the GUE-minor kernel, as follows:
\be
\begin{aligned}
&\BK^{\mbox{\tiny tac}}_{\beta, \rho}%^  {\mbox{\tiny two-minor}}
 (u_1,y_1;u_2,y_2)
%\\&=  \BK^ {\mbox{\tiny minor}}  (u_1,\beta-y_1 ;~u_2 ,\beta-y_2 ) \\
%&~~~+2\Bigl\la (\Id - {\cal K}^\beta ( \lambda ,\kappa ))^{-1}_{-\rho} ~{\cal A}^{\beta,y_1-\beta}_{u_1 }(\kappa), {\cal B}^{\beta,y_2-\beta}_{u_2 }(\lambda)\Bigr\ra
 % _{_{ -\rho }}
 \\
 &\stackrel{(*)}{=}  \BK^ {\mbox{\tiny minor}}  (u_1,\beta-y_1 ;~u_2 ,\beta-y_2 ) \\
&~~~+2\sum_{\lambda=0}^{\max(\rho-1,\rho-1-u_2)}  \bigl((\Id - {\cal K}^\beta ( \lambda-\rho ,\kappa-\rho ))^{-1}  {\cal A}^{\beta,y_1-\beta}_{u_1 }\bigr)(\lambda-\rho)  {\cal B}^{\beta,y_2-\beta}_{u_2 }(\lambda-\rho) 
  _{_{ }}
 \\
 &\stackrel{(**)}{=}  \BK^ {\mbox{\tiny minor}}  (\rho-u_2,\beta+y_2 ;~\rho-u_1 ,\beta+y_1 ) \\
&~~~+2\!\!\sum_{\lambda=0}^{\max(\rho-1,u_1-1)}  \bigl((\Id - {\cal K}^\beta ( \lambda-\rho ,\kappa-\rho ))^{-1}  {\cal A}^{\beta,-y_2-\beta}_{\rho-u_2 }\bigr)(\lambda-\rho)  
%\\&~~~~~~~~~~~~~~~~~~\times
  {\cal B}^{\beta,-y_1-\beta}_{\rho-u_1 }(\lambda-\rho) 
  _{_{ }}
   \label{2min}\end{aligned}\ee

\end{proposition}

\remark This symmetry (\ref{inv}) is not surprising, since it corresponds to the symmetry of the geometry of the double Aztec diamond.

\
 % \remark Notice, for $v\leq 0$, the double integral in $ {\cal A}^{\beta,y }_{v }(\kappa)$ equals $0$ . Also, note that the double integral of the GUE-minor part of representation $*$ of (\ref{2min}) equals $0$ for $u_2\geq 0$; also the double integral of the GUE-minor part of representation $**$ of (\ref{2min}) equals $0$ for $u_1\leq \rho $.

%\newpage

\proof In order to prove this statement  we need the following functions,$$\begin{aligned}
G(\lb) &=  \int_{L} \frac{d\om}{2\pi i} e^{2\om^2-4\beta\om} \om^{-\lb-2}, 
~~~~~g_{y } (k) = \int_{L} \frac{d\om}{2\pi i} e^{\om^2-2(\beta-y )\om} \om^{-k-1}
\\
H(\kappa) &=  \int_{\Gamma_0} \frac{d\zeta}{2\pi i \zeta^{-\kappa}} e^{-2\zeta^2+4\beta\zeta}, \qquad    h_{y } {(\lb)}=  \int_{\Gamma_0} \frac{d\zeta}{2\pi i \zeta^{-\lb}} e^{- \zeta^2+2(\beta-y )\zeta},
\end{aligned}$$
and the corresponding operators
$$\begin{aligned}
\mathcal{G} (\kappa,\al) f(\al) &:= \sum_{\al \geq 0} G (\kappa-\rho+\al) f(\al)
\\
\mathcal{H} (\lb,\al) f(\al) & :=  \sum_{\al\geq 0} H (\lb-\rho+\al) f(\al)
.\end{aligned}$$
The kernel ${\mathcal K}^{\beta}$ and its resolvent, and the functions ${\mathcal A}^{\beta,y}_{v}(\kappa)$ and ${\mathcal B}^{\beta,y}_{u}(\lb)$ as in (\ref{defAB}) can then be expressed as follows, 
\be
\begin{aligned}
&    {\cal K}^{\beta} (\lb,\kappa) = \sum_{\al \geq 0} %\qquad 
 G(\lb+\al) \ H(\al+\kappa),~~~~{\cal K}_\rho (\lb,\kappa) := {\cal K}^{\beta} (\lb-\rho,\kappa-\rho),
\\
&  \mathcal{K}_\rho = \mathcal{G} \mathcal{H} \qquad  , \qquad 
{\mathcal K}^\top_\rho = \mathcal{H}\mathcal{G}, \qquad   \mbox{with} ~~\mathcal{G}^\top = \mathcal{G} \ , \ \mathcal{H}^\top = \mathcal{H},
\\ &  \Id + \mathcal{R} (\lb,\kappa) := (\Id - \mathcal{K}_\rho (\lb,\kappa))^{-1} = \sum_{\al \geq 0} \mathcal{K}^\al_\rho
\\
%&    {\cal K}_\rho (\lb,\kappa) := {\cal K} (\lb-\rho,\kappa-\rho)
%\\
&    \mathcal{A}_{u_1}^{\beta,y_1-\beta} (\kappa-\rho) = g_{-y_1} (\kappa-\rho+u_1) - \sum_{\al \geq 0} G(\kappa-\rho+\alpha) h_{y_1} (\alpha-u_1)
\\
&   \hspace{28mm}     = g_{-y_1} (\kappa-\rho+u_1)  - \mathcal{G} (\kappa,\cdot) h_{y_1} (\cdot -u_1)
\\
&  \mathcal{B}^{\beta,y_2-\beta}_{u_2} (\lb-\rho) = h_{-y_2} ( \lb -\rho +u_2) -\sum_{\al \geq 0}  H(\lambda-\rho+\alpha) g_{y_2} (\alpha-u_2)
\\
&   \hspace{28mm}   =  h_{-y_2} (\lb -\rho +u_2) - \mathcal{H} (\lb,\cdot) g_{y_2} (\cdot-u_2)
\end{aligned}\label{defog}\ee
Using the definition (\ref{2min0}) of the kernel and the expressions (\ref{defog}), we have the following identities:
\be\begin{aligned}
&                 \frac{1}{2} \  \BK^{{\mbox{\tiny tac}}}_{\beta,\rho} (u_1,\beta-y_1 ; u_2, \beta-y_2) 
\\
& 
                  = -\Id_{u_1>u_2} 2^{u_1-u_2-1} \BH^{u_1-u_2} (y_2-y_1) + \sum_{\al \geq 0} g_{y_2}(\al-u_2) h_{y_1}(\al-u_1)
\\
& 
                  + \Bigl    \la  (\Id + {\cal R}(\lb,\kappa))g_{-y_1} (\kappa-\rho+u_1) \ , \ h_{-y_2} {(\lb - \rho + u_2)}   \Bigr     \ra_{\geq 0}
\\
& 
                  +   \Bigl   \la \mathcal{H} (\lb,\al)g_{y_2} (\al-u_2) \ , \ (\Id+{\cal R}(\lb,\kappa)) \mathcal{G} (\kappa,\al) h_{y_1} (\al-u_1)    \Bigr   \ra_{\geq 0}
\\
& 
                 -  \Bigl  \la (\Id + {\cal R}(\lb,\kappa)) g_{-y_1} (\kappa-\rho+u_1) \ , \ \mathcal{H} (\lb,\al) g_{y_2} (\al-u_2)   \Bigr     \ra_{\geq 0)}
\\
& 
                 -   \Bigl    \la (\Id + {\cal R}(\lb,\kappa)) \mathcal{G} (\kappa,\al) h_{y_1} (\al-u_1) \ , \ h_{-y_2} (\lb-\rho+u_2)   \Bigr \ra_{\geq 0}
\\
& 
                 = - \Id_{u_1 > u_2} 2^{u_1-u_2-1} \BH^{u_1-u_2} (y_2-y_1)
\\
& 
                + \Bigl   \la g_{-y_1} (\kappa-\rho+u_1) \ , \ (\Id + {\cal R}^{\top} (\lb,\kappa)) h_{-y_2} (\kappa-\rho+u_2)     \Bigr \ra_{\geq 0}
\\
&                 + \Bigl   \la g_{y_2} (\al-u_2) \ , \ (\Id +  \mathcal{H}^T (\Id+{\cal R})\mathcal{G}) h_{y_1} (\al -u_1)     \Bigr \ra_{\geq 0}
\\
&                - \Bigl   \la  \mathcal{H} (\Id+{\cal R}) g_{-y_1} (\kappa-\rho+u_1) \ , \   g_{ y_2}  (\al-u_2)   \Bigr\ra_{\geq 0}
\\
&               - \Bigl   \la (\Id +{\cal R}) \mathcal{G} h_{y_1} ((\al-u_1) \ , \  h_{-y_2}  (\lb - \rho + u_2)    \Bigr \ra_{\geq 0}.
\end{aligned}\label{symm}\ee
Given  the involution (\ref{inv}), 
%\[
%u_1 \leftrightarrow \rho -u_2 \quad , \quad y_1 \leftrightarrow -y_2,
%\]
all terms in the last expression (\ref{symm}) are self-dual, except that the second and third terms interchange, because of the operator identity (see (\ref{defog}))
\[
 \Id+\mathcal{H}^{\top}  (\Id +{\cal  R}) \mathcal{G}=\Id+\mathcal{H}^{ }  (\Id +{\cal  R}) \mathcal{G}=\Id+{\cal  R}^{\top} 
\]
and the self-adjointness of $\mathcal{H}(\Id+\mathcal{R})$ and $(\Id +{\cal  R})\mathcal{G}$.

To prove the second statement (\ref{2min}) on finite perturbation, one notices from~\eqref{defAB} that the double integral in $ {\cal B}^{\beta,y }_{u }(\lambda)$ equals $0$ for $\lambda\geq 0$, since the integrand as a function of $\zeta $ has no pole at $0$ and, similarly the single integral equals $0$ for $u+\lambda\geq 0$. Thus  $ {\cal B}^{\beta,y_2-\beta}_{u_2 }(\lambda-\rho) =0$ for $\lambda\geq \rho$ and for $\lambda\geq \rho-u_2$. This proves from~\eqref{2min0}, the first equality $\stackrel{(*)}{=} $. The second equality $\stackrel{(**)}{=} $ is obtained by the involution (\ref{inv}). \qed

%\newpage

\section{Interlacing pattern of the $\BL$-process} \label{Sec:Interlacing}

In this section we prove the interlacing properties as explained in Proposition \ref{Interla}. We first need the following Lemma (remember $\rho=n-2m$).

 \begin{lemma}\label{dots1}
 The total number of dots along the line $\xi=2i$ equals the difference of height $\Dt h$ between the extreme points of that line, as is given by the boundary values of the height function: \footnote{a dot-particle $x$ is to the right of a dot-particle $y$ means $\eta(x)\geq \eta(y)$. }
 $$\begin{array}{llllllllll}
 
 \mbox{lines $\xi=2i$}&\vline   &\mbox{$h_{\mbox{\tiny left}}$}  &\vline&\mbox{$h_{\mbox{\tiny right}}$} &\vline &\Dt h  & \vline& \mbox{\# of blue and red dots} \\
 \hline
 0\leq  i\leq n-\rho &\vline  &n  & \vline&  i & \vline&n-i&\vline& \mbox{$n-i$ blue dots  }\\
 n-\rho <  i <  n &\vline&\rho+i&\vline& i&\vline&   \rho &\vline& \left\{\begin{aligned} &\mbox{$\rho+i-n $ red dots}\\ 
 &\mbox{to the left of}\\&\mbox{$\ n-i$ blue dots}   \end{aligned}\right.  \\
  n \leq  i\leq  2n-\rho &\vline&\rho+i&\vline&n&\vline&   \rho\!+\!i\!-\!n &\vline& \mbox{$\rho+i-n$ red dots  }\\
 \end{array}
 $$
 
 \end{lemma}
 \proof The statement on the first row in the table above follows from the fact that the height $h$ along the lines $\xi=2i$ for $0\leq i\leq n-\rho$ decreases from $h_{\mbox{\tiny left}}=n$ to $h_{\mbox{\tiny right}}=i$ for $0\leq i\leq n-\rho$ (going from left to right) and from the fact that each decrease of height by $1$ produces a dot. The same statement holds for the range on the third line of the table by the obvious symmetry consisting of flipping the figure about the middle of the $\xi$-axis and the middle of the $\eta$-axis. Also note that the heights of the $B$-level curves range over the half-integers from $n+1/2$ to $2n-1/2$. Therefore the lines $\xi=2i$ for $0\leq i\leq n-\rho $, which have height at most $n$, will never intersect those lower-level lines and vice-versa, showing that on the first line (resp. last line) of the table above only blue (red resp.) dots appear.
 
 In the overlap region of the two diamonds, the boundary values of the height function show that $h_{\mbox{\tiny left}}$, $h_{\mbox{\tiny right}}$ and $\Dt h=h_{\mbox{\tiny left}}-h_{\mbox{\tiny right}}$ is as indicated in the table. Moreover, since the height of the $B$-level curves is $\geq n+1/2$ and the height of the $A$-level curves is $\leq n-1/2$, the red dots are all to the left of the blue dots along the lines $\xi=2(n-\rho)$ up to $2n$, with numbers as indicated in the table.

   \begin{lemma}\label{dots2}
 Let the lines $\xi=2k$ and $\xi=2k+2$ for $0\leq k\leq  n-1$ have $\ell-1$ dots starting from the right boundary. Then the $\ell ^{th}$ dot on the line $\xi=2k+2$ must be to the left of or coincide with the $\ell ^{th}$ dot on the line $\xi=2k$.
\end{lemma}

\proof Note that the right most point of the double Aztec diamond on the line $\xi=2i$ has height $i$, provided $0\leq i\leq  n$. Therefore, if there are $\ell$ dots on the line $\xi=2k+2$, counting from the right hand boundary, then the left-lower vertex of the square $A'$, containing the $\ell$th dot, has height $k+\ell+1$; see Figure~\ref{FigD}. Consider two cases:

(i) assuming no dot in the corresponding square $A$ on the line $\xi=2k$, then the only way to cover the squares $A$ and $A'$ with domino's such that $A'$ carries a dot and not $A$, is given by the four upper configurations of Figure~\ref{FigC}; putting in the heights forces the height of the left-lower and right-upper vertices of the square $A$ to be $k+\ell$ as indicated in Figure~\ref{FigD}. This shows there must be $\ell$ dots on the line $\xi=2k$ strictly to the right of the square $A$. So, the $\ell ^{th}$ dot on the line $\xi=2k+2$ must strictly be to the left of the $\ell ^{th}$ dot on the line $\xi=2k$, at least if $A$ contains no dot. %One then uses induction on $\ell$ to show the validity of the statement above.

(ii) Assume a dot in $A$ on the line $\xi=2k$; the only way for this to occur is given by the four lower configurations of Figure~\ref{FigC}. From them one deduces that, if the height of the lower-left corner of $A'$ is $k+\ell+1$, then the height of the lower-left corner of $A$ must be $k+\ell$ or $k+\ell+1$. In the former case (i.e., $k+\ell$), the dots in $A$ and $A'$ are the $\ell$th ones from the right, proving the claim; in the latter case (i.e., $k+\ell +1$), the dot in $A$ is the $\ell+1$st one and the dot in $A'$ the $\ell$th one. So, the $\ell$th dot on the line $\xi=2k$ is to the right of the $\ell$th one on the line $\xi=2k+2$.\qed

\begin{figure}
\begin{center}
\includegraphics[height=2.5in]{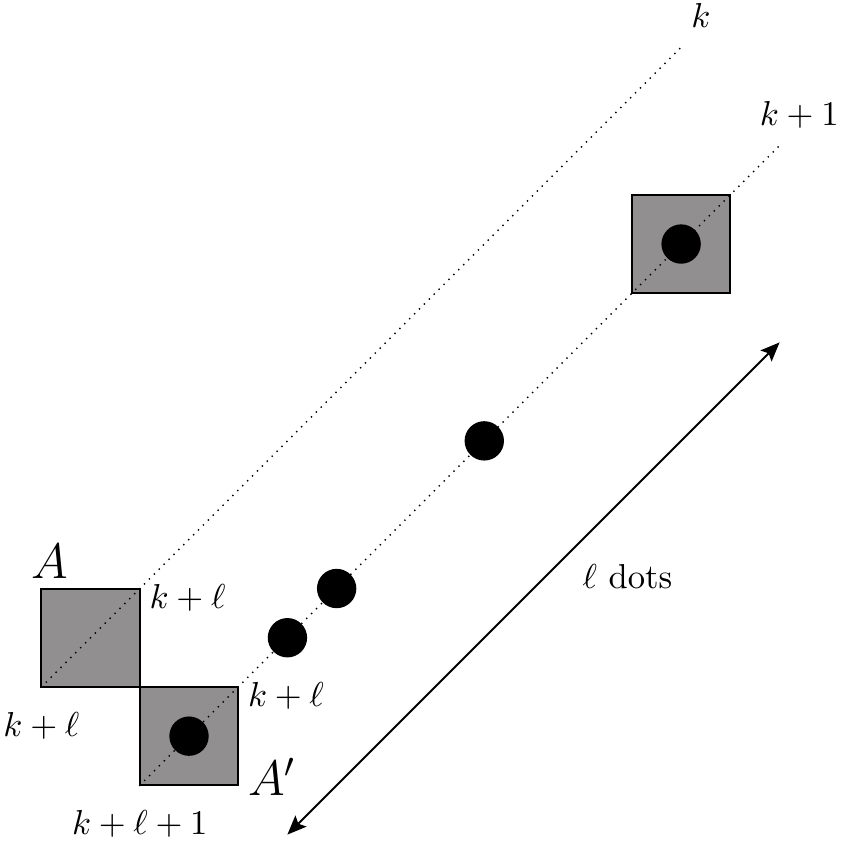}
\caption{Assume the $\ell$th dot on the line $\xi=2k+2$ (counted from the right) appears in $A'$, and assume no dot in the square $A$, then the height of $A$ must be as indicated. }
\label{FigD}
\end{center}
\end{figure}

\begin{figure}
\begin{center}
\includegraphics[height=2in]{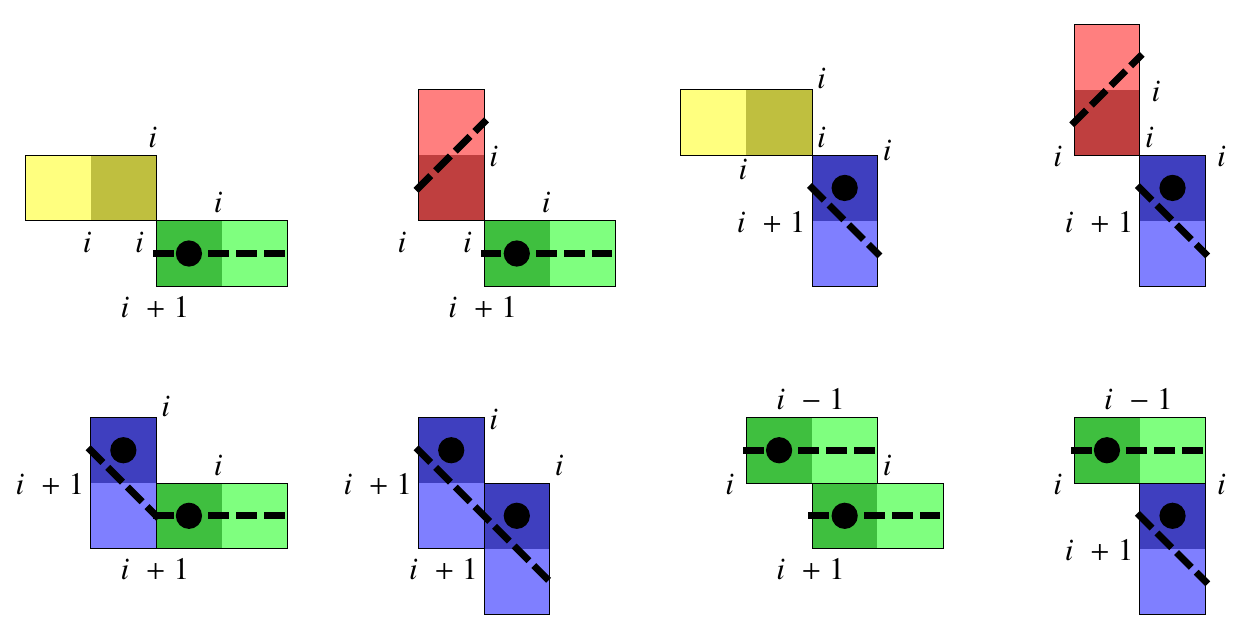}
\caption{The four upper configurations are the only coverings of $A$ and $A'$ of Figure 11, with $A'$ carrying a dot and not $A$. The four lower configurations are the only coverings of $A$ and $A'$, with both $A$ and $A'$ carrying a dot.}
\label{FigC}
\end{center}
\end{figure}

\begin{figure}
\begin{center}
\includegraphics[height=2.5in]{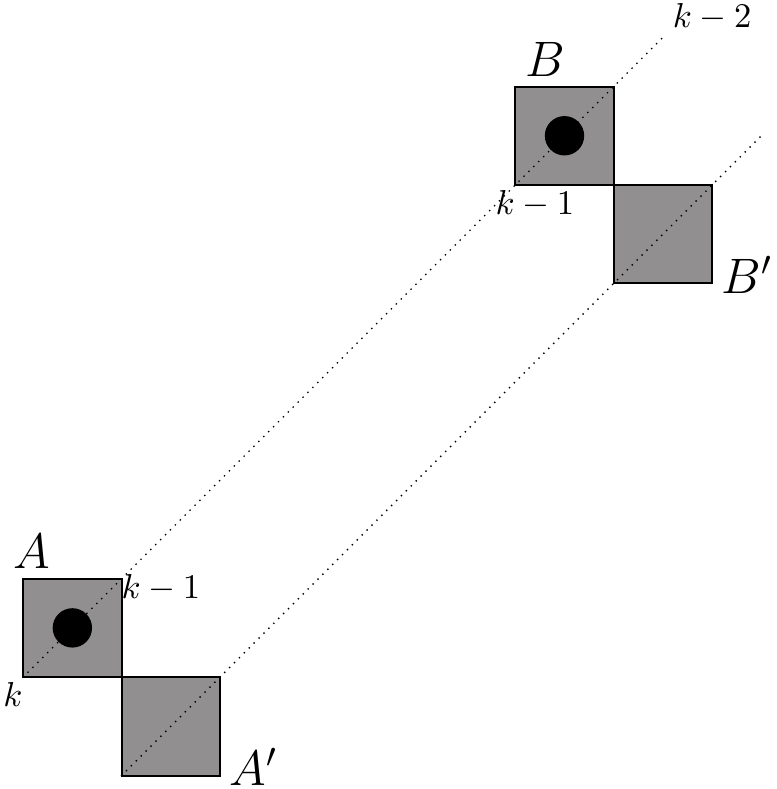}
\caption{Between the two gray squares labeled $A$ and $B$ the height function stays constant; therefore the line between $A$ and $B$ contains no dots.}
\label{FigA}
\end{center}
\end{figure}

\begin{figure}
\begin{center}
\includegraphics[height=1.75in]{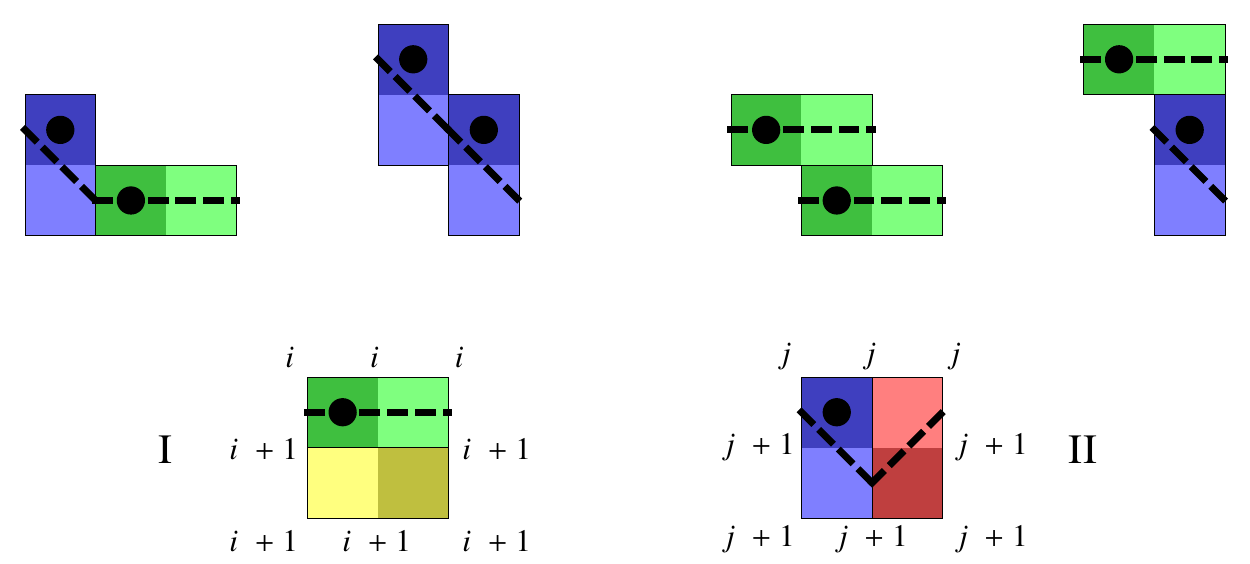}
\caption{Let the squares $A$ and $A'$ (as in Figure \ref{FigA}) each contain a dot, then the four upper figures are the only possible covers of $(A,A')$. If $A$ contains a dot and $A'$ does not, then the two lower figures are the only possible coverings.
% and The six configurations for $(A,A')$ and $(B,B')$ with a dot in the upper left hand square. We include the level lines to show the height differences. The two configurations I and II have no dots in the lower right hand corner.
} 
\label{FigB}
\end{center}
\end{figure}

  \noindent {\em Proof of Proposition \ref{Interla}}: Consider two consecutive lines $\xi=2\alpha$ and $\xi=2\alpha +2$ through blue dots, with the squares $A$ and $B$,  containing each a dot and no dot in between $A$ and $B$; see Figure~\ref{FigA}. This is to say, the level of the line $\xi=2\alpha$ goes down from $k$ to $k-1$ within square $A$, stays flat in between $A$ and $B$ and then goes down from $k-1$ to $k-2$ within square $B$. We now consider the line $\xi=2\alpha +2$ between the two corresponding squares $A'$ and $B'$, with same $\eta$ coordinates as $A$ and $B$ respectively.

 \noindent We show there must be at least one dot in between the squares $A'$ and $B'$, possibly including $A'$ or $B'$. One checks there are exactly six configurations with a dot in the upper-left square; see Figure~\ref{FigB}. Superimposing any of the four upper configurations on $(A,A')$ or $(B,B')$ will give a dot in $A'$ or $B'$. Assuming no dot, neither at $A'$, nor at $B'$, the configuration $(A,A')$ or $(B,B')$ at Figure~\ref{FigA} can be covered by any combination of configurations ${\rm (I)}$ and ${\rm (II)}$ in Figure~\ref{FigB}. Indeed, 
  $$\begin{array}{cccc}
 (A,A') &&&  (B,B')\\
 ${\rm I}(i=k-1)$ &&&${\rm I}(i=k-2)$ \\
 ${\rm I}(i=k-1)$ &&&${\rm II}(j=k-2)$ \\
 ${\rm II}(j=k-1)$ &&&${\rm I}(i=k-2)$ \\
 ${\rm II}(j=k-1)$ &&&${\rm II}(j=k-2)$ \\
 \end{array}$$
 In all four cases, the difference in height between the lower-left vertex of $A'$, having height $k$, and the upper-right vertex of $B'$, having height $k-1$, will always  $=1$, thus creating a jump in between and thus one dot. So in all cases, there will be at least one dot in one of the squares on the segment $(A',B')$, including possibly on the extremities. 
 
 Finally, this fact together with Lemma \ref{dots1} on the number of blue and red dots and Lemma \ref{dots2} imply the interlacing, with regard to the $\eta$-coordinate.\qed

% \vspace*{.5cm}

 %\newpage

\section{Scaling limit of the $\BL$ and $\BK$ -processes}

In this section we will prove theorem \ref{a=1} and theorem \ref{Ka=1}. Let $\delta\in\{0,1\}$, where $\delta=0$ will correspond to the
$\BL$-kernel and $\delta=1$ to the $\BK$-kernel. We will ignore the integer parts in the scaling (\ref{scK}) and (\ref{sc1}). This makes no
essential difference but simplifies the notation. Set
$$
f_{t,\delta}(u_1,y_1;~u_2,y_2)=(-a)^{(y_2-y_1)\sqrt{t}}(-\sqrt{t})^{u_2-u_1}(\sqrt{t})^{1-2\delta}(-1)^{\delta}.
$$
Then the prefactor in (\ref{limL}) for the $\BL$-kernel can be written
$$
(-a)^{(\eta_1-\eta_2)/2}(-\sqrt{t})^{(\xi_1-\xi_2)/2}\sqrt{t}=a^{2(y_1-y_2)\sqrt{t}}f_{t,0}(u_1,y_1;u_2,y_2)
$$
and the prefactor for the $\BK$-kernel in (\ref{limK}) is
\begin{align}
&a^{r_2-r_1}(\sqrt{t})^{x_1-x_2+r_2-r_1}(-1)^{x_1-x_2}\\
&=a^{2(y_2-y_1)\sqrt{t}}(-a)^{(y_1-y_2)\sqrt{t}}(-\sqrt{t})^{u_1-(u_2+1)}(-\sqrt{t})\\
&= a^{2(y_2-y_1)\sqrt{t}}f_{t,1}(u_2+1,y_2;~u_1,y_1).
\end{align}
Define
\begin{align}\label{C1}
&C^{(1)}_{2t+\epsilon,\rho,\delta}(u_1,y_1;u_2,y_2)=-f_{t,\delta}(u_1,y_1;~u_2,y_2)(1+a^2)\left((1-\delta) \Id_{u_2<u_1}+\delta \Id_{y_1<y_2} \right)\\
&\times\frac 1{2\pi \I}\oint_{\Gamma_{0,a}}\frac{(1+aw)^{(y_1-y_2)\sqrt{t}-1+\delta}}
{(w-a)^{(y_1-y_2)\sqrt{t}+1-\delta}}w^{u_2-u_1}~dw,
\end{align}
\begin{align}\label{C2}
&C^{(2)}_{2t+\epsilon,\rho,\delta}(u_1,y_1;u_2,y_2)
=f_{t,\delta}(u_1,y_1;u_2,y_2)\\
&\times\frac{(1+a^2)}{(2\pi \I)^2}\oint_{\Gamma_{0,a}}dz\oint_{\Gamma_{0,a,z}}\frac{dw}{w-z}\frac{w^{u_2}(1+az)^{t+y_1\sqrt{t}-1+\delta}(z-a)^{t-y_1\sqrt{t}+\epsilon+\delta}}
{z^{u_1}(1+aw)^{t+y_2\sqrt{t}}(w-a)^{t-y_2\sqrt{t}+1+\epsilon}}
\end{align}
and
\begin{align}\label{C3}
&C^{(3)}_{2t+\epsilon,\rho,\delta}(u_1,y_1;u_2,y_2)\\
&=-f_{t,\delta}(u_1,y_1;u_2,y_2)
(1+a^2)\\
&~~~\times\Big\langle ((I-K_{2t+\epsilon})^{-1}_{\ge 2t+\epsilon-\rho+1}A_{4t+2\epsilon-2u_1,2t+2y_1\sqrt{t}-1,\delta})(j)
\\&\hspace*{6cm}~,~B_{4t+2\epsilon-2u_2,2t+2y_2\sqrt{t}-1}(j)\Big\rangle_{\ge 2t+\epsilon-\rho+1},
\end{align}
where
\begin{align}\label{Adelta}
&A_{4t+2\epsilon-2u_1,2t+2y_1\sqrt{t}-1,\delta}(k)\\
&=-\frac{(-1)^k}{2\pi \I}\oint_{\Gamma_{0,a}}\frac{w^{2t+\epsilon-u_1-k}}{(1+aw)^{t+\epsilon-y_1\sqrt{t}+1-\delta}(w-a)^{t+y_1\sqrt{t}+1-\delta}}~dw\\
&+\frac{(-1)^k}{(2\pi \I)^2}\oint_{\Gamma_{0,a}}dz\oint_{\Gamma_{0,a,z}}\frac{dw}{w-z}\frac{w^{2t+\epsilon-k}(1+az)^{t+y_1\sqrt{t}-1+\delta}(z-a)^{t-y_1\sqrt{t}+\epsilon+\delta}}
{z^{u_1}(1+aw)^{2t+\epsilon}(w-a)^{2t+\epsilon+1}},
\end{align}
which is a slight modification of $A_{\xi_1,\eta_1}(k)$ in (\ref{oneAzt}), and where $B_{\xi_2,\eta_2}(k)$ and $K_{2t+\epsilon}$ are as given in (\ref{oneAzt}).
With these definitions it follows from (\ref{Lkernel1}) that
\begin{align}\label{LC}
&(-a)^{(\eta_1-\eta_2)/2}(-\sqrt{t})^{(\xi_1-\xi_2)/2}\BL_{2t+\epsilon,\rho}(\xi_1,\eta_1;\xi_2,\eta_2)\sqrt{t}\\
&=a^{2(y_1-y_2)\sqrt{t}}(C^{(1)}_{2t+\epsilon,\rho,0}+C^{(2)}_{2t+\epsilon,\rho,0}+C^{(3)}_{2t+\epsilon,\rho,0})(u_1,y_1;u_2,y_2),
\end{align}
if we have the scaling (\ref{scK}). Similarly, it follows from (\ref{K1}) that
\begin{align}\label{KC}
&(1-p_n)a^{r_2-r_1}(\sqrt{t})^{x_1-x_2+r_2-r_1}(-1)^{x_1-x_2}\BK_{2t+\epsilon,\rho}(2r_1,x_1;2r_2,x_2)\\
&=\frac 2{1+a^2}a^{2(y_2-y_1)\sqrt{t}}\left(\left(C^{(1)}_{2t+\epsilon,\rho,1} +C^{(2)}_{2t+\epsilon,\rho,1}+C^{(3)}_{2t+\epsilon,\rho,1}\right)(u_2+1,y_2;u_1,y_1)\right).
\end{align}
We will now use (\ref{LC}) to prove (\ref{limL}). The proof of (\ref{limK}) from (\ref{KC}) is completely analogous since the change from
$\delta=0$ to $\delta=1$ has no effect in the limit. Note that $a^{2(y_1-y_2)\sqrt{t}}\to e^{2\beta(y_2-y_1)}$ as $t\to\infty$ since $a=1-\beta/\sqrt{t}$.
We see that (\ref{limL}) follows from
\begin{equation}\label{limC} 
\lim_{t\to\infty}\sum_{i=1}^3 C^{(i)}_{2t+\epsilon,\rho,0}(u_1,y_1;~u_2,y_2)=e^{2\beta(y_1-y_2)}\BK_{\beta,\rho}(u_1,y_1;~u_2,y_2).
\end{equation}

Let $\mathcal{C}_1$ be the positively oriented unit circle and let $\mathcal{C}_2=\mathcal{C}_2'+\mathcal{C}_2''$, where $\mathcal{C}_2'$ consists of two infinite
line segments $t\to\beta+\I t$, $t\in (-\infty,-2]\cup[2,\infty)$, and $\mathcal{C}_2''$ is a smooth curve that goes from $\beta-2\I$ to $\beta+2\I$ to the right of $\mathcal{C}_1$,
see Figure~\ref{Fig:contours}.

\begin{figure}
\begin{center}
\includegraphics[height=3in]{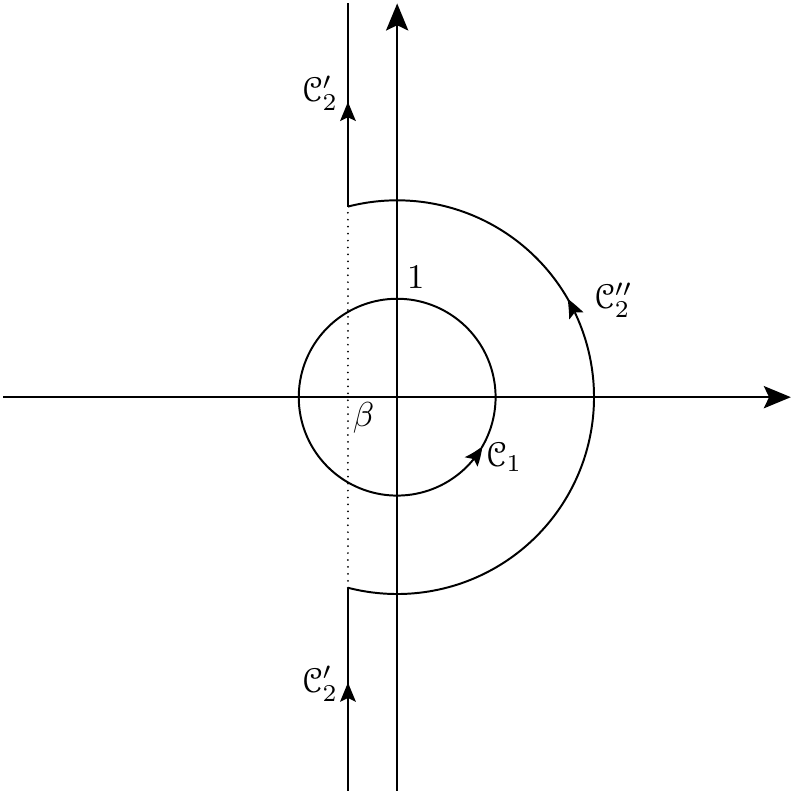}
\caption{The contour paths }
\label{Fig:contours}
\end{center}
\end{figure}

Let $\overline{\mathcal{C}_2}$ be $\mathcal{C}_2$ reflected through the origin. 

Set 
$$
G_{x,t}(\zeta)=\left(\frac{a^{-1}-\zeta/\sqrt{t}}{a+\zeta/\sqrt{t}}\right)^{x\sqrt{t}}
$$
and
$$
F_{x,t}(\zeta)=\left(a^{-1}-\zeta/\sqrt{t}\right)^{t+x\sqrt{t}}\left(a+\zeta/\sqrt{t}\right)^{t-x\sqrt{t}}
$$
so that $F_{x,t}(\zeta)=F_{0,t}(\zeta)G_{x,t}(\zeta)$. Also, we write
$$
g_{x,\beta}(\zeta)=e^{2x(\beta-\zeta)}  ,~~\mbox{and}~f_\beta(\zeta)=e^{2\beta \zeta -\zeta^2}.
$$
The next lemma contains the estimates we need.

\begin{lemma}\label{Est}
Fix $A>0$, $\beta\in\BR$ and $k\ge 1$. There is a $t_0$ and a constant $C$ such that for all $t\ge t_0$, $x\in [-A,A]$,
$s\in\BR$ and $\zeta\in\mathcal{C}_1\cup\mathcal{C}_2''$ we have the following estimates
\begin{equation}\label{Fest}
\frac 1{|F_{x,t}(\beta+\I s)|}\le \frac 1{1+s^{2k}/2^kk!},
\end{equation}
\begin{equation}\label{Gest}
\left|\frac{G_{x,t}(\zeta)}{g_{x,\beta}(\zeta)}-1\right|\le\frac C{\sqrt{t}},
\end{equation}
and
\begin{equation}\label{Fest2}
\left|\frac{F_{0,t}(\zeta)}{f_{\beta}(\zeta)}-1\right|\le\frac C{\sqrt{t}}.
\end{equation}
\end{lemma}

\begin{proof}
We have, since $a=1-\beta/\sqrt{t}$, that
$$
|F_{x,t}(\beta+\I s)|=\left|a^{-1}-\frac{\beta+\I s}{\sqrt{t}}\right|^{t+x\sqrt{t}}\left|a+\frac{\beta+\I s}{\sqrt{t}}\right|^{t-x\sqrt{t} }.
$$
Now,
$$
\left|a+\frac{\beta+\I s}{\sqrt{t}}\right|^2=1+\frac{s^2}t
$$
and
$$
\left|a^{-1}-\frac{\beta+\I s}{\sqrt{t}}\right|^2=\left(\frac 1{1-\beta/\sqrt{t}}-\beta{\sqrt{t}}\right)^2+\frac{s^2}t\ge 1+\frac{s^2}t
$$
when $t$ is large enough. Thus, 
$$
|F_{x,t}(\beta+\I s)|\ge \left(1+\frac{s^2}t\right)^{\frac 12(t+ x\sqrt{t} )}\left(1+\frac{s^2}t\right)^{\frac 12(t- x\sqrt{t}) }=\left(1+\frac{s^2}t\right)^t.
$$
By the binomial theorem,
\begin{align}
\left(1+\frac{s^2}t\right)^t&=1+\sum_{k=1}^t\frac{t\dots(t-k+1)}{t^kk!}s^{2k}\ge 1+\frac{t\dots(t-k+1)}{t^kk!}s^{2k}\\
&\ge 1+\frac{(t/2)^k}{t^kk!}s^{2k}=1+\frac{s^{2k}}{2^kk!}
\end{align}
if $k\le t/2$. This proves (\ref{Fest}). Note that $\mathcal{C}_1\cup\mathcal{C}_2''$ is a fixed compact set. The estimates (\ref{Gest}) and
(\ref{Fest2}) follow from the inequalities
\begin{align}
\left| x\sqrt{t} \log(a+\zeta/\sqrt{t})-x(\zeta-\beta)\right|&\le\frac C{\sqrt{t}},
\end{align}
\begin{align}
\left| x\sqrt{t} \log(a^{-1}-\zeta/\sqrt{t})-x(\beta-\zeta)\right|&\le\frac C{\sqrt{t}},
\end{align}
and
$$
\left|t\log(a^{-1}-\zeta/\sqrt{t})(a+\zeta/\sqrt{t})-(2\zeta \beta -\zeta^2)\right|\le\frac C{\sqrt{t}},
$$
for sufficiently large $t$, which in turn follow from Taylor's theorem.
\end{proof}

Consider first $C^{(1)}_{2t+\epsilon,\rho,0}$. The case $y_1=y_2$ is special. In this case we obtain 
\begin{align}
&C^{(1)}_{2t+\epsilon,\rho,0}(u_1,y_1;u_2,y_2)\\
&=-(-\sqrt{t})^{u_2-u_1}\sqrt{t}\Id_{u_2<u_1}\frac{1+a^2}{2\pi \I}\oint_{\Gamma_{0,a}}\frac{w^{u_2-u_1}}{(1+aw)(w-a)}~dw\\
&=(-\sqrt{t})^{u_2-u_1}\sqrt{t}\Id_{u_2<u_1}\frac{a+1/a}{2\pi \I}\oint_{\Gamma_{-a^{-1}}}\frac{w^{u_2-u_1}}{(w+1/a)(w-a)}~dw\\
&=-(-\sqrt{t})^{u_2-u_1}\sqrt{t}\Id_{u_2<u_1}(-1/a)^{u_2-u_1}.
\end{align}
In the second inequality we deformed the contour through infinity to a contour surrounding $-1/a$.
If $u_2=u_1-1$ this equals $-a$ which goes to $-1$ as $t\to\infty$. If $u_2<u_1-1$ the last expression goes to $0$ as $t\to\infty$.

If $y_1>y_2$ then deforming the contour to $\Gamma_{-1/a}$ shows that $C^{(1)}_{2t+\epsilon,\rho,0}=0$. Assume that $y_1<y_2$. Then, for large enough $t$, 
since $w=a$ is not a pole,
\begin{align}
&C^{(1)}_{2t+\epsilon,\rho,0}(u_1,y_1;u_2,y_2)\\
&=-\Id_{u_2<u_1}\frac{1+a^2}{2\pi\I}\oint_{\Gamma_0}\left(\frac{a^{-1}+w}{a-w}\right)^{(y_1-y_2)\sqrt{t}}\frac{(-w\sqrt{t})^{u_2-u_1}}
{(1+aw)(a-w)}(-\sqrt{t})~dw\\
&=-\Id_{u_2<u_1}\frac{1+a^2}{2\pi\I}\oint_{\mathcal{C}_1} G_{y_1-y_2,t}(\omega)\frac{\omega^{u_2-u_1}}{(1-a\omega/\sqrt{t})(a+\omega/\sqrt{t})}~d\omega
\end{align}
by the change of variables $w=-\omega/\sqrt{t}$.
It now follows from lemma \ref{Est} that
\begin{align}
&\lim_{t\to\infty}C^{(1)}_{2t+\epsilon,\rho,0}(u_1,y_1;u_2,y_2)\\
&=-\Id_{u_2<u_1}\frac{2}{2\pi\I}\oint_{\mathcal{C}_1} e^{2(y_2-y_1)(\omega-\beta)}\frac{d\omega}{\omega^{u_1-u_2}}\\
&=-\Id_{u_2<u_1}e^{2\beta(y_1-y_2)}2^{u_1-u_2}\frac{(y_2-y_1)^{u_1-u_2-1}}{(u_1-u_2-1)!}.
\end{align}
Thus, for all $y_1,y_2$,
\begin{align}\label{C1lim}
&\lim_{t\to\infty}C^{(1)}_{2t+\epsilon,\rho,0}(u_1,y_1;u_2,y_2)\\
&=-\Id_{u_1=u_2+1}\Id_{y_1=y_2}-
\Id_{u_2<u_1}\Id_{y_1<y_2}e^{2\beta(y_1-y_2)}2^{u_1-u_2}\frac{(y_2-y_1)^{u_1-u_2-1}}{(u_1-u_2-1)!}\\
&=-e^{2\beta(y_1-y_2)}2^{u_1-u_2}H^{u_1-u_2}(y_2-y_1),
\end{align}
where
$$
H^m(z)=\frac{z^{m-1}}{(m-1)!}(\Id_{z>0}+\frac 12\Id_{z=0})
$$
for $m\ge 1$.

Consider now $C^{(2)}_{2t+\epsilon,\rho,0}$. We can write, using the fact the $z$ contour has no pole at $z=a$ for $t$ large, and by completing the $\overline{\mathcal{C}}_2/\sqrt{t}$ contour with an infinite semi-circle on the right,  
\begin{align} \nonumber
&C^{(2)}_{2t+\epsilon,\rho,0}(u_1,y_1;u_2,y_2)=\frac{1+a^2}{(2\pi\I)^2}\oint_{\mathcal{C}_1/\sqrt{t}}dz(-\sqrt{t})\\
&\oint_{\overline{\mathcal{C}_2}/\sqrt{t}}\frac{dw}{w-z}\frac{(-w\sqrt{t})^{u_2}
(a^{-1}+z)^{t+y_1\sqrt{t}}(a-z)^{t-y_1\sqrt{t}+\epsilon}}{(-z\sqrt{t})^{u_1}
(a^{-1}+w)^{t+y_2\sqrt{t}}(a-w)^{t-y_2\sqrt{t}+\epsilon}}\frac 1{(1+az)(a-w)}\\
&=\frac{1+a^2}{(2\pi\I)^2}\oint_{\mathcal{C}_1}~d\zeta\oint_{{\mathcal{C}_2}}\frac{d\omega}{\omega-\zeta}\frac{\omega^{u_2}
F_{y_1,t}(\zeta)}{\zeta^{u_1}
F_{y_2,t}(\omega)}\frac {(a+\zeta/\sqrt{t})^\epsilon}{(1-a\zeta/\sqrt{t})(a+\omega/\sqrt{t})^{1+\epsilon}}.
\end{align}
It now follows from lemma \ref{Est} that
\begin{align}\label{C2lim}
&\lim_{t\to\infty}C^{(2)}_{2t+\epsilon,\rho,0}(u_1,y_1;u_2,y_2)\\
&=\frac{2}{(2\pi\I)^2}\oint_{\mathcal{C}_1}~d\zeta\oint_{{\mathcal{C}_2}}\frac{d\omega}{\omega-\zeta}\frac{\omega^{u_2}
f_\beta(\zeta)g_{y_1,\beta}(\zeta)}{\zeta^{u_1}
f_\beta(\omega)g_{y_2,\beta}(\omega)}\nonumber\\
&=e^{2\beta(y_1-y_2)}\frac{2}{(2\pi\I)^2}\oint_{\mathcal{C}_1}~d\zeta\oint_{\mathcal{C}_2}\frac{d\omega}{\omega-\zeta}\frac{\omega^{u_2}
e^{-\zeta^2+2(\beta-y_1)\zeta}}{\zeta^{u_1}
e^{-\omega^2+2(\beta-y_2)\omega}}. \nonumber
\end{align}
We choose $k$ in (\ref{Fest}) so that $2k>u_2$, which gives a uniform $t$-independent upper bound on $\mathcal{C}_2'$.
Combining (\ref{C1lim}) and (\ref{C2lim}) we see that
\begin{equation}\label{C12lim}
\lim_{t\to\infty}(C^{(1)}_{2t+\epsilon,\rho,0}+C^{(2)}_{2t+\epsilon,\rho,0})(u_1,y_1;u_2,y_2)=
e^{2\beta(y_1-y_2)}\BK^{\text{minor}}(u_1,\beta-y_1;u_2,\beta-y_2).
\end{equation}

Next, consider
\begin{align}
&C^{(3)}_{2t+\epsilon,\rho,0}(u_1,y_1;~u_2,y_2)=-(-a)^{(y_2-y_1)\sqrt{t}}(-\sqrt{t})^{u_2-u_1}\sqrt{t}(1+a^2)\\
&\times \sum_{\kappa,\lambda=0}^\infty B_{4t+2\epsilon-2u_2,2t+2y_2\sqrt{t}-1} (\lambda+2t+\epsilon-\rho+1)\\
&\times (I-K_{2t+\epsilon})^{-1}_{\ge 0}(\lambda+2t+\epsilon-\rho+1,\kappa+2t+\epsilon-\rho+1)\\
&\times A_{4t+2\epsilon-2u_1,2t+2y_1\sqrt{t}-1,0}(\kappa+2t+\epsilon-\rho+1),
\end{align}
where $A_{\xi_1,\eta_1,0}$ is given by (\ref{Adelta}),
\begin{align}
&B_{4t+2\epsilon-2u_2,2t+2y_2\sqrt{t}-1} (j)\\
&=\frac{(-1)^j}{2\pi\I}\int_{\Gamma_{0,a}}(1+az)^{t-y_2\sqrt{t}+\epsilon}
(z-a)^{t+y_2\sqrt{t}}z^{u_2-1-2t-\epsilon+j}~dz\\
&+\frac{(-1)^j}{(2\pi\I)^2}\int_{\Gamma_{0,a}}dw\int_{\Gamma_{0,a,w}}\frac{dz}{w-z}
\frac{z^{u_2}(1+aw)^{2t+\epsilon}(w-a)^{2t+\epsilon}}{w^{2t+\epsilon-j+1}(1+az)^{t+y_2\sqrt{t}}(z-a)^{t-y_2\sqrt{t}+1+\epsilon}}
\end{align}
and
\begin{equation}
K_{2t+\epsilon}(j,k)=\frac{(-1)^{j+k}}{(2\pi\I)^2}\int_{\Gamma_{0,a}}dw\int_{\Gamma_{0,a,w}}\frac{dz}{z-w}
\frac{z^{2t+\epsilon-j}(1+aw)^{2t+\epsilon}(w-a)^{2t+\epsilon+1}}{w^{2t+\epsilon-k+1}(1+az)^{2t+\epsilon}(z-a)^{2t+\epsilon+1}}.
\end{equation}
Set
\begin{align}
\tilde{A}_{u_1,y_1,0}(\kappa)&=(-a)^{t-y_1\sqrt{t}}(-\sqrt{t})^{-u_1}(\sqrt{t})^{\rho-\kappa-1}\sqrt{t}(-1)^{-\epsilon}\\
&\times A_{4t+2\epsilon-2u_1,2t+2y_1\sqrt{t}-1,0}(\kappa+2t+\epsilon-\rho+1),
\end{align}
\begin{align}
\tilde{B}_{u_2,y_2}(\lambda)&=(-a)^{y_2\sqrt{t}-t}(-\sqrt{t})^{u_2}(\sqrt{t})^{\lambda-\rho+1}(-1)^{-\epsilon}\\
&\times B_{4t+2\epsilon-2u_2,2t+2y_2\sqrt{t}-1}(\lambda+2t+\epsilon-\rho+1)
\end{align}
and
\begin{equation}
\tilde{K}_{2t+\epsilon}(\lambda,\kappa)=(\sqrt{t})^{\kappa-\lambda}K_{2t+\epsilon}(\lambda+2t+\epsilon-\rho+1,\kappa+2t+\epsilon-\rho+1).
\end{equation}
Note that the matrix with elements
$$
(\sqrt{t})^{\kappa-\lambda}(I-K_{2t+\epsilon})^{-1}_{\ge 0}(\lambda+2t+\epsilon-\rho+1,\kappa+2t+\epsilon-\rho+1)
$$
is the inverse of the matrix with elements
$\delta_{\kappa,\lambda}-\tilde{K}_{2t+\epsilon}(\lambda,\kappa)$.
Thus,
\begin{equation}\label{C3mod}
C^{(3)}_{2t+\epsilon,\rho,0}(u_1,y_1;~u_2,y_2)=-(1+a^2)\sum_{\kappa,\lambda=0}^\infty\tilde{B}_{u_2,y_2}(\lambda)(I-\tilde{K}_{2t+\epsilon})^{-1}_{\ge 0}
\tilde{A}_{u_1,y_1,0}(\kappa)
\end{equation}
and we want to take the limit of this sum. (Note that the sum is finite even in the limit.) 

Rewriting in the same way as above we see from (\ref{Adelta}) that
\begin{align}
&\tilde{A}_{u_1,y_1,0}(\kappa)=-\frac 1{2\pi \I}\oint_{\mathcal{C}_2}\frac{\omega^{-\kappa-u_1+\rho-1}}{F_{-y_1,t}(\omega)}\frac{d\omega}
{(1-a\omega/\sqrt{t})^{1+\epsilon}(a+\omega/\sqrt{t})}\\
&+\frac 1{(2\pi \I)^2}\oint_{\mathcal{C}_1}d\zeta\oint_{\mathcal{C}_2}\frac{d\omega}{\omega-\zeta}\frac{\zeta^{-u_1}F_{y_1,t}(\zeta)}
{\omega^{\kappa+1-\rho}F_{0,t}(\omega)^2}\frac{(a+\zeta/\sqrt{t})^\epsilon}{(1-a\zeta/\sqrt{t})(a+\omega/\sqrt{t})^{1+\epsilon}(1-a\omega/\sqrt{t})^\epsilon}.
\end{align}
Using lemma \ref{Est} we can take the limit $t\to\infty$ to get
\begin{align}\label{Alim}
\lim_{t\to\infty}\tilde{A}_{u_1,y_1,0}(\kappa)&=-\frac 1{2\pi \I}\oint_{\mathcal{C}_2}\frac{\omega^{-\kappa-u_1+\rho-1}}{f_\beta(\omega)g_{-y_1,\beta}(\omega)}~d\omega\\
&+\frac 1{(2\pi \I)^2}\oint_{\mathcal{C}_1}d\zeta\oint_{\mathcal{C}_2}\frac{d\omega}{\omega-\zeta}\frac{\zeta^{-u_1}f_\beta(\zeta)g_{y_1,\beta}(\zeta)}
{\omega^{\kappa+1-\rho}f_\beta(\omega)^2}\\
&=-e^{2\beta y_1}\mathcal{A}_{u_1}^{\beta,y_1-\beta}(\kappa-\rho).
\end{align}
Similarly we get
\begin{align}
&\tilde{B}_{u_2,y_2}(\lambda)=\frac 1{2\pi \I}\oint_{\mathcal{C}_1}F_{-y_2,t}(\zeta)\zeta^{u_2+\lambda-\rho}(1-a\zeta/\sqrt{t})^\epsilon~d\zeta\\
&+\frac 1{(2\pi \I)^2}\oint_{\mathcal{C}_1}d\omega\oint_{\mathcal{C}_2}\frac{d\zeta}{\omega-\zeta}\frac{\omega^{\lambda-\rho}F_{0,t}(\omega)^2}
{\zeta^{-u_2}F_{y_2,t}(\zeta)}\frac{(a+\omega/\sqrt{t})^{1+\epsilon}(1-a\omega/\sqrt{t})^\epsilon}{(a+\zeta/\sqrt{t})^{1+\epsilon}}
\end{align}
and again by lemma \ref{Est} we find
\begin{align}\label{Blim}
\lim_{t\to\infty}\tilde{B}_{u_2,y_2}(\lambda)&=\frac 1{2\pi \I}\oint_{\mathcal{C}_1}f_\beta(\zeta)g_{-y_2,\beta}(\zeta)\zeta^{u_2+\lambda-\rho}~d\zeta\\
&+\frac 1{(2\pi \I)^2}\oint_{\mathcal{C}_1}d\omega\oint_{\mathcal{C}_2}\frac{d\zeta}{\omega-\zeta}\frac{\omega^{\lambda-\rho}f_\beta(\omega)^2}
{\zeta^{-u_2}f_\beta(\zeta)g_{y_2,\beta}(\zeta)}\\
&=e^{-2\beta y_2}\mathcal{B}_{u_2}^{\beta,y_2-\beta}(\lambda-\rho).
\end{align}
Finally,
\begin{equation}
\tilde{K}_{2t+\epsilon}(\lambda,\kappa)=\frac 1{(2\pi \I)^2}\oint_{\mathcal{C}_1}d\omega\oint_{\mathcal{C}_2}\frac{d\zeta}{\zeta-\omega}
\frac{\omega^{\kappa-\rho}F_{0,t(\omega)^2}}{\zeta^{\lambda-\rho+1}F_{0,t}(\zeta)^2}\frac{(1-a\omega/\sqrt{t})^\epsilon(a+\omega/\sqrt{t})^{\epsilon+1}}
{(1-a\zeta/\sqrt{t})^\epsilon(a+\zeta/\sqrt{t})^{\epsilon+1}}
\end{equation}
and we see from lemma \ref{Est} that
\begin{align}\label{Ktildelim}
\lim_{t\to\infty}\tilde{K}_{2t+\epsilon}(\lambda,\kappa)&=
\frac 1{(2\pi \I)^2}\oint_{\mathcal{C}_1}d\omega\oint_{\mathcal{C}_2}\frac{d\zeta}{\zeta-\omega}
\frac{\omega^{\kappa-\rho} f_\beta(\omega)^2}{\zeta^{\lambda-\rho+1}f_\beta(\zeta)^2}\\
&=\mathcal{K}^{\beta}(\lambda-\rho,\kappa-\rho).
\end{align}
It now follows from (\ref{C3mod}), (\ref{Alim}), (\ref{Blim}) and (\ref{Ktildelim}) that
\begin{equation}
\lim_{t\to\infty}C^{(3)}_{2t+\epsilon,\rho,0}(u_1,y_1;~u_2,y_2)=
2e^{2\beta(y_1-y_2)}\langle(I-\mathcal{K}^\beta)_{\ge-\rho}^{-1}\mathcal{A}_{u_1}^{\beta,y_1-\beta}(\lambda),{B}_{u_2}^{\beta,y_2-\beta}(\lambda)
\rangle_{\ge-\rho}.
\end{equation}
Together with (\ref{C12lim}) this proves (\ref{limC}). 

It is not difficult to get $t$-independent bounds on the $\BL$-kernel using the same arguments as above and in this
way we can show, in a standard manner, that the appropriate Fredholm determinant converges and obtain weak convergence of the
$\BL$-particle point process. We will not enter into the details.

%\newpage

\section{Proof of the inverse Kasteleyn formula}
\label{Sec:inverseK}

In this section we prove Theorem~\ref{thm:inverseK}. 
We will use the fact that
\begin{equation}
\begin{split}
&\mathbb{K}_{n,\rho} \left(b_2+1,\frac{b_2-b_1+2m+1}{2};w_2+1,\frac{w_2-w_1+2m+1}{2} \right)\\
&=-K_{n,m}^{\mathrm{inlier}} \left(b_2+1,\frac{b_2-b_1+2m+1}{2};w_2+1,\frac{w_2-w_1+2m+1}{2} \right),
\label{kasttokkernel}
\end{split}
\end{equation}
where the kernel $K_{n,m}^{\mathrm{inlier}}$ is the inlier kernel from \cite{AJvM}, dual to $\mathbb{K}_{n,p}$. We will use the form of the inlier kernel that
comes directly from the Eynard-Mehta theorem.
Let 
\begin{equation} \label{inverseK:defn_of_psi}
\begin{split}
\tilde{\psi}_{2r+\varepsilon_1,2s+\varepsilon_2}(x,y)&=\psi_{2r+\varepsilon_1,2s+\varepsilon_2}(x,y)\Id_{2r+\varepsilon_1<2s+\varepsilon_2} \\
%&=\frac{\mathbbm{I}_{2r+\varepsilon_1<2s+\varepsilon_2}}{ 2\pi \mathrm{i}} \int_{\Gamma_{0,a}} \frac{(1+a w)^{s-r}}{(1-a/w)^{s-r+\varepsilon_2-\varepsilon_1} w^{x-y}}\frac{dw}{w}. 
\end{split}
\end{equation}
where $\psi$ is defined in~\eqref{K12}.

Let $\mathtt{w}=(w_1,w_2) \in \mathtt{W}$, $\mathtt{b}=(b_1,b_2) \in \mathtt{B}$, $u_1=w_2+1$, $u_2=(w_2-w_1+2m+1)/2$, $v_1=b_2+1$, $v_2=(b_2-b_1+2m+1)/2$ and denote
\begin{equation} \label{inverseK:defn_of_sgn}
      \mathrm{sgn}(\mathtt{w}, \mathtt{b}) =  (-1)^{(w_1-w_2+b_1-b_2+2)/4} 
\end{equation}

Define
\begin{equation} \label{inverseK:defn_of_f1}
     \tilde{ f}_1(\mathtt{w}, \mathtt{b} ) = - \mathrm{sgn}(\mathtt{w},\mathtt{b})   \tilde{\psi}_{v_1,u_1} ( v_2,u_2)
\end{equation}
and
\begin{equation} \label{inverseK:defn_of_f2}
    \tilde{  f}_2 (\mathtt{w}, \mathtt{b} )=  \mathrm{sgn}(\mathtt{w},\mathtt{b})  \sum_{i,j=1}^{2m+1} \tilde{\psi}_{v_1,2n+1} (v_2,i-m-1) (A^{-1})_{i,j} \tilde{\psi}_{0,u_1} (j-m-1,u_2) 
\end{equation}
where 
\begin{equation} \label{inverseK:matrixA}
      A=\left( \tilde{\psi}_{0,2n+1} ( i-m-1,j-m-1) \right)_{i,j=1}^{2m+1}. 
\end{equation}

We then get
\begin{equation}
\begin{split}
&-(-1)^{(w_1-w_2+b_1-b_2+2)/4}\mathbb{K}_{n,\rho} \left(v_1,v_2;u_1,u_2 \right)
=\tilde{f}_1(\mathtt{w},\mathtt{b})+\tilde{f}_2(\mathtt{w},\mathtt{b})
\end{split}
\end{equation}
and we want to prove that
\begin{equation}
 K_a^{-1} (\mathtt{w},\mathtt{b})=\tilde{f}_1(\mathtt{w},\mathtt{b})+\tilde{f}_2(\mathtt{w},\mathtt{b}).
\end{equation}

To make the computations simpler, we define $T_a$ and $C$ with
 \begin{equation} \nonumber
  K_a( \mathtt{b},\mathtt{w}) =- (-1)^{(b_1+b_2-1)/2} T_a(\mathtt{b},\mathtt{w})
 \end{equation}
 and
 \begin{equation}
    \tilde{f}_1(\mathtt{w},\mathtt{b})+ \tilde{f}_2 (\mathtt{w},\mathtt{b})=- (-1)^{-(b_1+b_2-1)/2} C(\mathtt{w},\mathtt{b}).
 \end{equation}
and we will write $$f_i(\mathtt{w},\mathtt{b})=- (-1)^{(b_1+b_2-1)/2} \tilde{f}_i(\mathtt{w},\mathtt{b}).$$
Therefore, showing $K_a.(\tilde{f}_1+\tilde{f}_2) = \mathbbm{I}$ is equivalent to showing $T_a.C=\mathbbm{I}$.

 We will use the notation that $\mathtt{b}=(b_1,b_2)$ and $\mathtt{y}=(y_1,y_2)$ are black vertices.  We have that
 \begin{equation} \label{inverseK:KtimesKinv}
	(T_af_i)(\mathtt{b},\mathtt{y}) =\sum_{\mathtt{w} \sim \mathtt{b}} T_a(\mathtt{b},\mathtt{w}) f_i(\mathtt{w},\mathtt{y})
 \end{equation}
for $i \in \{1,2\}$ where $\mathtt{w} \sim \mathtt{y}$ means that $\mathtt{w}$ is nearest neighbors to $\mathtt{b}$ because $T_a(\mathtt{b}, \mathtt{w})=0$ if $\mathtt{b}$ and $\mathtt{w}$ are not nearest neighbors. We can then write
\begin{equation} \nonumber
	(T_a C) (\mathtt{b}, \mathtt{w}) = \sum_{ i \in \{1, 2\}}( T_a f_i)(\mathtt{b},\mathtt{y}). 
\end{equation}

% For the double Aztec diamond, we have to check that $(T_aC)(\mathtt{b}, \mathtt{y})= \delta_{\mathtt{b}=\mathtt{y}}$. 
 The number of terms on the right hand side of equation~\eqref{inverseK:KtimesKinv} is dependent on the location of  $\mathtt{b}$ and so we split the computation for finding $T_aC(\mathtt{b}, \mathtt{y})$ into the different locations of $\mathtt{b}$. These are given by
 \begin{enumerate}[(i)]
  \item  the interior, labeled $\mathcal{I}$,
  \item  the left hand boundary, labeled $\mathcal{L}$,
  \item  the bottom boundary, labeled $\mathcal{B}$, 
  \item  the top boundary but not equal to $(2n,2n-1)$, labeled $\mathcal{T}$ and 
  \item  the special point, $(2n,2n-1)$.
 \end{enumerate}

The left hand boundary, $\mathcal{L}$, consists of vertices $\mathtt{b}=(0,b_2)$ where $b_2 \in 2\mathbb{Z}+1$ and $1 \leq b_2 \leq 2n-1$.  For $\mathtt{b} \in \mathcal{L}$, we have that $\mathtt{b}$ has two neighboring white vertices given by $\mathtt{b}+e_1$ and $\mathtt{b}-e_2$. 

The bottom boundary, $\mathcal{B}$, consists of vertices $\mathtt{b}=(b_1,-1)$ where $b_1 \in 2 \mathbb{Z}$ and $4m+2 \leq b_1 \leq 4m+2n$. For $\mathtt{b} \in \mathcal{B}$, we have that $\mathtt{b}$ has two neighboring white vertices given by $\mathtt{b}+e_1$ and $\mathtt{b}+e_2$. 

The top boundary, $\mathcal{T}$,  consists of vertices $\mathtt{b}=(b_1,2n-1)$ where $b_1 \in 2 \mathbb{Z}$ and $2n+2 \leq b_1 \leq 4m+2n$. For $\mathtt{b} \in \mathcal{T}$, we have that $\mathtt{b}$ has two neighboring white vertices given by $\mathtt{b}-e_1$ and $\mathtt{b}-e_2$. 

For the special point, $\mathtt{b}=(2n,2n-1)$, we have that $\mathtt{b}$ has three neighboring white vertices given by $\mathtt{b}+e_2, \mathtt{b}-e_2$ and $\mathtt{b}-e_1$. 

The interior, $\mathcal{I}$, is given by the remaining vertices.  For $\mathtt{b} \in \mathcal{I}$, we have that $\mathtt{b}$ has four neighboring white vertices given by $\mathtt{b} \pm e_r$ for $r \in \{1,2\}$. 

In each of the above cases, we evaluate~\eqref{inverseK:KtimesKinv}.   
Due to the formulas for $f_1$ and $f_2$ being rather complicated, we used computer algebra to help with the computations. We give the calculation for the first case with full details and for the remaining cases, we provide an overview of the main steps.
We now proceed with checking the above cases.

\subsection*{The interior} 
% We first evaluate $T_a.C$ for $\mathtt{b}\in \mathcal{I}$.
 Using~\eqref{inverseK:defn_of_K} and the definition of $T_a(\mathtt{b}, \mathtt{w})$, we have that for $b\in \mathcal{I}$
\begin{equation}
\begin{split} \nonumber
	T_a(\mathtt{b},\mathtt{w})= \left\{
		\begin{array}{ll}
		\pm 1 &\mbox{if } \mathtt{w}=\mathtt{b}\pm e_1 \\
		\mp a \mathrm{i} & \mbox{if }\mathtt{w}=\mathtt{b}\pm e_2 \\
		0 & \mbox{otherwise}.  
		\end{array}	\right.
\end{split}
\end{equation}
 When $\mathtt{b} \in \mathcal{I} $ and $\mathtt{y} \in  \mathtt{B}_{AD}$, we have that~\eqref{inverseK:KtimesKinv} reads
 \begin{equation} \label{inverseK: T_a fia}
	 T_a f_i(\mathtt{b},\mathtt{y})= f_i(\mathtt{b}+e_1,\mathtt{y})- f_i(\mathtt{b}-e_1,\mathtt{y}) + (f_i(\mathtt{b}-e_2,\mathtt{y}) - f_i(\mathtt{b}+e_2,\mathtt{y})) a \mathrm{i}  
\end{equation}
for $i \in \{1,2\}$.

We first simplify~\eqref{inverseK: T_a fia} for $i=1$.  We can expand out the definition of $f_1$ in terms of $\tilde{\psi}$.  This means that we can rewrite~\eqref{inverseK: T_a fia} for $i=1$ in terms of  $\tilde{\psi}$. We obtain
\begin{equation}
\begin{split} \label{inverseK:interiorf1expand}
&T_a f_1(\mathtt{b},\mathtt{y})= -(-1)^{(b_1-b_2+y_1-y_2)/4} \mathrm{i}^{y_1+y_2}\\
& \left( a \tilde{\psi}_{y_2+1,b_2} \left( \frac{2m+1-y_1+y_2}{2}, \frac{2m-1-b_1+b_2}{2} \right)\right.\\
&+\tilde{\psi}_{y_2+1,b_2} \left( \frac{2m+1-y_1+y_2}{2} , \frac{2m+1-b_1+b_2}{2} \right) \\
&-\tilde{\psi}_{y_2+1,b_2+2} \left( \frac{2m+1-y_1+y_2}{2} , \frac{2m+1-b_1+b_2}{2} \right) \\
&+\left.a \tilde{\psi}_{y_2+1,b_2+2} \left( \frac{2m+1-y_1+y_2}{2} , \frac{2m+3-b_1+b_2}{2} \right)  \right)
\end{split}
\end{equation}
where $\tilde{\psi}$ is defined in~\eqref{inverseK:defn_of_psi}.  We shall evaluate $T_af_1(\mathtt{b},\mathtt{w})$ in three cases: $b_2=y_2, b_2>y_2$ and $b_2<y_2$. 

For $y_2=b_2$, we only need to consider the last two terms of~\eqref{inverseK:interiorf1expand} because the first two terms %involving $\tilde{\psi}$ 
 involving $\tilde{\psi}$ are zero by~\eqref{inverseK:defn_of_psi}.  Using~\eqref{inverseK:defn_of_psi}, we can rewrite~\eqref{inverseK:interiorf1expand} in terms of $\psi$ and hence write each expression as an integral.  We find  for $b_2=y_2$
\begin{align} \nonumber
      T_af_1(\mathtt{b}, \mathtt{y} )
&= (-1)^{\frac{b_1+y_1}{4}} \mathrm{i}^{y_1} \left(
 {\psi}_{y_2+1,y_2+2} \left( \frac{2m+1-y_1+y_2}{2} , \frac{2m+1-b_1+y_2}{2} \right) \right. \\ 
&-\left.a {\psi}_{y_2+1,y_2+2} \left( \frac{2m+1-y_1+y_2}{2} , \frac{2m+3-b_1+y_2}{2}\right) \right) \\
&=\frac{ (-1)^{\frac{b_1+y_1}{4}} \mathrm{i}^{y_1}}{2 \pi \mathrm{i}}
 \int_{\Gamma_{0,a}} \left(-\frac{w^{\frac{1}{2} \left(b_1-y_1\right)}}{a-w}+\frac{a w^{\frac{1}{2} \left(b_1-y_1-2\right)}}{a-w}\right)dw 
\\
 &=\frac{ (-1)^{\frac{b_1+y_1}{4}} \mathrm{i}^{y_1}}{2 \pi \mathrm{i}} \int_{\Gamma_{0,a}} w^{\frac{b_1-y_1-2}{2}} dw \\
		&= \left\{ \begin{array}{ll}
		             \mathrm{i}^{2y_1}=1 &\mbox{if } b_1=y_1 \\
		            0 &\mbox{otherwise}
		           \end{array} \right.
\end{align}
because $y_1 \in 2\mathbb{Z}$.

As both $\mathtt{b}$ and $\mathtt{y}$ are black vertices we have that $b_2$ and $y_2$ are both odd integers.  Therefore,  the condition that $b_2>y_2$ is equivalent to $b_2>y_2+1$.  
For $b_2>y_2$, all four terms of~\eqref{inverseK:interiorf1expand} involving $\tilde{\psi}$ are nonzero and each term can be rewritten using $\psi$.  We find that for $b_2>y_2$
\begin{equation}
\begin{split} \label{inverseK:interiorf1expand2}
&T_a f_1(\mathtt{b},\mathtt{y})= -(-1)^{(b_1-b_2+y_1-y_2)/4} \mathrm{i}^{y_1+y_2} \\
&\left( a \psi_{y_2+1,b_2} \left( \frac{2m+1-y_1+y_2}{2}, \frac{2m-1-b_1+b_2}{2} \right)\right.\\
&+\psi_{y_2+1,b_2} \left( \frac{2m+1-y_1+y_2}{2} , \frac{2m+1-b_1+b_2}{2} \right) \\
&-\psi_{y_2+1,b_2+2} \left( \frac{2m+1-y_1+y_2}{2} , \frac{2m+1-b_1+b_2}{2} \right) \\
&+\left.a \psi_{y_2+1,b_2+2} \left( \frac{2m+1-y_1+y_2}{2} , \frac{2m+3-b_1+b_2}{2} \right)  \right)
\end{split}
\end{equation}
To evaluate~\eqref{inverseK:interiorf1expand2}, we need to evaluate an expression of the form
\begin{equation} \label{inverseK:interiorf1expand3}
\begin{split}
&a \psi_{2r,2s+1}(x_1,x_2)+\psi_{2r,2s+1}(x_1,x_2+1) \\& -  \psi_{2r,2s+3}(x_1,x_2+1)+a\psi_{2r,2s+3} (x_1,x_2+2)
\end{split}
\end{equation}
for $r,s,x_1,x_2 \in \mathbb{Z}$.  We can expand~\eqref{inverseK:interiorf1expand3} in terms of its integral decomposition and combine all the terms under one integral.  We obtain
\begin{equation}\nonumber
\oint_{\Gamma_{0,a}} \frac{dz}{2\pi \mathrm{i}} z^{x_1-x_2} \frac{(1+a z)^{s-r}}{(1-\frac{a}{z})^{s-r+1}} \left(a+\frac{1}{z} -\frac{z^{-1} (1+a z)}{1-\frac{a}{z}} + \frac{a z^{-2} (1+az)}{1-\frac{a}{z}}\right).
\end{equation}
In the above equation, the term inside the parenthesis is zero.  This  means  we can write
\begin{equation} \label{inverseK:interiorf1expand4}
\begin{split}
&a\tilde{ \psi}_{2r,2s+1}(x_1,x_2)+\tilde{\psi}_{2r,2s+1}(x_1,x_2+1) \\& - \tilde{ \psi}_{2r,2s+3}(x_1,x_2+1)+a\tilde{\psi}_{2r,2s+3} (x_1,x_2+2)=0 \hspace{20mm}\mbox{for } r<s.
\end{split}
\end{equation}
 Using the relation in~\eqref{inverseK:interiorf1expand4}, we have that right hand side of~\eqref{inverseK:interiorf1expand2} is equal to zero for $b_2>y_2$ which means that $T_a f_1(\mathtt{b},\mathtt{w})=0$ for $b_2>y_2$.  

As $b_2$ and $y_2$ are both odd integers, the condition that $b_2<y_2$ is equivalent to $b_2+1<y_2$.  We can expand~\eqref{inverseK:interiorf1expand} using the definition of $\tilde{\psi}$ given in~\eqref{inverseK:defn_of_psi} and we find that all the terms of~\eqref{inverseK:interiorf1expand} are equal to zero.  Therefore, we have that $T_a f_1(\mathtt{b}, \mathtt{w}) =0$ for $b_2<y_2$.  
 
We have shown that for $b\in \mathcal{I} $
\begin{equation} \nonumber
       T_a f_1(\mathtt{b}, \mathtt{y}) = \left\{ \begin{array}{ll}
                                              1 & \mbox{if } \mathtt{b}=\mathtt{y},    \\
                                              0 & \mbox{ otherwise}.
                                             \end{array} \right.
\end{equation}
For the term $T_af_2(\mathtt{b},\mathtt{y})$, using~\eqref{inverseK: T_a fia} we can write $T_af_2(\mathtt{b},\mathtt{y})$ under one sum.  We obtain 
\begin{equation}\nonumber
\begin{split}
 T_a f_2(\mathtt{b}, \mathtt{y})& = (-1)^{\frac{b_1-b_2+y_1-y_2}{4}} \mathrm{i}^{y_1+y_2} \sum_{i,j=1}^{2m+1} (A^{-1})_{ij} \\
&\tilde{\psi}_{1+y_2,2n+1} \left( \frac{2m+1-y_1+y_2}{2}, i-m-1 \right)\\ 
&\Delta_{\mathcal{I}} \tilde{\psi}_{0,b_2+1} \left(j-m-1,\frac{2m+2-b_1+b_2}{2} \right)
\end{split}
\end{equation}
where 
\begin{equation} \nonumber
\begin{split}
 &\Delta_{\mathcal{I}} \tilde{\psi}_{0,b_2+1} \left(j-m-1,\frac{2m+2-b_1+b_2}{2} \right)=\\
 & a \tilde{\psi}_{0,b_2}\left(j-m-1,\frac{2m-1-b_1+b_2}{2} \right) 
+\tilde{\psi}_{0,b_2}\left(j-m-1,\frac{2m+1-b_1+b_2}{2} \right)  \\
  &-  \tilde{\psi}_{0,b_2+2}\left(j-m-1,\frac{2m+1-b_1+b_2}{2} \right) \\
&+ a \tilde{\psi}_{0,b_2+2}\left(j-m-1,\frac{2m+3-b_1+b_2}{2} \right) 
\end{split}
\end{equation}
As $b_2>0$, we can use the relation given in~\eqref{inverseK:interiorf1expand4} to find that $$\Delta_{\mathcal{I}} \tilde{\psi}_{0,b_2+1} \left(j-m-1,\frac{2m+2-b_1+b_2}{2} \right)=0.$$  Therefore, we have 
\begin{equation}\nonumber
T_af_2 (\mathtt{b}, \mathtt{y})=0 \hspace{10mm} \mathtt{b} \in \mathcal{I}, \mathtt{y} \in  \mathtt{B}_{AD}.  
\end{equation}
To summarize, we have 
\begin{equation}\nonumber
      T_aC(\mathtt{b},\mathtt{y}) = \left\{ \begin{array}{ll}
                                            1 & \mbox{if } \mathtt{b}=\mathtt{y}\\
                                            0 & \mbox{otherwise}.
                                           \end{array}
                                    \right.
\end{equation}

\subsection*{The left hand boundary}
Next, we check $T_a.C$ for $\mathtt{b}$ on the left hand boundary.  For $\mathtt{b} \in \mathcal{L} $ we have that 
\begin{equation} \nonumber
T_a(\mathtt{b},\mathtt{w}) = \left\{ \begin{array}{ll}
			1& \mbox{if } \mathtt{w}= \mathtt{b}+e_1 \\
			a \mathrm{i}& \mbox{if } \mathtt{w}=\mathtt{b}-e_2 \\
			0 & \mbox{otherwise}. \end{array} \right.
\end{equation}
 For  $\mathtt{b}\in \mathcal{L} $ and $\mathtt{y} \in  \mathtt{B}_{AD}$, using the above equation we find that~\eqref{inverseK:KtimesKinv} is given by
\begin{equation}\label{Kinverse:KKinvleft}
	T_af_i (\mathtt{b},\mathtt{y}) = f_i(\mathtt{b}+e_1,\mathtt{y})+f_i(\mathtt{b}-e_2,\mathtt{y})a \mathrm{i}
\end{equation}
Similar to the interior, we can expand $T_a f_1(\mathtt{b},\mathtt{w})$ in terms of $\tilde{\psi}$ and rewrite $T_af_1(\mathtt{b},\mathtt{w})$ in terms of an integral using the definition of $\psi$ given in~\eqref{K12}. By a computation, we find that for $b_2=y_2$, we obtain
\begin{equation} \label{inverseK:LB1}
\begin{split}
      T_af_1 (\mathtt{b},\mathtt{y})&= \frac{1}{2\pi \mathrm{i} } \int_{\Gamma_{0,a}} \frac{e^{3 \pi \mathrm{i}y_1/4} w^{-y_1/2}}{w-a} dw \\    
            &= \left\{ \begin{array}{ll} 
      1 & \mbox{if } y_1=0\\
      0 & \mbox{otherwise}. 
\end{array} \right.
\end{split}
\end{equation}
because $y_1 \in2 \mathbb{N}$, as the integrand has no pole at infinity. 
For $b_2>y_2$, we also find by computation that  
\begin{equation}\nonumber
\begin{split}
       T_a f_1(\mathtt{b},\mathtt{y}) &= \frac{(-1)^{(-b_2+y_1-y_2)/4} \mathrm{i}^{y_1+y_2}}{2 \pi \mathrm{i} } (1+a^2) \\
      &\int_{\Gamma_{0,a}} w^{-y_1/2} (w-a)^{(y_2-b_2-2)/2} (aw+1)^{(b_2-y_2-2)/2} dw
\end{split}
\end{equation}
As $b_2>y_2$ (i.e. $b_2 \geq y_2+2$) and $y_1\geq 0$, the integrand has no pole at infinity or $-1/a$, hence
\begin{equation}
     T_a f_1(\mathtt{b},\mathtt{y}) =0. \label{inverseK:LB2}
\end{equation}
We find that for  $b_2<y_2$ 
\begin{equation} \label{inverseK:LB3}
      T_af_1( \mathtt{b},\mathtt{y})=0
\end{equation}
by using the same reasoning as the case for $\mathtt{b}$ in the interior, that is, each term in the expansion of $T_a f_1(\mathtt{b},\mathtt{y})$ in terms of $\tilde{\psi}$ is equal to zero. 

  Using equations~\eqref{inverseK:LB1},~\eqref{inverseK:LB2} and~\eqref{inverseK:LB3} we have
\begin{equation} \label{inverseK:LB4}
       T_a f_1(\mathtt{b},\mathtt{y}) = \left\{
      \begin{array}{ll}
       1 & \mbox{if } \mathtt{b}=\mathtt{y} \\
       0 & \mbox{otherwise}
      \end{array} \right.
\end{equation}
for $\mathtt{b} \in \mathcal{L}$.  
Similar to the interior case, using the expansion of $T_af_2(\mathtt{b},\mathtt{y})$ given in~\eqref{Kinverse:KKinvleft}, we can expand using the definition of ${f}_2(\mathtt{b},\mathtt{w})$ to obtain
\begin{equation} \label{Kinverse:LBf2}
\begin{split}
   &   T_af_2( \mathtt{b}, \mathtt{y}) =(-1)^{(y_1-b_2-y_2)/4} \mathrm{i}^{y_1+y_2} \sum_{i,j=1}^m  (A^{-1})_{i,j}\\
& \tilde{\psi}_{1+y_2,2n+1} \left( \frac{2m+1-y_1+y_2}{2}, i-m-1 \right)  
      \Delta_{\mathcal{L}} \tilde{\psi}_{0,b_2+1} \left(j-m-1, \frac{2m+2+b_2}{2} \right)
\end{split}
\end{equation}
where 
\begin{equation}\nonumber
\begin{split} 
       \Delta_{\mathcal{L}} \tilde{\psi}_{0,b_2+1} \left(j-m-1, \frac{2m+2+b_2}{2} \right)&=a \tilde{\psi}_{0,b_2}  \left(j-m-1,\frac{2m-1+b_2}{2} \right) \\ 
       &- \tilde{\psi}_{0, b_2+2} \left(j-m-1, \frac{2m+1+b_2}{2} \right).
      \end{split}      
\end{equation}
We can rewrite the right hand side of the above equation in terms of its contour integral using~\eqref{inverseK:defn_of_psi} which gives 
\begin{equation}\nonumber
     - \frac{1}{2\pi \mathrm{i}} \int_{\Gamma_{0,a}} (1+a^2) w^{-1+j-2m} (w-a)^{(-3-b_2)/2} (1+a w)^{(b_2-1)/2} dw.
\end{equation}
Since $1\leq j \leq m, b_2\geq 1$ the integrand has no pole at $-1/a$ or infinity,hence the above quantity is zero and so the right hand side of~\eqref{Kinverse:LBf2} is equal to zero.  Combining with~\eqref{inverseK:LB4} gives
\begin{equation}\nonumber
      T_aC(\mathtt{b},\mathtt{y} ) = \left\{
      \begin{array}{ll}
       1 & \mbox{if } \mathtt{b}=\mathtt{y} \\
       0 & \mbox{otherwise}
      \end{array} \right.
\end{equation}
for $\mathtt{b}\in\mathcal{L}$ and $\mathtt{y} \in  \mathtt{B}_{AD}$. 

\subsection*{The Bottom boundary}
We now consider the bottom boundary. We have that for $\mathtt{b} \in \mathcal{B} $
\begin{equation} \nonumber
	T_a (\mathtt{b},\mathtt{w})= \left\{
			\begin{array}{ll}
				1 & \mbox{if } \mathtt{w}=\mathtt{b}+e_1\\
				-a \mathrm{i} & \mbox{if } \mathtt{w} = \mathtt{b}+e_2\\
				0 & \mbox{otherwise}.
			\end{array} \right.
\end{equation}
Using the above equation,~\eqref{inverseK:KtimesKinv} can be rewritten for  $\mathtt{b} \in \mathcal{B}$ and is given by
\begin{equation} \label{inverseK: T_a fBB}
      T_af_i(\mathtt{b},\mathtt{y}) =f_i (\mathtt{b}+e_1,\mathtt{y})- a\mathrm{i}f_i (\mathtt{b}+e_2,\mathtt{y}).
\end{equation}
We first consider $T_af_1(\mathtt{b},\mathtt{w})$ for $b_2=y_2=-1$. Similar to the analogous computation for $\mathtt{b}$ in the interior, we can expand the right hand side of~\eqref{inverseK: T_a fBB} in terms of $\psi$ and rewrite the expression as an integral. By a computation,  we find that
\begin{equation} \label{inverseK:BB1}
      \begin{split}
      T_af_1( \mathtt{b},\mathtt{y}) &= \frac{(-1)^{(b_1+y_1)/4} \mathrm{i}^{y_1}}{2 \pi \mathrm{i}} \int_{\Gamma_{0,a}}  w^{(b_1-y_1-2)/2} dw \\
      &= \left\{ \begin{array}{ll}
      1 &\mbox{if } \mathtt{b}=\mathtt{y} \\
      0 &\mbox{otherwise}.  \end{array} \right.
      \end{split}
\end{equation}
For $b_2=-1<y_2$, using the same reasoning as given for $\mathtt{b}$ in the interior, we have
\begin{equation} \label{inverseK:BB2}
       T_a f_1(\mathtt{b},\mathtt{y})=0.
\end{equation}
Combining equations~\eqref{inverseK:BB1} and~\eqref{inverseK:BB2}, we have
\begin{equation} \label{inverseK:BB3}
      T_af_1(\mathtt{b},\mathtt{y})= \left\{ \begin{array}{ll} 
                        1   & \mbox{if } b_1=y_1 \\
                        0 & \mbox{otherwise}
                                    \end{array} \right.
\end{equation}
for $\mathtt{b}\in \mathcal{B}$ and $y \in  \mathtt{B}_{AD}$.  For $T_af_2(\mathtt{b},\mathtt{y})$, using~\eqref{inverseK: T_a fBB} and following the same steps for the analogous computation for $\mathtt{b}$ in the interior, we have
\begin{equation}\nonumber
  \begin{split}
      T_af_2( \mathtt{b},\mathtt{y}) &= (-1)^{(b_1-b_2+y_1-y_2)/4} \mathrm{i}^{y_1+y_2} \sum_{i,j=1}^{2m+1}   (A^{-1})_{i,j}\\
&\tilde{\psi}_{y_2+1,2n+1}\left( \frac{2m+1-y_1+y_2}{2},i-m-1\right)\\  
    &\Delta_{\mathcal{B}} \tilde{\psi}_{0,2+b_2} \left( j-m-1, \frac{2m+2-b_1+b_2}{2} \right) 
  \end{split}
\end{equation}
where
\begin{equation}\nonumber
 \begin{split}
 &   \Delta_{\mathcal{B}} \tilde{\psi}_{0,2+b_2} \left( j-m-1, \frac{2m+2-b_1+b_2}{2} \right)=\\
&-\tilde{\psi}_{0,b_2+2} \left( j-m-1, \frac{2m+1-b_1+b_2}{2} \right)+ a \tilde{\psi}_{0,b_2+2} \left( j-m-1, \frac{2m+3-b_1+b_2}{2} \right).
 \end{split}
\end{equation}
We can rewrite the above equation using the integral definition of $\psi$.  A computation gives
\begin{equation}\nonumber
	       \Delta_{\mathcal{B}} \tilde{\psi}_{0,2+b_2} \left( j-m-1, \frac{2m+2-b_1+b_2}{2} \right)=\frac{-1}{2\pi \mathrm{i}} \int_{\Gamma_{0,a}} w^{-2+j-2m+b_1/2} dw=0
\end{equation}
because $b_1 \geq 4m+2$ by the co-ordinates of the bottom boundary.  We have obtained
\begin{equation} \label{inverseK:BB4}
        T_af_2( \mathtt{b},\mathtt{y})=0
\end{equation}
for $\mathtt{b} \in \mathcal{B}$ and $ \mathtt{y}\in \mathtt{B}_{AD}$.  Using~\eqref{inverseK:BB3} and~\eqref{inverseK:BB4} gives
\begin{equation}\nonumber
      T_aC(\mathtt{b},\mathtt{y})=\left\{ \begin{array}{cc}
                                          1 & \mbox{if } \mathtt{b}=\mathtt{y} \\
                                          0 & \mbox{otherwise} 
                                         \end{array} \right.
\end{equation}
for $\mathtt{b} \in \mathcal{B}$ and $ \mathtt{y}\in \mathtt{B}_{AD}$.

\subsection*{The Top Boundary}

We can now consider $\mathtt{b}=(b_1,b_2)$ on the top boundary which means that $b_2=2n-1$ and $b_1>2n$.  We have that for $\mathtt{b} \in \mathcal{T}$ 
\begin{equation}
T_a (\mathtt{b},\mathtt{w}) = \left \{
		\begin{array}{ll}
			-1 & \mbox{if } \mathtt{w}=\mathtt{b}-e_1 \\
			a \mathrm{i} & \mbox{if } \mathtt{w}=\mathtt{b}-e_2 \\
			0 & \mbox{otherwise}
		\end{array} \right.
\end{equation}
We can use the above equation to rewrite~\eqref{inverseK:KtimesKinv}.  We obtain for
   $\mathtt{b} \in \mathcal{T} $ and $y \in  \mathtt{B}_{AD}$
\begin{equation} \label{inverseK: T_a fiTB}
       T_a f_{i}(\mathtt{b},\mathtt{y})=-f_i(\mathtt{b}-e_1,\mathtt{y})+a\mathrm{i}f_i(\mathtt{b}-e_2,\mathtt{y})
\end{equation}
We first consider $T_a f_1(\mathtt{b},\mathtt{y})$ for $b_2=y_2=2n-1$.  By using the analogous expansion for $\mathtt{b}$ as the interior case in terms of $\psi$ and its integral definition, we find that for $b_2=y_2=2n-1$
\begin{equation}
    T_af_1(\mathtt{b},\mathtt{y})=0. \label{inverseK:TBf1firstpart}
\end{equation}
For $y_2<b_2=2n-1$, we  find that
\begin{equation} \label{inverseK:TBf1part}
\begin{split}
    T_af_1(\mathtt{b},\mathtt{y}) &= -\frac{(-1)^{(1+b_1+y_1-y_2)/4} \mathrm{i}^{y_1+y_2-n}}{2\pi \mathrm{i} } \\
& \int_{\Gamma_{0,a}}\frac{ \left(1-\frac{a}{w}\right)^{(1-2n+y_2)/2}  (1+a w)^{(2n-1-y_2)/2}}{w^{(2n+1-b_1+y_1-y_2)/2}} dw.
\end{split}
\end{equation}
For $T_af_2(\mathtt{b},\mathtt{y})$, using~\eqref{inverseK: T_a fiTB} we can follow the analogous computation in the interior case, we find that
\begin{equation}
\begin{split} \label{inverseK:TB T_a f2}
 &     T_af_2 (\mathtt{b},\mathtt{y}) = (-1)^{(b_1-b_2+y_1-y_2)/4} \mathrm{i}^{y_1+y_2} \sum_{i,j=1}^{2m+1}(A^{-1})_{i,j} \\
& \tilde{\psi}_{y_2+1,2n+1} \left( \frac{2m+1-y_1+y_2}{2}, i-m-1 \right)  \Delta_{\mathcal{T}} \tilde{\psi}_{0,b_2} \left(j-m-1,\frac{2m - b_1+b_2}{2}\right) 
\end{split}
\end{equation}
where
\begin{equation}
\begin{split}\nonumber
     & \Delta_{\mathcal{T}} \tilde{\psi}_{0,b_2} \left(j-m-1,\frac{2m - b_1+b_2}{2}\right)=a \tilde{\psi}_{0,b_2} \left(j-m-1,\frac{2m-1 - b_1+b_2}{2}\right)\\
      &+\tilde{\psi}_{0,b_2} \left(j-m-1,\frac{2m +1- b_1+b_2}{2}\right)\\
      &= \frac{1}{2\pi \mathrm{i}} \int_{\Gamma_{0,a}} \frac{w^{-2+j-2m-n+b_1/2} (1+a w)^n}{\left(1-\frac{a}{w} \right)^n }dw
\end{split}
\end{equation}
which is found by expanding out the integrands using the integral definition of ${\psi}$ given in~\eqref{K12} where we have used the fact that $b_2=2n-1$.  Notice that we can rewrite the right hand side of the above equation as 
\begin{equation}
\begin{split} \label{inverseK:TBrearrange}
&      \frac{1}{2\pi \mathrm{i}} \int_{\Gamma_{0,a}} \frac{w^{-2+j-2m-n+b_1/2} (1+a w)^n}{\left(1-\frac{a}{w} \right)^n }dw=\\
& \tilde{\psi}_{0,2n+1} \left(j-m-1,\frac{2m- b_1+2n}{2}\right)
-a\tilde{\psi}_{0,2n+1} \left(j-m-1,\frac{2m +2- b_1+2n}{2}\right).
\end{split}
\end{equation}
which follows by using expanding the right hand side of the above equation using the integral definition of $\psi$. 
By definition of the matrix $A$ given in~\eqref{inverseK:matrixA}, we have
\begin{equation} \label{inverseK:TBAAinv}
      \sum_{j=1}^m (A^{-1})_{i,j} \tilde{\psi}_{0,2n+1} \left( j-m-1,k \right) = \delta_{i,k+m+1}.
\end{equation}
Using~\eqref{inverseK:TBAAinv} and~\eqref{inverseK:TBrearrange} in~\eqref{inverseK:TB T_a f2} gives
\begin{equation} \label{inverseK:TBf2part}
 \begin{split}
     & T_af_2 (\mathtt{b},\mathtt{y}) = (-1)^{(b_1-2n+1+y_1-y_2)/4} \mathrm{i}^{y_1+y_2}\\& \left(\tilde{\psi}_{y_2+1,2n+1} \left( \frac{2m+1 -y_1+y_2}{2} ,\frac{2m - b_1+2n}{2}\right) \right. \\
      &-a\left. \tilde{\psi}_{y_2+1,2n+1} \left( \frac{2m+1 -y_1+y_2}{2} ,\frac{2m+2 - b_1+2n}{2}\right) \right)\\
      &= \frac{(-1)^{(1+b_1+y_1-y_2)/4} \mathrm{i}^{y_1+y_2-n}}{2\pi \mathrm{i}} \int_{\Gamma_{0,a}} \frac{\left(1- \frac{a}{w} \right)^{(1-2n+y_2)/2} (1+ a w)^{(2n-1-y_2)/2} }{w^{(2n+1-b_1+y_1-y_2)/2}} dw 
 \end{split}
\end{equation}
Therefore, for $y_2<b_2$ using~\eqref{inverseK:TBf1part} and~\eqref{inverseK:TBf2part} gives
\begin{equation} \label{inverseK:TBy2<b2}
      T_a.(f_1+f_2)(\mathtt{b},\mathtt{y})= 0.
\end{equation}
 For $y_2=b_2=2n-1$, using~\eqref{inverseK:TBf2part} we have
\begin{equation} \label{inverseK:TBf2party2=b2}
      T_af_2 (\mathtt{b},\mathtt{y}) = \frac{(-1)^{(b_1+y_1)/4}}{4} \mathrm{i}^{y_1} \int_{\Gamma_{0,a}} w^{(b_1-y_1-2)/2} dw=  \left\{ \begin{array}{ll}
                                              1 & \mbox{if } b_1=y_1 \\
                                              0 & \mbox{otherwise}.
                                             \end{array} \right.  
\end{equation}
Finally, using~\eqref{inverseK:TBf1firstpart}, ~\eqref{inverseK:TBy2<b2} and~\eqref{inverseK:TBf2party2=b2}, we have
\begin{equation}\nonumber
      T_aC (\mathtt{b},\mathtt{y} ) = \left\{ \begin{array}{ll}
                                              1 & \mbox{if } b_1=y_1 \\
                                              0 & \mbox{otherwise}
                                             \end{array} \right. 
\end{equation}
for $\mathtt{b} \in \mathcal{T}$ and $y \in  \mathtt{B}_{AD}$. 

\subsection*{The Special Point}
Finally, we need to consider the special point, i.e. when we choose $\mathtt{b}=(2n,2n-1)$. The special point has three neighboring white vertices and we have
\begin{equation}
	T_a( (2n,2n-1),\mathtt{w}) = \left\{
			\begin{array}{ll}
			-1 & \mbox{if } \mathtt{w}=(2n-1,2n-2) \\
			-a \mathrm{i} & \mbox{if } \mathtt{w}=(2n-1,2n) \\
			a \mathrm{i} & \mbox{if } \mathtt{w} =(2n+1,2n-2) \\
			0 & \mbox{otherwise}
			\end{array} \right.
\end{equation}
 For $\mathtt{b}=(2n,2n-1)$ and $y\in  \mathtt{B}_{AD}$,~\eqref{inverseK:KtimesKinv} becomes
\begin{equation} \label{inverseK: T_a fiSP}
       T_a f_i(\mathtt{b},\mathtt{y})= -f_i(\mathtt{b}-e_1,y) + a \mathrm{i} f_i (\mathtt{b}-e_2,\mathtt{y})- a \mathrm{i} f_i (\mathtt{b}+e_2,\mathtt{y}).
\end{equation}
For $y_2=2n-1$, using the analogous steps for $\mathtt{b}$ in the interior, we can write using~\eqref{inverseK: T_a fiSP} and the integral definition of $\psi$ 
\begin{equation} \label{inverseK:SPf1firstpart}
       T_a f_1(\mathtt{b},\mathtt{y}) = \frac{(-1)^{y_1/4} \mathrm{i}^{n+y_1} }{2\pi \mathrm{i}} \int_{\Gamma_{0,a}} \frac{a w^{(2n-1-y_1)/2}}{a-w} dw=0,
\end{equation}
since $y_1 \geq 2n+2$, when $y_2=2n-1$.
For $y_2<2n-1$, we can write $T_a f_1 (\mathtt{b},\mathtt{w})$ using the integral definition of $\psi$ given in~\eqref{K12}.  We obtain
\begin{equation}\label{inverseK:SPf1secondpart}
\begin{split}
     T_af_1 (\mathtt{b},\mathtt{y})&=- \frac{(-1)^{(1+y_1-y_2)/4}\mathrm{i}^{y_1+y_2}}{2\pi \mathrm{i}} \\
&\int_{\Gamma_{0,a}}  \frac{\left(1-\frac{a}{w} \right)^{(-1-2n+y_2)/2} (1+a w)^{(2n-1-y_2)/2} }{w^{(y_1+1-y_2)/2}} dw 
\end{split}
\end{equation}
We can now expand out $T_af_2(\mathtt{b},\mathtt{y})$ using~\eqref{inverseK: T_a fiSP} and the analogous computations given in $\mathtt{b}$ in the interior.  We find that
\begin{equation}
\begin{split}\label{inverseK:SPf2partstar}
 &     T_af_2(\mathtt{b},\mathtt{y}) = (-1)^{(b_1-b_2+y_1-y_2)/4} \mathrm{i}^{y_1+y_2} \sum_{i,j=1}^{2m+1}  (A^{-1})_{i,j}\\
& \tilde{\psi}_{1+y_2,2n+1} \left( \frac{2m+1-y_1+y_2}{2}, i-m-1 \right) 
\Delta_{\mathcal{S}} \tilde{\psi}_{0,b_2} \left( j-m-1, \frac{2m-b_1+b_2}{2} \right)
\end{split}
\end{equation}
where
\begin{equation} \label{inverseK:SPf2part1}
      \begin{split}
	    \Delta_{\mathcal{S}} \tilde{\psi}_{0,b_2} \left( j-m-1, \frac{2m-b_1+b_2}{2} \right) &=a  \tilde{\psi}_{0,b_2} \left( j-m-1, \frac{2m-1-b_1+b_2}{2} \right)  \\
       &+ \tilde{\psi}_{0,b_2} \left( j-m-1, \frac{2m+1-b_1+b_2}{2} \right)\\ 
       & +a\tilde{\psi}_{0,b_2+2} \left( j-m-1, \frac{2m+3-b_1+b_2}{2} \right).
      \end{split}
\end{equation}
By setting $\mathtt{b}=(2n,2n-1)$, we have by a computation that
\begin{equation}
\begin{split}\nonumber
   \Delta_{\mathcal{S}} \tilde{\psi}_{0,2n} \left( j-m-1, \frac{2m-2n+2n-1}{2} \right)&= \frac{1}{2\pi \mathrm{i}} \int_{\Gamma_{0,a}} \frac{w^{-2+j-2m} (1+aw)^n}{\left( 1- \frac{a}{w} \right)^{n+1} } dw \\ 
   &=\tilde{\psi}_{0,2n+1}\left(j-m-1,m+1 \right)
\end{split}
\end{equation}
where the last equation follows from~\eqref{K12} and~\eqref{inverseK:defn_of_psi}.  We can substitute the above equation into~\eqref{inverseK:SPf2partstar} and  use  equation~\eqref{inverseK:TBAAinv} in a  similar way to the analogous computation found in the previous subsection.   We find that
\begin{equation} \label{inverseK:SPf2secondpart}
\begin{split}
       T_a f_2(\mathtt{b},\mathtt{y})
&=(-1)^{(b_1-b_2+y_1-y_2)/4} \mathrm{i}^{y_1+y_2} \tilde{\psi}_{1+y_2,2n+1} \left( \frac{2m+1-y_1+y_2}{2},m+1 \right)\\
&=\frac{(-1)^{(1+y_1-y_2)/4} \mathrm{i}^{y_1+y_2} }{2\pi \mathrm{i} } \int_{\Gamma_{0,a}} \frac{w^{(y_2-y_1-1)/2} (1+a w)^{(2n-1-y_2)/2}}{\left(1- \frac{a}{w} \right)^{(2n+1-y_2)/2}} dw 
\end{split}
\end{equation}
We have that for $y_2=2n-1$,~\eqref{inverseK:SPf2secondpart} is equal to 
\begin{equation}
	T_a f_2(\mathtt{b},\mathtt{y})=\frac{1}{2\pi \mathrm{i}} \int_{\Gamma_{0,a}} (-1)^{y_1/4} \mathrm{i}^{n+y_1} w^{n-1+y_1/2} dw = \left\{  \begin{array}{ll}
					1 &\mbox{if } y_1=2n\\
					0 & \mbox{otherwise}.
					\end{array} \right. 
\end{equation}
for $\mathtt{b}=(2n,2n-1)$.  

Therefore, we have for $\mathtt{b}=(2n,2n-1)$
\begin{equation}\nonumber
     T_a.C( \mathtt{b},\mathtt{y} )=\sum_{i \in\{1,2\}} T_a f_i (\mathtt{b},\mathtt{y}) = \left\{ \begin{array}{ll}
                                            1 & \mbox{if } \mathtt{b}=\mathtt{y}\\
                                            0 & \mbox{otherwise}
                                           \end{array} \right. 
\end{equation}
by using equations~\eqref{inverseK:SPf1firstpart},~\eqref{inverseK:SPf1secondpart} and~\eqref{inverseK:SPf2secondpart} for $\mathtt{b}=(2n,2n-1)$ and $y\in  \mathtt{B}_{AD}$.

%\newpage

\section{Proof of the formula for the $\BL$-kernel} \label{Sec:Lkernel}

In this section, we derive the $\BL$ kernel from the inverse Kasteleyn matrix.  Let $b_j=(\xi_j,\eta_j)$, $1 \leq j \leq l$, be the positions of black vertices in Kasteleyn coordinates.  We want to show that 
\begin{equation}
\mathbb{P} \left[ \BL\mbox{-particles at }b_1,\dots, b_l \right] = \det \left( \BL_{n,\rho} \left(\xi_i, \eta_i;\xi_j,\eta_j\right)_{1 \leq i, j ,\leq l}\right) \label{Lprob}
\end{equation}

Note that we have an $\BL$-particle at a black vertex of $b$ if and only if a dimer (domino) covers the edges $(b,b+f_1)$ or $(b,b+f_2)$ where $f_1=e_1$ and $f_2=-e_2$.  Hence, by Theorem~\ref{Kenyonsthm} and the linearity of the determinant in its rows, we have that
\begin{equation}
\begin{split}
\mathbb{P} \left[ \BL\mbox{-particles at }b_1,\dots, b_l \right] &=\sum_{r_1, \dots, r_l}^2 \prod_{j=1}^l K_a(b_i,b_i +f_{r_i}) \det \left( K_a^{-1} (b_i+f_{r_i},b_j )\right)_{1 \leq i,j \leq l } \\ 
&= \det \left(   \sum_{r_i=1}^2 K_a(b_i,b_i +f_{r_i})K_a^{-1} (b_i+f_{r_i},b_j )\right)_{1 \leq i,j \leq l } 
\end{split}
\end{equation}
Recall that, from~\eqref{inverseK:defn_of_K}, we have that 
\begin{equation}
K_a(b_i,b_i+f_1)=(-1)^{-(\xi_i+\eta_i+1)/2}
\end{equation}
and 
\begin{equation}
K_a(b_i,b_i+f_2)= (-1)^{-(\xi_i+\eta_i+1)/2} a \mathrm{i}.
\end{equation}
If we use the formula~\eqref{InvKastFormula} for the inverse Kasteleyn matrix, we see that 
\begin{equation}
\begin{split}
&(-1)^{(\eta_2-\xi_2)/4-(\eta_1-\xi_1)/4} \sum_{r=1}^2 K_a(b_1,b_1+f_r) K^{-1}_a(b_1+f_r,b_2) \\
&=\BK_{n,\rho}\left(\eta_2+1, \frac{\eta_2-\xi_2+M}{2};\eta_1+2,\frac{\eta_1-\xi_1+M}{2} \right)\\
&-a \BK_{n,\rho}\left(\eta_2+1, \frac{\eta_2-\xi_2+M}{2};\eta_1,\frac{\eta_1-\xi_1+M-2}{2} \right),
\end{split}
\end{equation}
where $M=2m+1$ and we have used $(-1)^{-\eta_1}=-1$ since $\eta_1$ is odd.  Thus, to prove~\eqref{Lprob} it suffices to show that 
\begin{equation}
\begin{split}
\BL_{n,\rho}(\xi_1,\eta_1;\xi_2,\eta_2)&=\BK_{n,\rho}\left(\eta_2+1, \frac{\eta_2-\xi_2+M}{2};\eta_1+2,\frac{\eta_1-\xi_1+M}{2} \right)\\
&-a \BK_{n,\rho}\left(\eta_2+1, \frac{\eta_2-\xi_2+M}{2};\eta_1,\frac{\eta_1-\xi_1+M-2}{2} \right) 
\label{KLrelation}
\end{split}
\end{equation}

Now, with $\varepsilon_i \in \{0,1\}$ we have from~\eqref{K12} we have
\begin{equation}
\begin{split}
&\mathbb{K}_{n,\rho} (2r+\varepsilon_1,x; 2 s +\varepsilon_2,y)=-\Id_{2s+\varepsilon_2<2r + \varepsilon_1} \psi_{2r+\varepsilon_1,2s+\varepsilon_2}(x,y) \\
&+(-1)^{x-y} S(2r+\varepsilon_1,x;2s+\varepsilon_2,y) \\
&-(-1)^{x-y} \left<(( \mathbb{I}-\mathbb{K})^{-1}_{2m+1} a_{-y,2s+\varepsilon_2})(k), b_{-x,2r+\varepsilon_1}(k)  \right>_{\overline{l}^2(2m+1)}\\
&:= \sum_{i=0}^2 R_{n,m}^{(i)} (2r+\varepsilon_1,x; 2 s+\varepsilon_2,y).
\end{split}
\end{equation}
Define
\begin{equation} \label{Ti}
\begin{split}
T_{n,m}^{(i)} (\xi_1, \eta_1; \xi_2, \eta_2) &=R_{n,m}^{(i)} \left(\eta_2+1,\frac{\eta_2- \xi_2+M}{2} ; \eta_1+2, \frac{\eta_1- \xi_1 +M}{2} \right) \\
&-aR_{n,m}^{(i)} \left(\eta_2+1,\frac{\eta_2- \xi_2+M}{2} ; \eta_1, \frac{\eta_1- \xi_1 +M-2}{2} \right).
\end{split}
\end{equation}
 We have to show that 
\begin{equation} \label{SumT}
\sum_{i=0}^2 T_{n,m}^{(i)} (2r+\varepsilon_1,x; 2 s+\varepsilon_2,y) =\BL_{n,\rho} (\xi_1,\eta_1;\xi_2,\eta_2).
\end{equation} 
Since $\eta_1$ and $\eta_2$ are odd we can write $\eta_1=2s-1$, $\eta_2=2r-1$ so that $\eta_2+1=2r$ and $\eta_1=2(s-1)+1$ and $\eta_1+2=2s+1$.  Now by~\eqref{K12}, 
\begin{equation}
\begin{split}
&T_{n,m}^{(0)} (\xi_1,\eta_1; \xi_2,\eta_2) \\
&=-\frac{\Id_{2s+1<2r} }{2 \pi i} \int_{\Gamma_{0,a} } \frac{(1+a z)^{s-r}}{(1-a/z)^{s-r+1}} z^{(\eta_2-\eta_1)/2 +( \xi_1-\xi_2)/2} \frac{d z}{z} \\
&+\frac{a\Id_{2s-1<2r} }{2 \pi i} \int_{\Gamma_{0,a} } \frac{(1+a z)^{s-1-r}}{(1-a/z)^{s-r}} z^{(\eta_2-\eta_1)/2 +( \xi_1-\xi_2)/2+1} \frac{d z}{z}. \\
\end{split}\label{T0}
\end{equation}
If $\eta_1>\eta_2$, i.e. $s \geq r+1$, then $2s+1>2s-1>2r$ and the expression in the right hand side of~\eqref{T0} is equal to zero.  If $\eta_1=\eta_2$, i.e. $r=s$, we get
\begin{equation}
\begin{split}
&\frac{a}{2\pi i } \int_{\Gamma_{0,a} } \frac{z^{(\xi_1-\xi_2)/2}}{1+a z} dz = \frac{a \Id_{\xi_1<\xi_2}}{2\pi i } \int_{\Gamma_{0,a}} \frac{z^{(\xi_1-\xi_2)/2}}{1+a z} dz \\
&=-\frac{\Id_{\xi_1<\xi_2}}{2\pi i } \int_{\Gamma_{-1/a} } \frac{z^{(\xi_1-\xi_2)/2}}{ z+1/a} dz=-\Id_{\xi_1< \xi_2}\left(- \frac{1}{a} \right) ^{(\xi_1-\xi_2)/2} .
\end{split}
\end{equation}
because of the change of orientation from moving the contour $\Gamma_{0,a}$ to the contour $\Gamma_{-1/a}$.  
Thus, 
\begin{equation}
\begin{split}\label{T02}
T_{n,m}^{(0)} (\xi, \eta_1 ; \eta_2, \xi_2)& =-\Id_{\eta_1<\eta_2} \frac{1+a^2}{2\pi i} \int_{\Gamma_{0,a}} \frac{ (1+a z)^{(\eta_1 -\eta_2)/2-1} }{( z- a)^{(\eta_1-\eta_2)/2+1} } z^{(\xi_1-\xi_2)/2} dz\\
&-\Id_{\xi_1 < \xi_2} \Id_{\eta_1 =\eta_2} (-a)^{(\xi_2-\xi_1)/2} .
\end{split} 
\end{equation}
If $\xi_1 \geq \xi_2$ and $\eta_1 < \eta_2$ the integral in~\eqref{T02} is equal to 0 since the integrand is analytic inside $\Gamma_{0,a}$. Thus, the first term in~\eqref{T02} equals
\begin{equation}
\begin{split} 
&-\Id_{\eta_1<\eta_2} \Id_{\xi_1<\xi_2} \frac{1+a^2}{2\pi i } \int_{\Gamma_{0,a} } \frac{ (1+a z)^{(\eta_1-\eta_2)/2-1}}{(z-a)^{(\eta_1-\eta_2)/2+1}} z^{(\xi_1 - \xi_2)/2}  dz \\ 
&=-(1-\Id_{\eta_1 >\eta_2}  -\Id_{\eta_1=\eta_2} ) \Id_{\xi_1 <\xi_2} 
\frac{1+a^2}{2\pi i } \int_{\Gamma_{0,a} } \frac{ (1+a z)^{(\eta_1-\eta_2)/2-1}}{(z-a)^{(\eta_1-\eta_2)/2+1}} z^{(\xi_1 - \xi_2)/2}  dz. 
\end{split}
\end{equation}
The last integral is equal to zero if $\eta_1>\eta_2$ and hence the last expression equals 
\begin{equation}
\begin{split}
&- \Id_{\eta_1=\eta_2} \Id_{\xi_1<\xi_2} \frac{1+a^2}{2\pi i } \int_{\Gamma_{-1/a} } \frac{z^{(\xi_1-\xi_2)/2}}{(1+a z)(z-a)} dz \\
&- \Id_{\xi_1 < \xi_2} \frac{1+a^2}{2 \pi i } \int_{\Gamma_{0,a} } \frac{( 1+ a z)^{(\eta_1-\eta_2)/2-1} }{ (z-a)^{(\eta_1-\eta_2)/2+1}} z^{(\xi_1-\xi_2)/2} dz.
\end{split}
\end{equation}
The first term in this expression equals
\begin{equation}
\Id_{\eta_1=\eta_2} \Id_{\xi_1<\xi_2} \left(-\frac{1}{a} \right)^{(\xi_1-\xi_2)/2} 
\end{equation}
 so combining this result with~\eqref{T02} we find
\begin{equation}
\begin{split} \label{TO3}
&T_{n,m}^{(0)} (\xi_1,\eta_1; \xi_2, \eta_2) \\
&=-\Id_{\xi_1<\xi_2} \frac{1+a^2}{2 \pi i} \int_{\Gamma_{0,a}} \frac{ (1+a z)^{(\eta_1-\eta_2)/2-1}}{(z-a)^{(\eta_1-\eta_2)/2+1} } z^{(\xi_1-\xi_2)/2}.
\end{split}
\end{equation}
Next, using~\eqref{K12} and inserting it into~\eqref{Ti} we get
\begin{equation}
\begin{split}
&T_{n,m}^{(1)} (\xi_1,\eta_1;\xi_2,\eta_2) \\
&=\frac{1+a^2}{(2\pi i)^2 } \int_{\Gamma_{0,a}} dv \int_{\Gamma_{0,a,v}} \frac{dv}{v-u}\frac{ v^{n-\xi_2/2}}{u^{n- \xi_1/2}} \frac{(1+a u)^{(\eta_1-1)/2} (u- a)^{n-(\eta_1+1)/2} }{ (1+a v)^{(\eta_2+1)/2} (v-a)^{n-(\eta_2-1)/2} } 
\label{T1}
\end{split}
\end{equation}
after a short computation.  Finally, by~\eqref{Ti} and~\eqref{K12} we see that
\begin{equation}
\begin{split} \label{T21} 
&T_{n,m}^{(2)}(\xi_1,\eta_1 ; \xi_2, \eta_2) \\
&= \left< (\mathbbm{I}-K)^{-1}_{2m+1} (a_{-(\eta_1- \xi_1 +M)/2,2s+1} +a a_{-(\eta_1-\xi_1 +M-2)/2 ,2s-1}) (k),\right.\\
& \left. b_{-(\eta_2-\xi_2+M)/2,2r} (k) \right> _{\ge 2m+1} (-1)^{(\eta_2-\xi_2+M)/2-(\eta_1-\xi_1+M-2)/2}.
\end{split}
\end{equation}
Now, using~\eqref{K12} a computation gives
\begin{equation}
\begin{split}
&(-1)^{-(\eta_1 - \xi_1+M-2)/2} \left( a_{-(\eta_1- \xi_1 +M)/2,2s+1}(k) +a a_{-(\eta_1-\xi_1 +M-2)/2 ,2s-1}(k) \right) \\
&= \frac{(-1)^k (1+a^2)}{ (2 \pi i)^2}\int_{\Gamma_{0,a}} du \int_{\Gamma_{0,a,u} } \frac{dv}{ u-v}\frac{v^{(\xi_1 - \eta_1-1)/2}}{ u^{k+1}} \frac{ (1+a v)^{(\eta_1-1)/2} (1-a /v)^{n-(\eta_1+1)/2} }{ (1+a v)^n (v-a)^{n+1} } \\
&= -(1+a^2) A_{\xi_1, \eta_1}(k).
\label{A11}
\end{split}
\end{equation}
Note that by moving the $v$-contour in~\eqref{A11} inside the $u$-contour we get the expression in~\eqref{oneAzt}.  Similarly, by~\eqref{K12},
\begin{equation}
\begin{split}
&b_{-(\eta_2-\xi_2+M)/2,2r}(k) (-1)^{(\eta_2-\xi_2+M)/2}\\
&=\frac{(-1)^k}{(2\pi i )^2} \int_{\Gamma_{0,a}} du \int_{\Gamma_{0,a,u}} \frac{dv}{v-u} \frac{v^k}{u^{(\xi_2-\eta_2+1)/2}} \frac{(1+a v)^n(w-a)^{n+1} }{ (1+a v)^{(\eta_2+1)/2} (1-a /v)^{n-(\eta_2-1)/2} } \\
&=B_{\xi_2,\eta_2}(k)
\end{split}
\end{equation}
Again by moving the $v$-contour inside the $u$-contour we get the expression in~\eqref{oneAzt}. Thus,
\begin{equation}
T_{n,m}^{(2)}(\xi_1,\eta_1 ; \xi_2, \eta_2)=-\left< ((\mathbbm{I}-K)^{-1}_{2m+1} A_{\xi_1,\eta_1})(k) , B_{\xi_2,\eta_2}(k) \right>_{\ge 2m+1} (1+a^2).\label{T2}
\end{equation}
If we use~\eqref{TO3},~\eqref{T1} and~\eqref{T2} we obtain~\eqref{SumT} which is what we wanted to prove.

\end{document}